\documentclass{amsart}

\makeatletter
\newcommand{\LABEL}[1]{\label{#1}%\show\@currentlabel
\edef\@CURRENT{\@currentlabel}\def\@NAME{#1}}

\newcommand{\myindex}[1]{
\@ifundefined{@CURRENT}{\index{#1}}%
{\index{#1 (\@CURRENT)}}}

\makeatother

\usepackage[utf8]{inputenc}
\usepackage[hidelinks]{hyperref}
\usepackage{color}
\usepackage{amssymb}
\usepackage{mathrsfs} % For \Ss
\usepackage{mathtools}

\input xy
\xyoption{all}

\theoremstyle{definition}
\newtheorem{mydef}{Definition}[section]
\newtheorem{lem}[mydef]{Lemma}
\newtheorem{thm}[mydef]{Theorem}
\newtheorem{conjecture}[mydef]{Conjecture}
\newtheorem{cor}[mydef]{Corollary}

\newtheorem{hypothesis}[mydef]{Hypothesis}

\newtheorem{defin}[mydef]{Definition}
\newtheorem{example}[mydef]{Example}
\newtheorem{remark}[mydef]{Remark}
\newtheorem{notation}[mydef]{Notation}
\newtheorem{fact}[mydef]{Fact}
\newtheorem{discussion}[mydef]{Discussion}

\newcommand{\fct}[2]{{}^{#1}#2}

\newcommand{\bM}{\bar{M}}
\newcommand{\ba}{\bar{a}}
\newcommand{\bb}{\bar{b}}

\newcommand{\bx}{\bar{x}}
\newcommand{\by}{\bar{y}}

\newcommand{\Ksatpp}[2]{{#1}^{#2\text{-sat}}}
\newcommand{\Ksatp}[1]{\Ksatpp{\K}{#1}}
\newcommand{\Ksat}{\K^{\text{sat}}}
\newcommand{\Kbrim}{\K^{\text{brim}}}
\newcommand{\Kprop}{\K^{\text{proper}}}
\newcommand{\Kpropa}{\K^{\text{proper}, \ast}}

\newcommand{\blueq}[1]{{\color{blue} #1}} 
\newcommand{\redq}[1]{{\color{red} #1}}

\newcommand{\inc}{\redq{incomp}}

\newcommand{\sea}{\mathfrak{C}}
\newcommand{\dom}{\operatorname{dom}}

\newcommand{\cf}[1]{\text{cf} (#1)}
\newcommand{\seq}[1]{\langle #1 \rangle}
\newcommand{\rest}{\upharpoonright}

\newcommand{\s}{\mathfrak{s}}

\newcommand{\ts}{\mathfrak{t}}
\newcommand{\is}{\mathfrak{i}}
\newcommand{\isbr}{\is^{\text{brim}}}
\newcommand{\isbrim}{\isbr}
\newcommand{\js}{\mathfrak{j}}

\newcommand{\id}{\text{id}}

\newcommand{\leap}[1]{\le_{#1}}
\newcommand{\ltap}[1]{<_{#1}}

\newcommand{\geap}[1]{\ge_{#1}}

\newcommand{\lta}{\ltap{\K}}
\newcommand{\lea}{\leap{\K}}

\newcommand{\gea}{\geap{\K}}

\def\lee{\preceq}

\newcommand{\K}{\mathfrak{K}}

\newbox\noforkbox \newdimen\forklinewidth
\forklinewidth=0.3pt \setbox0\hbox{$\textstyle\smile$}
\setbox1\hbox to \wd0{\hfil\vrule width \forklinewidth depth-2pt
 height 10pt \hfil}
\wd1=0 cm \setbox\noforkbox\hbox{\lower 2pt\box1\lower
2pt\box0\relax}
\def\unionstick{\mathop{\copy\noforkbox}\limits}

\setbox0\hbox{$\textstyle\smile$}
\setbox1\hbox to \wd0{\hfil{\sl /\/}\hfil} \setbox2\hbox to
\wd0{\hfil\vrule height 10pt depth -2pt width
               \forklinewidth\hfil}
\newbox\doesforkbox
\setbox\doesforkbox\hbox{\lower 0pt\box1 \lower
2pt\box2\lower2pt\box0\relax}

\newcommand{\nf}{\unionstick}

\newcommand{\nfs}[4]{#2 \nf_{#1}^{#4} #3}

\def\1nf{\unionstick^{(1)}}

\def\2nf{\unionstick^{(2)}}
\def\3nf{\unionstick^{(3)}}

\def\nfm{\overline{\nf}}
\newcommand{\nfcl}[4]{#2 \overset{#4}{\underset{#1}{\overline{\nf}}} #3}

\newcommand{\tp}{\otp}
\newcommand{\gtp}{\otp}

\newcommand{\gS}{\oS}
\newcommand{\oSp}[1]{\mathscr{S}_{#1}}
\newcommand{\oS}{\oSp{}}

\newcommand{\Ii}{\mathbb{I}}

\newcommand{\hanf}[1]{h (#1)}

\newcommand{\hanfe}[1]{\beth_{\left(2^{#1}\right)^+}}
\newcommand{\ehanf}[1]{\hanfe{#1}}

\newcommand{\Ll}{\mathbb{L}}
\newcommand{\otp}{\mathbf{tp}}

\newcommand{\goodp}{\text{good}^+}

\newcommand{\PC}{\operatorname{PC}}

\newcommand{\NF}{\operatorname{NF}}

\newcommand{\LS}{\text{LS}}

\newcommand{\Iis}{\mathcal{I}}
\newcommand{\Jis}{\mathcal{J}}

\newcommand{\m}{\mathbf{m}}

\newcommand{\Ps}{\mathcal{P}}
\newcommand{\Psm}{\Ps^-}

\newcommand{\WGCH}{\operatorname{WGCH}}

\newcommand{\muunif}{\mu_{\text{unif}}}

\newcommand{\barsign}{|}

\title{Categoricity and multidimensional diagrams}
\date{\today}

\makeindex

\author {Saharon Shelah}
\email{shelah@math.huji.ac.il}
\urladdr{http://shelah.logic.at}
\address{Einstein Institute of Mathematics\\
Edmond J. Safra Campus, Givat Ram\\
The Hebrew University of Jerusalem\\
Jerusalem, 91904, Israel, and Department of Mathematics\\
 Hill Center - Busch Campus \\ 
 Rutgers, The State University of New Jersey \\
 110 Frelinghuysen Road \\
 Piscataway, NJ 08854-8019, USA}
 
\thanks{This is paper 842 in the first author archive.  The first author would like to thank the Israel Science Foundation for partial support of this research (Grant No. 242/03) and the European Research Council grant 338821, and the National Science Foundation grant no: 136974  and NSF-BSF 2021: grant with M. Malliaris, NSF 2051825, BSF 3013005232 (2021-10-2026-09). The first  author is grateful to an individual who prefers to remain anonymous for providing typing services and  thanks the typist for the careful and beautiful typing that were used during late  work on the paper.}

\author{Sebastien Vasey}

\email{sebv@math.harvard.edu}

\urladdr{http://math.harvard.edu/\textasciitilde sebv/}

\address{Department of Mathematics \\ Harvard University \\ Cambridge, Massachusetts, USA}

\usepackage[hidelinks]{hyperref}
\renewcommand{\s}{\mathfrak{s}}
\renewcommand{\K}{\mathfrak{K}}

\renewcommand{\m}{\mathbf{m}}

\begin{document}
\makeatletter\def\shfiuwefootnote{\gdef\@thefnmark{}\@footnotetext}\makeatother\shfiuwefootnote{Version 2023-02-22. See \url{https://shelah.logic.at/papers/842/} for possible updates.}

\dedicatory{Dedicated to the memory of Oren Kolman.}

 % \thanks{In later versions the first author would like to thank the typist for his work and is also grateful for the generous funding of typing services donated by a person who wishes to remain anonymous.}

\begin{abstract}
  We study multidimensional diagrams in independent amalgamation in the framework of abstract elementary classes (AECs). We use them to prove the eventual categoricity conjecture for AECs, assuming a large cardinal axiom. More precisely, we show assuming the existence of a proper class of strongly compact cardinals that an AEC which has a single model of \emph{some} high-enough cardinality will have a single model in \emph{any} high-enough cardinal. Assuming a weak version of the generalized continuum hypothesis, we also establish the eventual categoricity conjecture for AECs with amalgamation. 
\end{abstract}

\maketitle

\newpage
\tableofcontents

%%%%%%%%%%%%%%%%%%%%%%%%%%%%%%%%%%%%%%%%%%%%%%
\newpage

\section{Introduction}

\subsection{General background and motivation}

One of the most important result of modern model theory is Morley's categoricity theorem \cite{morley-cip}: a countable first-order theory categorical in \emph{some} uncountable cardinal is categorical in \emph{all} uncountable cardinals. Here, we call a theory (or more generally a class of structures) \emph{categorical} in a cardinal $\lambda$ if it has (up to isomorphism) exactly one model of cardinality $\lambda$. Morley's theorem generalizes the fact that algebraically closed fields of characteristic zero are categorical in all uncountable cardinals.

The machinery developed in the proof of Morley's theorem has proven to be important, both for pure model theory and for applications to other fields of mathematics (on the former, see e.g.\ the first author's book \cite{shelahfobook}; on the latter, see e.g.\ \cite{mordell-lang}). Specifically, the proof of Morley's theorem shows that there is a notion of independence generalizing algebraic independence (as well as linear independence) for the models of the theory in question. This was later precisely generalized and studied by the first author, who called this notion \emph{forking}. Forking has now become a central concept of model theory.

It is natural to look for extensions of Morley's theorem, the real goal being to look for generalizations of \emph{forking} to other setups. The first author has established \cite{sh31} that for \emph{any} (not necessarily countable) first-order theory $T$, if $T$ is categorical in \emph{some} $\mu > |T|$, then $T$ is categorical in \emph{all} $\mu' > |T|$. One possible next step would be to prove versions of Morley's theorem for infinitary logics, like $\Ll_{\omega_1, \omega}$ (countably-many conjunctions are allowed within a single formula). Indeed, already in the late seventies the first-author conjectured (see \cite[Conjecture 2]{sh87a}):

\begin{conjecture}[Categoricity conjecture for $\Ll_{\omega_1, \omega}$, strong version]\LABEL{strong-categ}
  Let $\psi$ be an $\Ll_{\omega_1, \omega}$-sentence. If $\psi$ is categorical in \emph{some} $\mu \ge \beth_{\omega_1}$, then $\psi$ is categorical in \emph{all} $\mu' \ge \beth_{\omega_1}$.
\end{conjecture}

Forty years down the road, and despite a lot of partial approximations (see the history given in the introduction of \cite{ap-universal-apal}), Conjecture \ref{strong-categ} is still open. The main difficulty is that the compactness theorem fails, hence one has to work much more locally. The interest is that one can derive weak version of compactness from the categoricity assumption itself, often combined with large cardinals or combinatorial set-theoretic assumptions. Thus any proof is likely to exhibit a rich interplay between set theory and model theory.

The threshold $\beth_{\omega_1}$ appears in Conjecture \ref{strong-categ} because it is provably best possible (for lower thresholds, there is a standard counterexample due to Morley, and written down explicitly with full details in \cite[4.1]{mv-universal-v4}). Nevertheless, the spirit of the conjecture is captured by the following eventual version:

\begin{conjecture}[Categoricity conjecture for $\Ll_{\omega_1, \omega}$, eventual version]\LABEL{eventual-categ}
  There is a threshold cardinal $\theta$ such that if $\psi$ is an $\Ll_{\omega_1, \omega}$-sentence categorical in \emph{some} $\mu \ge \theta$, then $\psi$ is categorical in \emph{all} $\mu' \ge \theta$.
\end{conjecture}

\subsection{Eventual categoricity from large cardinals}

Conjecture \ref{eventual-categ} is also still open. In the present paper, we prove it assuming that a large cardinal axiom holds. More precisely and generally:

\textbf{Corollary \ref{compact-main-cor-2}.} Let $\kappa$ be a strongly compact cardinal and let $\psi$ be an $\Ll_{\kappa, \omega}$-sentence. If $\psi$ is categorical in \emph{some} $\mu \ge \ehanf{\kappa}$, then $\psi$ is categorical in \emph{all} $\mu' \ge \ehanf{\kappa}$.

This solves Question 6.14(1) in the first author's list of open problems \cite{sh702} and improves on a result of Makkai and the first author \cite{makkaishelah} who proved the statement of Corollary \ref{compact-main-cor-2} under the additional assumption that the initial categoricity cardinal $\mu$ is a \emph{successor} cardinal. We see this as quite a strong assumption, since it allows one to work directly with models of a single cardinality $\mu_0$ (where $\mu = \mu_0^+$). The gap between \cite{makkaishelah} and Corollary \ref{compact-main-cor-2} is further reflected by the several additional levels of sophistication that we use to prove the latter. For example, we work in the framework of abstract elementary classes (AECs). This was introduced by the first author \cite{sh88} as a general semantic framework encompassing in particular classes axiomatized by ``reasonable'' logics. The following version of Conjecture \ref{eventual-categ} has been stated there \cite[N.4.2]{shelahaecbook}:

\begin{conjecture}[Eventual categoricity conjecture for AECs]\LABEL{categ-aec}
  An AEC categorical in \emph{some} high-enough cardinal is categorical in \emph{all} high-enough cardinals.
\end{conjecture}

In this paper, we more generally prove Conjecture \ref{categ-aec} assuming a large cardinal axiom (a proper class of strongly compact cardinals). Again, Conjecture \ref{categ-aec} was known in the special case where the starting categoricity cardinal is a successor and there is a proper class of strongly compact cardinals (a result of Boney \cite{tamelc-jsl}). The second author had also proven it for AECs closed under intersections \cite[1.7]{ap-universal-apal}.

\subsection{Multidimensional diagrams}

The proof of Corollary \ref{compact-main-cor-2} only tangentially uses large cardinals. We have already mentioned that one key step in the proof is the move to AECs: this is done because of generality, but also because the framework has a lot of closure, allowing us for example to take subclasses of saturated models and still remain inside an AEC. The other key step is the use of multidimensional diagrams in independent amalgamation. Essentially, the idea is to generalize \cite{sh87a, sh87b} (which discussed $\Ll_{\omega_1, \omega}$ only) to AECs. This was started in the first author's two volume book \cite{shelahaecbook, shelahaecbook2}, where the central concept of a \emph{good $\lambda$-frame} was introduced. Roughly, an AEC $\K$ has a good $\lambda$-frame if it has several basic structural properties in $\lambda$ (amalgamation, no maximal models, and stability), and also has a superstable-like notion of forking defined for models of cardinality $\lambda$. The idea here is that everything is assumed only for models of cardinality $\lambda$. The program, roughly, is to develop the theory of good frames so that one can start with a good $\lambda$-frame, extend it to a good $\lambda^+$-frame, and keep going until the structure of the whole class is understood.

In order to carry out this program in practice, one has to prove that the good $\lambda$-frame has more structural properties. Specifically, one first proves (under some conditions) that a notion of \emph{nonforking amalgamation} can be defined (initially, the definition of a good frame only requires a notion of nonforking for types of singles elements, but here we get one for types of models). Thus one has at least a \emph{two-dimensional} notion of nonforking (it applies to squares of models). From this, one seeks to define a three-dimensional notion (for cubes of models), and then inductively an $n$-dimensional notion. It turns out that once one has a good $n$-dimensional notion for all $n < \omega$ (still for models of cardinality), then one can lift this to the whole class above $\lambda.$ In \cite{sh87a, sh87b}, an $\mathbb{L}_{\omega_{1}, \omega}$ sentence with all the good $n$-dimensional properties was called \emph{excellent}, so we naturally call an AEC with those properties an \emph{excellent} AEC. It turns out that this is truly a generalization of  \cite{sh87a, sh87b}, and in fact another consequence of the methods of this paper is Theorem \ref{few-model-thm-2},  which generalizes \cite{sh87a, sh87b} to any $\PC_{\aleph_{0}}$-AEC. A particular case is that of $\mathbb{L}_{\omega_{1}, \omega}(Q)$ classes, see on this \cite[II.3.5]{shelahaecbook}. In particular see \cite[3.5]{Sh:600} and see there reading plan 3 in the end of \S0 how this give results. In particular \cite{sh87a, sh87b} does indeed generalize to $\mathbb{L}_{\omega_{1}, \omega}(Q),$ in particular if $\psi \in \mathbb{L}_{\omega_{1}, \omega}(Q)$ is categorical in $\aleph_{n +1}$ and $2^{\aleph_{n}} < 2^{\aleph_{n+1}}$ for $n < \omega$ \underline{then} $\psi$ is categorical in every $\lambda > \aleph_{0}.$ It had been a longstanding open question whether such a generalization was possible. 

The use of multidimensional diagrams has had several more applications. For example, the \emph{main gap theorem} of the first author \cite{main-gap-ams} is a theorem about countable first-order theory but was proven using multidimensional diagrams. Thus we believe that the systematic study of multidimensional diagrams undertaken in this paper is really its key contribution (in fact, a main gap theorem for certain AECs now also seems within reach - the axioms of \cite{abstract-decomp} hold inside $\Ksatp{\LS (\K)^+}$ if $\K$ is an excellent class, see Section \ref{excellent-sec} - but we leave further explorations to future work). Let us note that multidimensional diagrams for AECs were already studied in \cite[III.12]{shelahaecbook}, but we do not rely on this here. 

\subsection{Eventual categoricity in AECs with amalgamation}

We mention another application of multidimensional diagram, which is really why this paper was started. In \cite[Theorem IV.7.12]{shelahaecbook}, it was asserted that the eventual categoricity conjecture holds for AECs with amalgamation, assuming a weak version of the generalized continuum hypothesis (GCH). However, the proof relied on a claim of \cite[III.12]{shelahaecbook} that was not proven there, but promised to appear in \cite{sh842}. See also \cite[Section 11]{Vas17} for an exposition of the proof, modulo the claim.

We prove this missing claim\footnote{A very careful expert may note that the statement of Corollary \ref{categ-extendible-cor} is not exactly the same as that of \cite[Claim~11.2]{Vas17}. Terminology aside, the two statements are very close and the result of the present paper is (arguably) stronger and more explicit. None of this rely matters since we do not rely on any way on how the missing claim was previously formulate.} in the present paper (Corollary \ref{categ-extendible-cor}). As a consequence, we obtain the consistency of the eventual categoricity conjecture in AECs with amalgamation:

\textbf{Corollary \ref{ap-main-cor}.} Assume $2^\theta < 2^{\theta^+}$ for all cardinals $\theta$. Let $\K$ be an AEC with amalgamation. If $\K$ is categorical in \emph{some} $\mu \ge \ehanf{\aleph_{\LS(\K)^{+}}}$, then $\K$ is categorical in \emph{all} $\mu' \ge \ehanf{\aleph_{\LS(\K)^{+}}}$.

Note that a similar-looking result starting with categoricity in a \emph{successor} cardinal was established (in ZFC) in \cite{sh394}. There however, only a downward transfer was obtained (i.e.\ nothing was said about categoricity \emph{above} the starting categoricity cardinal $\mu$). Moreover the threshold cardinal was bigger. After the present paper was written, the second author \cite{categ-amalg-v1} established that the threshold cardinal in Corollary \ref{ap-main-cor} can be lowered further\footnote{The reasons for the appearance of the strange cardinal $\ehanf{\aleph_{\LS (\K)^+}}$ are perhaps best explained by the proof itself. The short version is that $\aleph_{\LS (\K)^+}$ is the minimal limit cardinal with cofinality greater than $\LS (\K)$.}, all the way to $\ehanf{\LS (\K)}$.

Note also that it is known (see for example \cite[1.13]{makkaishelah} or \cite[7.3]{tamelc-jsl}) that AECs categorical above a strongly compact cardinal have (eventual) amalgamation; an early work is \cite{Sh:932}. Thus Corollary \ref{ap-main-cor} can be seen as a generalization of Corollary \ref{compact-main-cor-2}. However, the two results are formally incompatible since there are no cardinal arithmetic assumptions in the hypotheses of Corollary \ref{compact-main-cor-2}. By a result of Kolman and the first author \cite{kosh362}, some amount of amalgamation already follows from categoricity above a \emph{measurable} cardinal. In fact, this amount of amalgamation is enough to carry out the proof of Corollary \ref{ap-main-cor}, see Theorem \ref{ap-main}. Thus the present paper also establishes that Conjecture \ref{categ-aec} is consistent with a proper class of measurable cardinals.\footnote{In fact, the machinery of this paper was designed to use a measurable rather than a strongly compact in Corollary \ref{compact-main-cor-2}. However, some technical points remain on how exactly to do this, so for simplicity we focus on the strongly compact case and leave the measurable case to future works.} More precisely: 

\textbf{Corollary \ref{a78}.} Assume $2^\theta < 2^{\theta^+}$ for all cardinals $\theta$. Let $\K$ be an AEC and let $\kappa > \LS(\K)$ be a measurable cardinal.  If $\K$ is categorical in \emph{some} $\mu \ge \ehanf{\aleph_{\kappa^{+}}}$, then $\K$ is categorical in \emph{all} $\mu' \ge \ehanf{\aleph_{\kappa^+}}$.

See more in \cite{Sh:E102} and \cite{Sh:F1980} and see Remark \ref{14.38}. 

\subsection{Other approaches to eventual categoricity} 

The eventual categoricity conjecture is known to hold under several additional assumptions (in addition to those already mentioned). Grossberg and VanDieren \cite{tamenesstwo, tamenessthree} established an \emph{upward} categoricity transfer from categoricity in a successor in AECs (with amalgamation) that they called \emph{tame}.  Roughly, orbital types in such AECs are determined by their restrictions to small sets. This is a very desirable property which is known to follow from large cardinals \cite{tamelc-jsl}.

Recently, the second author \cite{ap-universal-apal, categ-universal-2-selecta} established the eventual categoricity conjecture for \emph{universal classes}: classes of models axiomatized by a universal $\Ll_{\infty, \omega}$ sentence (or alternatively, classes of models closed under isomorphism, substructure, and unions of chains). The proof is in ZFC and the lower bound on the starting categoricity cardinal is low (not a large cardinal). Further, the starting categoricity cardinal itself does not need to be a successor. To prove this result, the second author established eventual categoricity more generally in AECs that have amalgamation are tame, and in addition have \emph{primes} over models of the form $Ma$ \cite{categ-primes-mlq}.

Now, it is folklore that any excellent AEC (for the purpose of this discussion, an AEC satisfying strong multidimensional amalgamation properties) should have amalgamation, be tame, and have primes. Thus the main contribution of the present paper is to establish excellence for the setups mentioned above (large cardinals or amalgamation with weak GCH). Then one can essentially just quote the categoricity transfer of the second author.

Note however that even with additional locality assumptions, the known categoricity transfers for AECs are proven using local approaches (mainly good frames or a close variant\footnote{See \cite{Vas17} on how the Grossberg-VanDieren result can be proven using good frames.}), and transferring the local structure upward using the tameness assumptions. The present paper takes the local approach further by using it to \emph{prove} tameness conditions, through excellence. Of course, in the present paper we pay a price: cardinal arithmetic or large cardinal assumptions. It remains open whether paying such a price is necessary, or whether excellence (and therefore eventual categoricity) can be proven in ZFC.

\subsection{How to prove excellence}

While the present paper is admittedly complex, relying on several technical frameworks to bootstrap itself, we can readily describe the main idea of the proof of excellence. A starting point is a result of the first author \cite[I.3.8]{shelahaecbook}: assuming $2^{\lambda} < 2^{\lambda^+}$ (the weak diamond, see \cite{dvsh65}), an AEC categorical in $\lambda$ and $\lambda^+$ must have amalgamation in $\lambda$. This can be rephrased as follows: the $1$-dimensional amalgamation property in $\lambda$ holds provided that a strong version of the $0$-dimensional amalgamation property holds in $\lambda^+$ (we think of the $0$-dimensional amalgamation property as joint embedding).

A key insight (a variant of which already appears in \cite[III.12.30]{shelahaecbook}) is that this argument can be repeated by looking at \emph{diagrams} (or \emph{systems}) of models rather than single models (we do not quite get an AEC, but enough of the properties hold). We obtain, roughly, that $(n + 1)$-dimensional amalgamation in $\lambda$ follows from a strong version of $n$-dimensional amalgamation in $\lambda^+$. Since we are talking about \emph{nonforking} amalgamation, we call these properties ``uniqueness'' properties rather than ``amalgamation'' properties, but the idea remains. Parametrizing (and forgetting about the ``strong'' part), we have that $(\lambda^+, n)$-uniqueness implies $(\lambda, n + 1)$-uniqueness. The precise statement is in \ref{uq-lambda-step}. Note that the converse (the $(\lambda, n+1)$)-properties implying the $(\lambda^{+},n)$-properties is well known and standard). 

Now, it was already known how to obtain $(\lambda, 1)$, and even $(\lambda, 2)$-uniqueness. In fact, it is known (under reasonable assumptions) how to obtain a $\lambda$ such that $(\lambda^{+n}, 2)$-uniqueness holds for all $n < \omega$ (such a setup is called an \emph{$\omega$-successful frame}). Assuming $2^{\lambda^{+m}} < 2^{\lambda^{+(m + 1)}}$ for all $m < \omega$, we obtain from the earlier discussion used with each $\lambda^{+m}$ that $(\lambda^{+m}, 3)$-uniqueness holds for all $m < \omega$. Continuing like this, we get $(\lambda^{+m}, n)$-uniqueness for all $n, m < \omega$.

The implementation of this sketch is complicated by the fact that there are other properties to consider, including extension properties. Further, we have to ensure a \emph{strong} uniqueness property (akin to categoricity) in the successor cardinal, so the induction is much more complicated.

Forgetting this for the moment, let us explain how, assuming large cardinals, one can remove the cardinal arithmetic assumption. For this we assume more generally that we have an increasing sequence $\seq{\lambda_m : n < \omega}$ of cardinals such that for all $m < \omega$, $2^{<\lambda_{m + 1}} = 2^{\lambda_m} < \lambda_{m + 1}$ (in other words, $\lambda_{m + 1}$ is least such that $2^{\lambda_m} < 2^{\lambda_{m + 1}}$). Let $\lambda_\omega \coloneqq  \sup_{m < \omega} \lambda_m$. Assume that $(\lambda, 2)$-uniqueness holds for all $\lambda \in [\lambda_0, \lambda_\omega)$. In the previous setup, we had that $\lambda_{m + 1} = \lambda_m^+$. Now, as in \cite[1.2.4]{shvi635}, we can generalize the first author's aforementioned result on getting amalgamation from two successive categoricity to: whenever $\lambda < \lambda'$ and $2^\lambda = 2^{<\lambda'} < 2^{\lambda'}$ and categoricity in $\lambda'$ implies there exists an amalgamation base in $[\lambda, \lambda')$. Essentially, in some $\mu \in [\lambda, \lambda')$, we will have amalgamation. The $n$-dimensional version of this will essentially be that if $(\lambda', n)$-uniqueness holds, then $(\mu, n + 1)$-uniqueness holds for \emph{some} $\mu \in [\lambda, \lambda')$.

So far, we have not used large cardinals. However, we use them to ``fill in the gaps'' and make sure that we have amalgamation in every cardinals below $\mu$. So let $\mu_0 < \mu$. Fix a triple $(M_0, M_1, M_2)$ of size $\mu_0$ to amalgamate. Since we have large cardinals, we can take ultraproducts by a sufficiently complete ultrafilters to obtain an extension $(M_0', M_1', M_2')$ of this triple that has size $\mu$. Now using amalgamation in $\mu$ we can amalgamate this new triple, and hence get an amalgam of the original triple as well. The suitable $n$-dimensional version of this argument can be carried out, and we similarly obtain $(\lambda, n)$-uniqueness for all $\lambda \in [\lambda_0, \lambda_\omega)$ and $n < \omega$.

\subsection{Structure of the paper}

We assume that the reader has a solid knowledge of AECs, including \cite{baldwinbook09} (and preferably at least parts of the first three chapters of \cite{shelahaecbook}) but also the more recent literature. We have tried to repeat the relevant background but still, to keep the paper to a manageable size we had to heavily quote and use existing arguments.

Schematically, for each $n$ section $n$ discusses a certain framework, call it framework $n$. Framework $n$ is more powerful than framework $n - 1$. Section $n$ discusses how to get into framework $n$, starting with some instances of framework $n - 1$. Then enough properties of framework $n$ are also proven so that one can develop the properties of framework $n + 1$ in the next section. We have tried to minimize the cross-section dependencies: the reader can often forget about previous framework and focus on the current one. At the end of the paper, an index will also help the reader getting back to previous definitions.

The abstraction seems unfortunately needed to make the intricate arguments work in the end (i.e. we need to prove very general statements so that we can later move from our AEC to e.g. an AEC of independent systems of saturated models, maybe with a fancy ordering). It can be easy to get lost in the maze of abstract definitions and lose track of the big picture. The usual advices apply in this case, go back to the definition. Usually the abstract parameters we have are introduced with a specialization in mind.  Always keep these in mind. 

For the reader convenience, we add a list of frameworks that we use, each one with an example: 

\begin{enumerate}
    \item[(a)] Abstract class (see Definition \ref{ac-def}, page  \pageref{ac-def}). \underline{Example}: an elementary class restricted to the rigid models in it. 
    
    \item[(b)] Skeleton (see Definition \ref{skel-def}, page \pageref{skel-def}). The main example of a skeleton is that of limit/saturated models. $M \leq_{\K_{\inc}} N$ iff $N$ is limit over $M.$ 
    
    \item[(c)] $\phi \in \mathbb{L}_{\omega_{1}, \omega}(Q)$ (see Example \ref{resolv-def}, page \pageref{resolv-def}).
    
    \item[(d)] Fragmented AEC, (see Definition \ref{r5}(4), page \pageref{r5}). \underline{Example}: the class of saturated models of a (first order) superstable theory. 
    
    \item[(e)] Semi AEC (see Definition \ref{r5}(7), page \pageref{r5}). \underline{Note}: increasing chain with union of bigger cardinality may be omitted.
    
    \item[(f)] Resolvable (see Definition \ref{resolv-def}, page \pageref{resolv-def}). \underline{Note}: it says we can represented $M \in \mathfrak{K}_{\lambda}$ by a chain of smaller cardinality model. 
    
    \item[(g)] $\kappa$-compact $\mathfrak{K}$ (see Definition \ref{compact-def}, page \pageref{compact-def}). \underline{Note}: if $\kappa$ is a compact cardinal, $\mathfrak{K}$ is closed under ultra powers and more, the situation is more similar to f.u. 
    
    \item[(h)] (Semi)-good frame (see Definition \ref{good-frame-def}, page \pageref{good-frame-def}). \underline{Note}: this is an AEC parallel to superstable, not just for one cardinality. 
    
    \item[(i)] A two-dimensional independence relation $\nf$ (see Definition \ref{relation-def}, page \pageref{relation-def}) may be good/very good (see Definition \ref{good-def}/Definition \ref{very-good-def}). \underline{Example}: the class of models of a stable first order theory. 
    
    \item[(j)] A two-dimensional independence notion on $\mathfrak{K}$ is good (see Definition \ref{good-def}, page \pageref{good-def}). \underline{Note} here we have non-forking of models, $\nfs{M_{0}}{M_{1}}{M_{2}}{M_{3}}.$ 
    
    \item[(k)]  A multidimensional independence relation (see Definition \ref{multi-def}), page \pageref{multi-def}. \underline{Note} here we have the parallel of $\mathcal{P}$-stable diagram for $\mathcal{P} \subseteq \mathcal{P}(n).$ It can be very good (see Definition \ref{very-good-twodim-def}, page \pageref{very-good-def}) and even excellent (see Definition \ref{excellent-def}, page \pageref{excellent-def}).  
\end{enumerate}

% \underline{Debts}: 

% \begin{itemize}
%     \item The lecture in \inc. 
    
%     \item Reference to \cite{Sh:932}. 
    
%     \item Reference to \cite{Sh:705} in 6.7. 
    
%     \item Reference to \cite{Sh:734} in 3.21-3.4 page 14 \inc - \S14.
    
%     \item \blueq{Ortp}.
    
%     \item l.c. $\K.$ 
% \end{itemize}

A key is that to study independent systems, it is easier to study the limit ones (i.e. the ones that are ``very saturated''). This is because they are canonical. Honestly, and especially with claims having to do  with a multidimensional systems, sometimes it is better if the reader tries to prove the claim before reading the proof. We feel that the main ideas should be straightforward-enough, even if the technical realization is involved. A dream of the second author is that introducing category-theoretic language can simplify the proofs.   

The paper is organized as follows. After some notational preliminaries, AECs and more generally abstract classes are discussed (Section \ref{ac-sec}). Then a special framework that will be used to deal with large cardinals, compact AECs, is introduced (Section \ref{compact-sec}). We do some combinatorics, in particular the weak diamond argument alluded to earlier, in Section \ref{combinatorics-sec}. The next framework, good frames, is discussed in Section \ref{good-frames-sec}. We then move on to two-dimensional independence notions (Section \ref{twodim-sec}), and multidimensional independence notions (Section \ref{multidim-sec}). Some properties of multidimensional independence notions are then proven that are purely combinatorial, in the sense of only requiring finitely many steps (no arguments involving unions of chains and resolutions), see Sections \ref{multi-finite-sec}, \ref{multi-ap-sec}. One then moves on to looking at multidimensional independence notions with also requirements on chains of systems (Section \ref{multi-cont-sec}). After a section on primes (Section \ref{primes-sec}), we discuss excellent AECs (Section \ref{excellent-sec}) and conclude with the main theorems of this paper (Section \ref{main-sec}).

Each section begins with a short overview of its organization and main results.

\subsection{Acknowledgments}

We thank John Baldwin, Will Boney, Adi Jarden, and the referees for comments that helped improve the paper. The first author lecture on a very preliminary version in the conference in Chicago in October 2007.

%%%%%%%%%%%%%%%%%%%%%%%%%%%%%%%%%%%%%%%%%%%%%%
\newpage

\section{Preliminaries and semilattices}

Our notation is standard, and essentially follows \cite{shelahaecbook} and \cite{aec-stable-aleph0-v5-toappear}. While we may repeat some of these conventions elsewhere in the paper but it does not hurt to emphasize them here: for a structure $M,$ we write $\vert M \vert$ for its universe and $\Vert M \Vert$ for its cardinality.  We imitate this convention for $\mathfrak{K}$ (see Definition \ref{ac-def} and Notation \ref{ac-notation}): $\K$\myindex{$\K$} will denote a pair $(K, \lea)$\myindex{$\lea$}, where $K$ is a class of structures and $\lea$ is a partial order on $K$. We will sometimes write $|\K|$\myindex{$\barsign \K \barsign$} for $K$, but may identify $\K$ and $K$ when there is no danger of confusion. The following set-theoretic function is sometimes useful:

\begin{defin}\LABEL{hanf-def}
    $h(\lambda) \coloneqq  \beth_{(2^{\lambda})^{+}}.$
\end{defin}

% For a structure $M$, we write $|M|$\myindex{$\barsign M \barsign$} for its universe and $\|M\|$\myindex{$\norm M \norm$} for its cardinality. We imitate this convention for $\mathscr{K}$ (see Definition \ref{ac-def} and Notation \ref{ac-notation}): $\K$\myindex{$\K$} will denote a pair $(K, \lea)$\myindex{$\lea$}, where $K$ is a class of structures and $\lea$ is a partial order on $K$. We will sometimes write $|\K|$\myindex{$\barsign \K \barsign$} for $K$, but may identify $\K$ and $K$ when there is no danger of confusion.

We write $[A]^\mu$ (respectively $[A]^{<\mu}$) for the set of all subsets of $A$ of cardinality $\mu$ (respectively strictly less than $\mu$).\myindex{$[A]^{<\mu}$}\myindex{$[A]^\mu$}

Recall that a (meet) \emph{semilattice}\myindex{semilattice}\myindex{lattice} is a partial order $(I, \le)$ so where every two elements have a meet (i.e.\ a greatest lower bound). We will often identify $I$ with $(I, \le)$. We write $s \land t$\myindex{$s \land t$} for the meet of two elements $s, t \in I$. In this paper, $I$ will almost always be finite, and in this case we write $\bot$\myindex{$\bot$} for $\land I$ (when $I$ is not empty), the least element of $I$. We will also use interval notation: for $u \le v$ both in $I$, we write $[u, v]$ for the partial order with universe $\{w \in I \mid u \le w \le v\}$, similarly for $(u, v)$, $(u, \infty)$, and other interval notations.

When $I$ and $J$ are semilattices, we write $I \subseteq J$ to mean that $I$ is a \emph{subsemilattice}\myindex{subsemilattice} of $J$. In particular, the meet operation in $I$ and $J$ agree  on $I$. Notice that any initial segment of a semilattice is a subsemilattice.

Given $I$ and $J$ partial orders, we write $I \times J$\myindex{$I \times J$} for their ordered product: $(i_1, j_1) \le (i_2, j_2)$ if and only if $i_1 \le i_2$ and $j_1 \le j_2$. When $I$ and $J$ are semilattices, $I \times J$ is a semilattice: the meet is given by the meet of each coordinate.

For $u$ a set, we write $\Ps (u)$\myindex{$\Ps (u)$} for the semilattice $(\Ps (u), \subseteq)$, where the meet operation is given by $s \land t = s \cap t$ and $\bot = \emptyset$. We write $\Psm (u)$\myindex{$\Psm (u)$} for the semilattice on $\Ps (u) \backslash \{u\}$ induced by $\Ps (u)$. Very often, we will have $u = n = \{0, 1, \ldots, n - 1\}$\myindex{$\Ps (n)$}\myindex{$\Psm (n)$}. More generally, the semilattices we will work with will usually be finite initial segment of $\Ps (\omega)$.

Note that for any set $u$, $\Ps (u) \cong \Ps (|u|)$, and $\Psm (u) \cong \Psm (|u|)$. Moreover, $\Ps (n) \times \Ps (m) \cong \Ps (n + m)$ for $n, m < \omega$. 

$\delta$ denote a limit ordinal.

%%%%%%%%%%%%%%%%%%%%%%%%%%%%%%%%%%%%%%%%%%%%%%
\newpage

\section{Abstract classes and skeletons}\LABEL{ac-sec}

The definition of an abstract class gives the minimal properties we would like any class of structures discussed in this paper to satisfy. We also discuss the notion of a \emph{skeleton} of an abstract class: a subclass that captures many of the essential properties of the original class but may be more manageable.

We study special kind of abstract classes. \emph{Fragmented AECs} are the main new concept: they are AECs, except that the chain axiom is weakened so that union of chains that jump cardinalities may not be in the class. Class of saturated models in a superstable first-order theory are an example, and we will use them extensively to deal with such setups.

We finish this section with a review of some standard definitions and results on AECs.

Let us start with the definition of an abstract class (due to Rami Grossberg).

\begin{defin}\LABEL{ac-def}\myindex{abstract class}
  An \emph{abstract class} is a pair $\K = (K, \lea)$ such that $K$ is a class of structures in a fixed (here always finitary) vocabulary $\tau = \tau (\K)$, $\lea$ is a partial order on $K$ extending $\tau$-substructure, and $\K, \, \leq_{\K}$ are closed under isomorphisms. 
\end{defin}

We often do not distinguish between $\K$ and its underlying class, writing for example ``$M \in \K$''. When we need to be precise, we will use the following notation:

\begin{notation}\LABEL{ac-notation}
  Let $\K = (K, \lea)$ be an abstract class.

  \begin{enumerate}
  \item[(A)] We write $|\K|$ for the underlying class $K$ of $\K$.\myindex{$\barsign \K \barsign$}
  
  \item[(B)] Let $K_0 \subseteq K$ be a class of structures which is closed under isomorphisms. We let $\K \, {\rest} \, K_0 \coloneqq  (K_0, \lea \, {\rest} \, K_0)$, where $\lea \, {\rest} \, K_0$ denotes the restriction of $\lea$ to $K_0$.\myindex{$\K_0 \subseteq \K$}
  \end{enumerate}
  
\end{notation}

We will also use the following standard notation:

\begin{notation}\LABEL{r0}\myindex{$\K_\lambda$}\myindex{$\K_{\ge \lambda}$}\myindex{$\K_{\le \lambda}$}\myindex{$\K_{\Theta}$}
  Let $\K$ be an abstract class and let $\Theta$ be a class of cardinals. We write $\K_{\Theta}$ for the abstract class $\K \, {\rest} \, \{M \in \K \mid \|M\| \in \Theta\}$. That is, it is the abstract class whose underlying class is the class of models of $\K$ of size in $\Theta$, ordered by the restriction of the ordering of $\K$. We write $\K_{\lambda}$ instead of $\K_{\{\lambda\}}$, $\K_{\le \lambda}$ instead of $\K_{[0, \lambda]}$, etc. 
\end{notation}

We define amalgamation, orbital types and $\mathscr{S}_{\K}(M)$ the set of orbital types (usually omit $\K$) stability, etc.\ for abstract classes as in the preliminaries of \cite{sv-infinitary-stability-afml} but we mostly use the notation and terminology of \cite{shelahaecbook}\index{amalgamation}\index{orbital type}\index{stability}. In particular we write $\gtp_{\K} (\bb / A; N) \in \mathscr{S}(A, N)$ \myindex{$\tp_{\K} (\bb / A; N)$} for the orbital type of $\bb$ over $A$ in $N$ as computed inside $\K;$ this is really useful when we have enough amalgamation. Usually, $\K$ will be clear from context, so we will omit it.

The following locality notion (also called tameness) will play an important role. It is implicit already in \cite{sh394}, but is it considered as a property of AECs per se by Grossberg and VanDieren in \cite{tamenessone}.

\begin{defin}\LABEL{def-local}\myindex{local}\myindex{locality}
  \     
  
  (1) Let $\K$ be an abstract class and let $\kappa$ be an infinite cardinal.

  (2) $p \in \mathscr{S}(M)$ is called algebraic \myindex{algebraic} when $p = \gtp(a /  M; M)$ for some $a \in M.$ 
  
  (3) We say that $\K$ is \emph{$(<\kappa)$-tame} if for any $M \in \K$ and any $p, q \in \gS (M)$, if $p \, {\rest} \, A = q \, {\rest} \, A$ for any $A \in [M]^{<\kappa}$, then $p = q$. We say that $\K$ is \emph{$\kappa$-tame} if it is $(<\kappa^+)$-tame.
\end{defin}

We will also use the following standard notions of saturation:

\begin{defin}\LABEL{sat-defs}
  Let $\K$ be an abstract class and let $M \lea N$ both be in $\K$.

  \begin{enumerate}
  \item We say that $N$ is \emph{universal over $M$} if whenever $N_0$ is such that $M \lea N_0$ and $\|N_0\| = \|M\|$, we have that there exists $f: N_0 \xrightarrow[M]{} N$.\myindex{universal over}
  
  \item We say that $N$ is \emph{$(\lambda, \delta)$-limit over $M$} if $\Vert M \Vert = \Vert N \Vert$ and there exists an increasing continuous chain $\seq{M_i  \colon i \le \delta}$ in $\K_\lambda$ such that $M_0 = M$, $M_\delta = N$, and $M_{i + 1}$ is universal over $M_i$ for all $i < \delta$.\myindex{$(\lambda, \delta)$-limit over (a model)}
  
  \item We say that $N$ is \emph{limit over $M$} if it is $(\lambda, \delta)$-limit over $M$ for some $\lambda$ and $\delta$.\myindex{limit over (a model)} 
  
  \item\label{brimmed-set-def} We say that $N$ is \emph{limit over $A$} (for $A \subseteq |N|$ a set\footnote{Pedantically, if $A = \vert M \vert$, $M \leq_{\K} N,$ then \ref{sat-defs}(3), \ref{sat-defs}(4) are not equivalent, but our notion and context tell us what we mean.}) if there exists $M' \in \K$ such that $M' \lea N$, $A \subseteq |M'|$, and $N$ is limit over $M'$.\myindex{limit over (a set)} Note that it follows from the definition that $\Vert M' \Vert = \Vert N \Vert$ and that the $\delta$ witnessing the length of the chain is always (is a limit ordinal and) strictly less $\lambda^{+}.$ On the other hand $A$ could be very small relative to $N$ (nevertheless, in this paper it will most often be the union of a system of models of the same cardinality as $N$).
  
  \item We say that $N$ is \emph{limit} if it is limit over some $M'$.\myindex{limit}
  
  \item We say that $M$ is \emph{$\lambda$-model-homogeneous} if whenever $M_0 \lea N_0$ are both in $\K$ with $M_0 \lea M$ and $\|N_0\| < \lambda$, there exists $f: N_0 \xrightarrow[M_0]{} M$. We say that $M$ is \emph{model-homogeneous} if it is $\|M\|$-model-homogeneous.\myindex{model-homogeneous}
  
  \item We say that $M$ is \emph{$\lambda$-saturated} if for any $M_0 \lea M$ with $M_0 \in \K_{<\lambda}$, any $p \in \gS (M_0)$ is realized inside $M$. We say that $M$ is \emph{saturated} if it is $\|M\|$-saturated.\myindex{saturated}
  \end{enumerate}
\end{defin}

We give the definition of a few more useful general concepts:

\begin{defin}\LABEL{dom-def}\myindex{domain of an abstract class}\myindex{$\dom{\K}$}
  The \emph{domain} $\dom (\K)$ of an abstract class $\K$ is the class of cardinals $\lambda$ such that $\K_{\lambda} \neq \emptyset$.
\end{defin}

\begin{defin}\LABEL{r3}\myindex{sub-abstract class}
  A \emph{sub-abstract class} of an abstract class $\K = (K, \leap{\K})$ is an abstract class $\K^\ast = (K^\ast, \leap{\K^\ast})$ such that $K^\ast \subseteq K$ and for $M, N \in K^\ast$, $M \leap{\K^\ast} N$ implies $M \lea N$.
\end{defin}

The definition of a skeleton is due to the second author \cite[5.3]{indep-aec-apal}. We have further weakened the requirement on chains appearing there (although the real ``spirit'' of the definition remains!). This is useful in order to do study skeletons in a purely finitary way (i.e. without worrying about existence of colimits/unions of chains).

\begin{defin}\LABEL{skel-def}\myindex{skeleton of an abstract class}
  A \emph{skeleton} of an abstract class $\K$ is a sub-abstract class $\K^\ast$ of $\K$ such that:

  \begin{enumerate}
  \item[(A)]\LABEL{skel-1} For any $M \in \K$, there exists $N \in \K^\ast$ with $M \lea N$.
  \item[(B)]\LABEL{skel-2} For any $M, N \in \K^\ast$ with $M \lea N$, there exists $N' \in \K^\ast$ such that $M \leap{\K^\ast} N'$ and $N \leap{\K^\ast} N'$.
  \end{enumerate}
\end{defin}

Examples of skeletons will be given later (see Fact \ref{tameness-ap}). A simple example the reader can keep in mind is the class of $\aleph_0$-saturated models of a first-order theory, ordered either by elementary substructure or by being ``universal over''.

The following apparent strengthening of the definition of a skeleton will be useful.

\begin{lem}\LABEL{skel-order-lem}
  Let $\K^\ast$ be a skeleton of $\K$. Let $M \in \K$, $n < \omega$, and let $\seq{M_i : i < n}$ be a (not necessarily increasing) sequence of elements of $\K^\ast$ such that $M_i \lea M$ for all $i < n$. Then there exists $N \in \K^\ast$ such that $M \lea N$ and $M_i \leap{\K^\ast} N$ for all $i < n$.
\end{lem}
\begin{proof}
  We work by induction on $n$. If $n = 0$, use (\ref{skel-1}) in the definition of a skeleton to find $N \in \K^\ast$ with $M \lea N$. Assume now that $n = m + 1$. By the induction hypothesis, let $N_0 \in \K^\ast$ be such that $M \lea N_0$ and $M_i \leap{\K^\ast} N_0$ for all $i < m$. Now apply \ref{skel-def}(B) in the definition of a skeleton, with $M, N$ there standing for $M_m, N_0$ here. We get $N \in \K^\ast$ such that $N_m \leap{\K^\ast} N$ and $N_0 \leap{\K^\ast} N$. By transitivity of $\leap{\K^\ast}$, we have that $M_i \leap{\K^\ast} N$ for all $i < m$, so $N$ is as desired.
\end{proof}

We now give the definition of several properties that an abstract class may have. In particular, it could be a fragmented AEC, or even just an AEC.

\begin{defin}\LABEL{r5}
  Let $\K$ be an abstract class.
  
  \begin{enumerate}
  \item Let $\LS (\K)$ be the least cardinal $\lambda \ge |\tau (\K)| + \aleph_0$ such that for any $M \in \K$ and any $A \subseteq |M|$, there is $M_0 \in \K$ with $M_0 \lea M$, $A \subseteq |M_0|$, and $\|M_0\| \le |A| + \lambda$. When such a $\lambda$ does not exist, we write $\LS (\K) = \infty$.\myindex{$\LS (\K)$}\myindex{L{\"o}wenheim-Skolem-Tarski number}
  
  \item Let $[\lambda_1, \lambda_2)$ be an interval of cardinals and let $\delta$ be a limit ordinal. We say that $\K$ is \emph{$([\lambda_1, \lambda_2), \delta)$-continuous} if whenever $\seq{M_i \colon i < \delta}$ is an increasing chain in $\K_{[\lambda_1, \lambda_2)}$, then $M_\delta \coloneqq  \bigcup_{i < \delta} M_i$ is such that $M_\delta \in \K$ and $M_0 \lea M_\delta$. We say that $\K$ is \emph{$[\lambda_1, \lambda_2)$-continuous} if it is $([\lambda_1, \lambda_2), \delta)$-continuous for all limit ordinals $\delta,$ (note that possibly $M_{\delta} \in \K_{\lambda_{2}}$). We say that $\K$ is \emph{continuous} if it is $\dom (\K)$-continuous defined similarly\footnote{Pedantically, $\dom(\K)$ may not be an interval.}.\myindex{continuous abstract class} Note we do \emph{not} require that the union is a \emph{least} upper bound: this is called smoothness later. 
            
  \item We say that $\K$ is \emph{coherent} if whenever $M_0, M_1, M_2 \in \K$, if $|M_0| \subseteq |M_1| \subseteq |M_2|$, $M_0 \lea M_2$, and $M_1 \lea M_2$, then $M_0 \lea M_1$\myindex{coherent abstract class} (this is Ax IV of AEC).
  
  \item $\K$ is a \emph{fragmented abstract elementary class} (AEC) if it satisfies:\myindex{fragmented abstract elementary class}:

    \begin{enumerate}
    \item (Coherence) $\K$ is coherent.
    \item (L{\"o}wenheim-Skolem-Tarski axiom) $\LS (\K) < \infty$.\myindex{L{\"o}wenheim-Skolem-Tarski axiom}
    \item (Restricted chain axioms) Let $\seq{M_i : i < \delta}$ be an increasing chain in $\K$. Let $M_\delta \coloneqq  \bigcup_{i < \delta} M_i$.\myindex{restricted chain axioms}\myindex{chain axioms}
    \begin{enumerate}
    \item \emph{If} $\|M_i\| = \|M_\delta\|$ for some $i < \delta$, then:
      \begin{enumerate}
        \item $M_\delta \in \K$.
        \item $M_0 \lea M_\delta$.
      \end{enumerate}
    \item (Smoothness) If $N \in \K$ is such that $M_i \lea N$ for all $i < \delta$, then $M_\delta \in \K$ and $M_{\delta} \lea N$.\myindex{smoothness axiom}
    \end{enumerate}
    \end{enumerate}
  \item 
  \begin{enumerate}
      \item[(a)] $\K$ is an \emph{abstract elementary class (AEC)} if it is a continuous fragmented AEC.\myindex{abstract elementary class}\index{AEC|see {abstract elementary class}}
      
      \item[(b)] Let AAEC (almost AEC)\myindex{AAEC} be defined similarly omitting the demand on $\LS(\K)$ exists. 
      
      \item[(c)] We say the abstract $\K$ has $\mu$-\emph{amalgamation}\myindex{$\mu$-amalgamation} \underline{when} if $M_{0} \leq_{\K} M_{\ell}$ for $\ell = 1, 2$ are from $\K_{\mu}$ then there are $N, f_{1}, f_{2}$ such that $M_{0} \leq_{\K} N \in \K_{\mu}$ and $f_{\ell}$ is a $\leq_{\K}$-embedding of $M_{\ell}$ into $N$ over $M_{0}$ for $\ell = 1, 2.$
      
      \item[(d)] We say $\K$ has \emph{local amalgamation}\myindex{local amalgamation} \underline{when} for every $\mu$ it has $\mu$-amalgamation. 
      
      \item[(e)]  We say $\K$ has \emph{amalgamation}\myindex{amalgamation} \underline{when} if $M_{0} \leq_{\K} M_{\ell}$ for $\ell = 1, 2$ are from $\K_{\mu},$ then there are $N, f_{1}, f_{2}$ such that $M_{0} \leq_{\K} N \in \K_{\ell}$ and $f_{\ell}$ is a $\leq_{\K}$-embedding of $M_{\ell}$ into $N$ over $M_{0}$ for $\ell = 1, 2.$ 
  \end{enumerate}

  \item We define also the notion of a \emph{[very] weak [fragmented] AEC}, where ``weak'' means that in addition the smoothness axiom may not hold and ``very weak'' means that the coherence axiom may also not hold.\myindex{weak AEC}
  
  \item We say $\s$ is a \emph{semi fragmented} AEC \myindex{semi fragmented AEC} \underline{when} in part (4) we replace clause (c)(ii) by: 
  
  \begin{enumerate}
      \item[(ii)'] (Weak smoothness) If $N \in \K$ is such that $M_{i} \leq_{\K} N$ for all $i < \delta$ and $M_{\delta} \in \K,$ then $M_{\delta} \leq_{\K} N$ (used in \ref{r7}).      
  \end{enumerate}
  \end{enumerate}
\end{defin}

Compared to AECs, fragmented AECs are not required to be closed under unions of ``long'' chains, i.e. those whose union is of bigger cardinality than the pieces. In some sense, they have to be studied ``cardinal by cardinal'', at least more so than AECs, and this is why they called ``fragmented''. 

The following are examples of fragmented AECs. The generalization of the first one to AECs will play an important role in the present paper. The second will not be studied here, but we give it as an additional motivation.

\begin{example}\LABEL{r7}
  Let $T$ be a superstable first-order theory and let $\K$ be its class of saturated models, ordered by elementary substructure. Then $\K$ is a semi fragmented AEC with $\LS (\K) \le 2^{|T|}$. It is \emph{not} an AEC because (for $T$ countable say) the union of an $\aleph_{2}$-chain of saturated models of cardinality $\aleph_{1}$ does not have to be saturated (it will only be $\aleph_{1}$-saturated). 
\end{example}

\begin{example}\LABEL{r9}
  Let $\phi$ be a complete $\Ll_{\omega_1, \omega} (Q)$-sentence ($Q$ is the quantifier ``there exists uncountably many'') and let $\Phi$ be a countable fragment of $\Ll_{\omega_1, \omega} (Q)$ containing $\phi$. Assume that for every  $\psi (\bx) \in \Phi$, there exists a predicate $R_\psi (\bx)$ such that $\phi \models \forall \overline{x}( \psi(\overline{x}) \leftrightarrow R_{\psi}(\overline{x})).$ Let $K$ be the class of models $M$ of $\Ll_{\omega, \omega}$ consequences of $\phi$ which also is a model $\psi$ if $M$ uncountable and if $\psi(\bar{x}) = \bigvee_{n < \omega} \psi_{n}(\bar{x})$ belongs to $\Phi,$ then $M$ omit the type $\{ \neg R_{\psi}(\bar{x}) \cap R_{\psi_{n}}: n < \omega \}$. Order it by $M \lea N$ if and only if $M \preceq N$ and for any $\psi (x, \by) \in \Phi$, if $M \models R_{\neg Q x \psi (x, \by)}[\ba]$, then $\psi (M, \ba) = \psi (N, \ba)$. \underline{Then} $\K$ is a fragmented AEC with $\LS (\K) = \aleph_0$. It is \emph{not} an AEC because an $\aleph_{1}$-union $M$ of countable models in $\K$ may fail to be in $\K,$ because maybe $\psi(\bar{y}) \equiv Qx\varphi(x, \bar{y}) \in \Phi$ and $M \models \psi[\bar{b}]$ but $\varphi(M, \bar{b})$ is countable. 
\end{example}

In any AEC, models of big cardinality can be resolved as an increasing union of smaller models. We make this into a definition:

\begin{defin}\LABEL{resolv-def}\myindex{resolvable}
  Let $\K$ be an abstract class and let $M_0, M \in \K$ with $M_0 \lta M$. 
  
  (1) We say that $M$ is \emph{$\delta$-resolvable over $M_0$} if there exists a strictly increasing continuous chain $\seq{N_i : i \le \delta}$ such that:

  \begin{enumerate}
      \item[(A)] $N_0 = M_0$, $N_\delta = M$.
      
      \item[(B)] $\|N_{i}\| = \|M_0\| + |i| < \Vert M \Vert$ for all $i < \delta$ (so necessarily $\delta = \Vert M \Vert,$ hence $\cf \delta = \rm{cf}( \Vert M \Vert)$).
  \end{enumerate}

  1A) We say that $M$ is \emph{resolvable over $M_0$} if it is $\delta$-resolvable over $M_0$ for some limit ordinal $\delta$.
  
  (2) We say that $M$ is \emph{resolvable} if it is resolvable over $M_0$ for every $M_0 \in \K$ with $M_0 \lea M$ and $\|M_0\| < \|M\|$.
  
  (3) We say that $\K$ is \emph{resolvable} if any $M \in \K$ is resolvable.
  
  (4) $\K$ is $< \lambda$-resolvable when every $M \in \K_{\mu}$ is $\mu$-resolvable whenever $\LS(\K) < \mu < \lambda.$ 
\end{defin}

\begin{remark}\LABEL{r10}\ 
    
    (1) Why we have ``semi fragmented AEC'' in \ref{r5}(7) not using only ``fragmented AEC''? 
    
    As then Example \ref{r7} fail as $\K$ is not a fragmented AEC as clause (c)(ii) of Definition \ref{r5}(4) in general fail. 
    
    (2) We can change Definition \ref{resolv-def}(1) replacing clause (B) by: 
    
    \begin{enumerate}
        \item[(B)$'$] $\Vert N_{i} \Vert = \Vert M_{0} \Vert + \vert i \vert$ for all $i \leq \delta.$ \myindex{reasonable$^{*}$}.
    \end{enumerate}
    
    Is it interesting? The difference is that we may allow $\Vert M_{0} \Vert = \Vert M \Vert,$ anyhow then we can use any limit $\delta < \Vert M \Vert^{+}.$ 
    
    (3) Many times we can use ``semi fragmented AEC'' if we add ``$\K_{\s}$ is resolvable'', and sometimes just if $p \in \mathscr{S}_{\s}(M)$ then $p$ does not fork over $N$ for some $N \leq_{\K} M$ of cardinality $\LS(\K).$   
    
    (4) Clearly for a fragmented AEC $\K,$ it is $\mu$-resolvable for every $\mu > \LS(\K).$ So we may wonder, for a semi fragmented $\K,$ is every $M \in \K$ resolvable? Let us try to prove ``every $M \in \K$ is reasonable'' assuming $\mu < \LS(\K) \Rightarrow \K_{\mu} = \emptyset.$ So assume $M_{0} \leq_{\K} M, \Vert M_{0} \Vert < \Vert M \Vert.$ Fix $\mu \geq \LS(\K).$ We try to prove by induction on $\lambda > \mu$ that $M$ is resolvable \underline{when} $\Vert M \Vert = \lambda, \, M_{0} \leq_{\K} M $ and $ \Vert M_{0} \Vert \geq \mu.$ We have partial success: we are stuck when $\lambda$ is (probably weakly) inaccessible. But then we \blueq{give} a counterexample.  
    
    \underline{Case 1}: $\lambda = \chi^{+}.$
    
    Then by induction hypothesis, without loss of generality $\Vert M_{0} \Vert = \chi.$
    
    Let $\{ a_{\alpha}: \alpha < \lambda \}$ list $\vert M \vert$ with $\vert M_{0} \vert = \{ a_{\alpha}: \alpha < \chi \}.$ By induction on $\alpha \in [\chi, \lambda)$ choose $N_{\alpha} \in K_{\chi}$ such that $\langle N_{\beta}: \beta \leq \alpha \rangle$ is $\leq_{\K}$-increasing continuous, $N_{0} = M_{0}$ and $\{ a_{\beta}: \beta < \alpha \} \subseteq N_{\alpha}.$ For $\alpha = \chi,$ let $N_{\alpha} = M_{0},$ for $\alpha = \beta +1$ use $A = \vert N_{\beta} \vert \cup \{ a_{\beta} \}$ recalling $\vert \alpha \vert = \chi + \vert \alpha \vert \geq \LS(\K)$ and $\alpha < \Vert M \Vert.$ So by ``$\LS(\K) \leq \chi$'' there is $N \leq_{\K} M$ of cardinality $\vert \chi \vert$ which included $N_{\beta} \vert \cup \{ a_{\beta} \}.$ As $\K$ is coherent (see Definition \ref{r5}(3)) and $N_{\beta} \leq_{\K} M, \vert N_{\beta} \vert \leq \vert N_{\alpha} \vert$ we have $N_{\beta} \leq_{\K} N_{\alpha}$ and similarly (or by transitivity) $i < \alpha \Rightarrow N_{i} \leq_{\K} N_{\alpha}.$ Lastly for limit $\alpha,$ let $N_{\alpha} = \bigcup \{ N_{\beta}: \beta < \alpha \}$ so as $\K$ satisfies clause (c)(ii) of the Definition \ref{r5}, we have $N_{0} \leq_{\K} N_{\alpha} \in \K_{\chi}$ and similarly $i < \alpha \Rightarrow N_{i} \leq_{\K} N_{\alpha}.$ Finally $N_{\alpha} \leq_{\K} M$ by Definition \ref{r5}(7)(ii'). 
    
    Clearly we are done for this case. 
    
    \underline{Case 2}: $\lambda$ is singular.
    
    Let $\kappa = \cf \lambda < \lambda.$ Let $\overline{\mu} = \langle \mu_{i}: i < \kappa \rangle$ be an increasing continuous sequence of cardinality with limit $\lambda$ such that $\mu_{0} = \mu.$ Let $\langle A_{i}: i < \kappa \rangle$ be $\subseteq$-increasing continuous sequence of subsets of $M$ with union $\vert M \vert$ such that $\vert A_{i} \vert = \mu_{i}.$ 
    
    We choose $\langle N_{n, i} f_{n, i}: i < \kappa \rangle$ by induction on $n$ such that: 
    
    \begin{enumerate}
        \item[$(*)_{1}$ (a)] $N_{n, i} \leq_{\K} M$ has cardinality $\mu_{i},$
        
        \item[(b)] $f_{n, i}$ is a $1$-to-$1$ function from $\mu_{i}$ onto $N_{n, i},$
        
        \item[(c)] $N_{n, i}$ include $A_{n, i} = \bigcup \{f_{m, j}''(M_{i} \cap \mu_{j}): m < n, j < \kappa \} \cup A_{i}.$
    \end{enumerate}
    
    Why can we carry the induction? Assuming we arrive to $i, A_{n, i}$ is a subset of $M$ of cardinality $\mu_{i}$ hence (as $\LS(\K) \leq \mu_{0} \leq \mu_{i}$) there is $N_{m, i}$ as required in $(*)_{1}(a), (c).$ Then choose $f_{n, i}$ as in $(*)_{1}(b).$ 
    
    Next, 
    
    $(*)$ let $N_{i} = \bigcup \{ N_{n, i}: n < \omega \}.$
    
    Now, $\langle N_{n, i}: n < \omega \rangle$ is $\subseteq$-increasing and $N_{n, i} \leq_{\K} M$ by $(\ast)_{1}$(a) and $\K$ being coherent hence this sequence is $\leq_{\K}$-increasing. As $\Vert N_{n, i} \Vert = \mu_{i},$ we have $N_{n, i} \leq_{\K} N_{i} \in \K$ for $n < \omega;$ hence by Definition \ref{r5}(7)(c)(ii') also $N_{i} \leq_{\K} M.$ 
    
    Also  if $i < j,$ then $N_{n, i} \subseteq N_{n, j}$ hence $N_{i} = \bigcup_{n \in \omega} N_{n, i} \subseteq \bigcup_{n \in \omega} N_{n, j} \subseteq N_{j},$ so as $N_{i} \leq_{\K} M, \, N_{j} \leq_{\K} N.$ By coherence we have $N_{i} \leq_{\K} N_{j}.$ As $A_{i} \subseteq N_{i};$ clearly $\bigcup \{ \vert N_{i} \vert: i < \kappa \} = \vert M \vert,$ so together $M = \bigcup_{i < \kappa} N_{i}.$ Lastly let $j < \kappa$ be a limit ordinal we should prove that $\bigcup_{i < j} N_{i} = N_{j},$ clearly $\bigcup_{i < j} N_{i} \subseteq N_{j}.$ For the other inclusion, $N_{j} = \bigcup \{ N_{n, j}: n < \omega \} = \bigcup \{ f_{n, j}(\mu_{j}): n < \omega \} = \bigcup \{ f_{n, i}(\mu_{i}): n < \omega, i < j \} \subseteq \bigcup \{ N_{n +1, i}: n < \omega, i < j \} \subseteq \bigcup \{ N_{i}: i < j \}.$
    
    \underline{Case 3}: $\lambda > \mu$ is a limit regular cardinal. 
    
    Here \underline{not clear} what to do, so we shall to try to build a counterexample in part ($4$).
    
    (4) Let $S$ be a set or class of cardinals. We define $\K = \K_{S}$ as follow:
    
    \begin{enumerate}
        \item[$(*)_{1}$ (a)] $K$ is the class of linear orders not isomorphic to  $(\mu, <)$ for any $\mu \in S,$
        
        \item[(b)] $M \leq_{\K} N$ \underline{iff} $N \subseteq N.$ 
    \end{enumerate}
    
    It is easy to verify (noting that if $M \in \K, \Vert M \Vert = \mu, M \ncong (\mu, <)$ then no linear order $N \supseteq M$ is isomorphic to $(\mu, <)$ because either $M$ is not well ordered or it is isomorphic to $(\alpha, <)$ for some $\alpha \in [\mu +1, \mu^{+})$). Hence easily. 
    
    \begin{enumerate}
        \item[$(*)_{2}$] $\K$ is a fragmented AEC with $\LS(\K) = \aleph_{0}.$  
    \end{enumerate}
    
    However, 
    
    \begin{enumerate}
        \item[$(*)_{3}$] if $\lambda$ is a uncountable regular limit cardinal and $\lambda \notin S, \lambda \cap S$ is \underline{stationary}, then $M = (\lambda, <) \in \K$ is not resolvable.
    \end{enumerate}    
    
    (5) Another example: 
    
    For a cardinal $\lambda$ we define $\K = \K^{[\lambda]}$ be: 
    
    \begin{enumerate}
        \item[$(*)$] 
        
        \begin{enumerate}
            \item[(a)] $M \in K$  \underline{iff} $M$ is a linear order, well ordered of order type in $[\lambda, \lambda^{+}],$
            
            \item[(b)] $M_{1} \leq_{\K} M_{2}$ \underline{iff} $M_{1}$ is an initial segment of $M_{2}.$ 
        \end{enumerate}
    \end{enumerate}
\end{remark}

It is known that any class of structures is contained in a smallest AEC. Indeed, by \cite[I.2.19]{shelahaecbook}:

\begin{fact}\LABEL{aec-intersec}
  Let $\{\K_i : i \in I\}$ be a non-empty collection of AECs, all in the same vocabulary $\tau$. Then $\bigcap_{i \in I} \K_i$ (defined as $(\bigcap_{i \in I} |\K_i|, \bigcap_{i \in I} \leap{\K_i})$) is an AEC with $\LS(\K) \leq \sup \{ \LS(\K_{i}): i \in I \} + \vert I \vert.$
\end{fact}

Also, 

\begin{fact}\LABEL{aec-generated}\ 
  
  (1) For any abstract class $\K$, there exists a unique smallest AAEC $\K^\ast$ such that $\K$ is a sub-abstract class of $\K^\ast$.
  
  (2) If $\lambda = \LS(\K)$ \underline{then} $\K^{\ast}$ is an AEC and $\LS(\K^{\ast}) \leq \lambda.$ 
  
  (3) If $\K$ is a set then $\K^{\ast}$ is an AEC when $\LS(\K^{\ast}) \leq \sum(\dom(\K)).$ 
\end{fact}

\begin{proof}\ 
  
  (1) Take the intersection of the collection of all AECs that include $\K$ as a sub-abstract class. Note that this collection is non-empty: it contains the AEC of all $\tau (\K)$-structures, ordered with substructure. Pedantically this is not O.K. as we quantify on classes. As earlier we were not accurate and to help part (2), we give full details.
  
  First we fix  $\chi,$ then we define $(K_{\chi, \alpha}^{\ast},  \leq_{\chi, \alpha}^{\ast})$ by induction on $\alpha \leq (2^{ < \chi})^{+}$ such that: \index{$K_{\chi, \alpha}^{\ast}$}
 
  \begin{enumerate}
      \item[$(*)$ \  (a)] $(K_{\chi, \alpha}^{\ast}, \leq_{\chi, \alpha}^{\ast})$ is an abstract class, except the transitivity of the order is omitted,  
      
      \item[(b)] $K_{\chi, \alpha}$ is a set of $\tau(\K)$-models from $\mathscr{H}(\chi),$
      
      \item[(c)] $\leq_{\chi, \alpha}$ is a two-place relation on $\K_{\chi, \alpha},$
      
      \item[(d)] if $\beta < \alpha,$ then $K_{\chi, \beta} \subseteq K_{\chi, \alpha}$ and $\leq_{\chi, \beta} \, \subseteq \, \leq_{\chi, \alpha},$
      
      \item[(e)] if $M \leq_{\chi, \alpha} N,$ then $M \subseteq N.$
  \end{enumerate}
  
  The induction is: 
  
  \underline{Case 1}: $\alpha = 0.$
  
  $(K_{\chi, \alpha}, \leq_{\chi, \alpha}) = \K \, {\rest} \, \mathscr{H}(\chi).$
  
  \underline{Case 2}: $\alpha$ a limit ordinal.
  
  $K_{\chi, \alpha} = \bigcup \{ K_{\chi, \beta}: \beta < \alpha \}$ and $\leq_{\chi, \alpha} = \bigcup \{ \leq_{\chi, \beta}: \beta < \alpha \}.$
  
  \underline{Case 3}: $\alpha = 3\beta + 1.$
  
  $K_{\chi, \alpha} = \{ M \in \mathscr{H}(\chi):$ for some directed partial order and sequence $\langle M_{s}: s \in I \rangle$ of members of $K_{\chi, 3 \beta}$ such that $s \leq_{I} t \Rightarrow M_{s} \leq_{\chi, 3 \beta} M_{t},$ we have $M = \bigcup_{s \in I} M_{s} \}.$ Also, $\leq_{\chi, \alpha} = \{ (M_{1}, M_{2}) \in \mathscr{H}(\chi):$ for some directed partials orders $I_{1} \subseteq I_{2}$ and sequence $\overline{M}^{\ell} = \langle M_{s}^{\ell}: s \in I_{2} \rangle$ of members of $K_{\chi, 3\beta}$ such that $s \leq_{I_{1}} t \Rightarrow M_{s}^{\ell} \leq_{\chi, 3 \beta} M_{t}^{\ell},$ we have $M_{\ell} = \bigcup_{s \in I_{\ell}} M_{s}$ for $\ell = 1, 2 \}.$
  
  \underline{Case 4}: $\alpha = 3 \beta +2.$ 
  
  $K_{\chi, 3\beta + 2} = K_{\chi, 3\beta +1 }$ and $\leq_{\chi, 3 \beta + 2} \, = \, \leq_{3 \beta + 1} \cup \, \{ (M_{1}, M_{2}) \in \mathscr{H}(\chi): M_{1} \subseteq M_{2}$ are from $K_{\chi, 3 \beta + 1}$ and for some $N \in K_{\chi, 3\beta +2},$ we have $M_{1} \leq_{\chi, 3 \beta +1} N, M_{2} \leq_{\chi, 3 \beta + 1} N \}.$
  
  \underline{Case 5}: $\alpha = 3 \beta + 3.$
  
  $K_{\chi, 3\beta + 3} = K_{\chi, 3 \beta + 2}$ and $\leq_{\chi, 3 \beta + 2} \, = \, \leq_{3 \beta + 1} \cup \, \{ (M_{1}, M_{2}): M_{1} \subseteq M_{2}$ are from $K_{\chi, 3 \beta + 1}$ and for some $N \in K_{\chi, 3\beta + 2},$ we have $M_{1} \leq_{\chi, 3\beta + 2} N \leq_{\chi, 3\beta +2} M_{2} \}.$
  
  Having carried the induction on $\chi,$ 
  
  $(*)$ letting $\K_{\chi} = (K_{\chi, \alpha}, \leq_{\chi, \alpha})$ for $\alpha = (2^{< \chi})^{+}$ clearly $\K_{\chi}$ almost is an AAEC, the only missing point is union which are not in $\K_{\chi}$. 
  
  However, 
  
  $(*)$ if $\chi < \theta, \underline{then}$ $\K_{\chi} \subseteq \K_{\theta}.$ 
  
  So easily $\K^{\ast} = \bigcup \{ \K_{\chi}: \chi$ is a cardinal$\}$ is as required. 
  
  (2) For each $\chi > \LS(\K)$ we can prove by induction on the ordinal $\alpha$ that $\K_{\chi, \alpha}$ satisfies $\LS(\K_{\chi, \alpha}).$ The only non-trivial case is $\alpha = 3 \beta +1,$ so assume $M \in \K_{3 \beta +1} \setminus \K_{3 \beta}$ and let $\langle M_{s}: s \in I \rangle$ be as in case 3, let $\LS(\K) \leq \mu \leq \Vert M \Vert$ and $A \subseteq M$ be of cardinality $\leq \mu;$ without loss of generality $\Vert A \Vert = \mu.$ We can find $J_{0} \subseteq I$ of cardinality $\leq \mu$ such that $A \subseteq \bigcup \{ M_{s}: s \in J_{0} \},$ also we can find $J_{1} \subseteq I$ of cardinality $\vert J_{0} \vert + \aleph_{0} \leq \mu$ such that $J_{1} \subseteq J$ and $J_{1}$ is directed. Let $J \subseteq J_{1}$ be cofinal in $J_{1}$ and be well founded. Next let $\langle s_{\varepsilon}: \varepsilon < \varepsilon_{\ast} \rangle$ list $J$ such that $s_{\varepsilon} <_{I} s_{\zeta} \Rightarrow \varepsilon < \zeta.$ Now choose $\bar{N_{s_{\varepsilon}}}$  by induction on $\alpha < \alpha_{\ast}$ such that: 
  
  \begin{enumerate}
      \item[$(\ast)$ (a)] $N_{s_{\varepsilon}} \leq_{\K} M_{s_{\varepsilon}},$
      
      \item[(b)] $N_{s_{\varepsilon}}$ has cardinality $\leq \mu,$
      
      \item[(c)] $N_{s_{\varepsilon}}$ includes $\bigcup \{N_{s_{\varepsilon}} \cup (A \cap M_{s_{\varepsilon}}): \zeta < \varepsilon$ and $s_{\zeta} <_{I} s_{\varepsilon}\}.$
  \end{enumerate}
  
  Easy to carry the induction, $\langle N_{s_{\varepsilon}} \colon \varepsilon < \varepsilon_{\ast} \rangle$ i.e., $\langle N_{s} \colon s \in J \rangle$ is $\leq_{\K}$-increasing in $(\K_{\chi, 3 \beta})_{\mu}$ so by its definition $N = \bigcup \{ N_{s_{\varepsilon}}: \varepsilon < \varepsilon_{\ast} \}$ belongs to $\K_{\chi, \alpha}.$ But why $N \leq_{\K_{\chi, \alpha}} M?$ use the definition of $\leq_{\K_{\chi, \alpha}}$ with $J, I$ here standing for $I_{1}, I_{2}$ there. 
  
  (3) Follows by (2).    
\end{proof}

\begin{defin}\LABEL{r15}\myindex{AEC generated by}\index{generated by|see {AEC generated by}}
  \ 
  
  (1) Let $\K$ be an abstract class. We call the AAEC $\K^\ast$ given by Fact \ref{aec-generated}, the AAEC \emph{generated by $\K$}.
  
  (2) If $\K^{\ast}$ is an AEC then we may say $\K^{\ast}$ is the AEC generated by $\K.$\myindex{AEC generated by $\K$} 
\end{defin}

AECs are uniquely determined by their restrictions of size $\lambda$:

\begin{fact}[{\cite[II.1.23]{shelahaecbook}}]\LABEL{canon-aec}
  Let $\K^1$ and $\K^2$ be AECs with $\lambda \coloneqq  \LS (\K^1) = \LS (\K^2)$. If $\K_{\le \lambda}^1 = \K_{\le \lambda}^2$, then $\K^1 = \K^2$.
\end{fact}

We now look into how much of this uniqueness is carried over in fragmented AECs. We will specifically look at totally categorical fragmented AECs (this is not such a strong assumption, since classes of saturated models are examples).

First, only part of the fragmented AEC suffices in order to specify the AEC it generates. 

\begin{lem}\LABEL{generated-fragmented}
  Let $\K$ be a fragmented AEC with $\K_{<\LS (\K)} = \emptyset$. Then the AEC generated by $\K$ is the same as the AEC generated by $\K_{\LS (\K)}$.
\end{lem}
\begin{proof}
  Let $\lambda \coloneqq  \LS (\K)$. Let $\K^1$ be the AEC generated by $\K$ and let $\K^2$ be the AEC generated by $\K_{\le \lambda}$. Since $\LS (\K) = \lambda$, by Fact \ref{aec-generated} clearly $\LS (\K^1) = \lambda$. Of course, we also have that $\LS (\K^2) = \lambda$. Moreover, $\K_{\lambda}^1 = \K_{\lambda}^2$. By Fact \ref{canon-aec} and Lemma \ref{generated-fragmented}, $\K^1 = \K^2$. Alternatively see the proof of \ref{aec-generated}.
\end{proof}

Toward studying categoricity in fragmented AECs, we look at homogeneous models. As in \cite[I.2.5]{shelahaecbook}, we have:

\begin{fact}\LABEL{mh-uq}
  Let $\K$ be a fragmented AEC and let $M_0 \lea M_\ell$, $\ell = 1,2$. If $\|M_1\| = \|M_2\| > \LS (\K)$ and both $M_1$ and $M_2$ are model-homogeneous, then $M_1 \cong_{M_0} M_2$.
\end{fact}

We obtain:

\begin{thm}\LABEL{fragment-mh-categ}\ 

  (1) Let $\K$ be a fragmented AEC with amalgamation and $\K_{<\LS (\K)} = \emptyset$. Assume that for every $\lambda \in \dom (\K)$, $\K$ is categorical in $\lambda$ and stable in $\lambda$. Then any $M \in \K_{>\LS (\K)}$ is model-homogeneous.

  (2) This holds also for semi-fragmented $\K.$
\end{thm}

\begin{proof}
  Let $\lambda \in \dom (\K) \cap (\LS (\K), \infty)$ and let $M \in \K_\lambda$. Pick $M_0 \lea N_0$ in $\K_{<\lambda}$ with $M_0 \lea M$.
  
  (1) By categoricity and stability in $\lambda$, $M$ is $(\lambda, \lambda)$-limit. If $\lambda$ is regular, this means that we can find $M^0 \in \K_\lambda$ such that $M_0 \lea M^0$ and $M$ is universal over $M^0$. Now by amalgamation, there exists $N^0 \in \K_\lambda$ and $f \colon N_0 \xrightarrow[M_0]{} N^0$ such that $M^0 \lea N^0$. Now by the universality of $M$ over $M_{0},$ let $g \colon N^0 \xrightarrow[M^0]{} M$. Then $g f$ embeds $N_0$ into $M$ over $M_0$.

  If $\lambda$ is singular, let $\mu \coloneqq  \|N_0\|$. Pick $M^0 \lea M$ with $M^0 \in \K_{\mu^+}$ and $M^0$ containing $M_0$. By the previous case, $N_0$ embeds into $M^0$ (and hence into $M$) over $M_0$, as desired.
  
  (2) Also for semi fragmented AEC. Let $\mu = \Vert N_{0} \Vert + \LS(\K).$ So clearly $\Vert M_{0} \Vert + \LS(\K) \leq \mu < \lambda,$ hence there is $M_{1} \in K_{\mu}$ such that $M_{0} \leq_{\K} M_{1} \leq_{\K} M.$ As clearly $\Vert M_{1} \Vert = \mu < \lambda = \Vert M \Vert,$ there is $M_{2} \in \K_{\mu^{+}}$ such that $M_{1} \leq_{\K} M_{2} \leq_{\K} M$  and as $\K$ is categorical in $\mu^{+}, \mu$ and stable in $\lambda,$ necessarily $M_{2}$ is universal over $M_{1}.$ So there is an $\leq_{\K}$-embedding of $N_{0}$ into $M_{2}$ over $M_{0},$ so we are done.   
\end{proof}

\begin{thm}[Canonicity of categorical fragmented AECs]\LABEL{fragment-canon}
  Let $\K^1$ and $\K^2$ be fragmented AECs with $\Theta \coloneqq  \dom (\K^1) = \dom (\K^2)$ and $\LS (\K^1) = \LS (\K^2)$. Suppose that for $\ell = 1,2$, $\K^\ell$ has amalgamation and is stable and categorical in every $\lambda \in \Theta$. If $\K_{\le \LS (\K^1)}^1 = \K_{\le \LS (\K^2)}^2$, then $\K^1 = \K^2$.
\end{thm}
\begin{proof}
  Without loss of generality, $\K_{<\LS (\K^\ell)}^\ell = \emptyset$ for $\ell = 1,2$. Let $\lambda > \LS (\K)$ be such that $\lambda \in \Theta$ and assume inductively that we know $\K_{<\lambda}^1 = \K_{<\lambda}^2$. We have that the model $M^\ell$ of cardinality $\lambda$ in $\K^\ell$ is model-homogeneous (Theorem \ref{fragment-mh-categ}) for $\ell = 1,2$. By categoricity and the inductive assumption, there exists $M_\ell \in \K^\ell$ of cardinality strictly less than $\lambda$ with $M_\ell \lea M^\ell$ and $f: M_1 \cong M_2$. Now use the proof of Fact \ref{mh-uq} (the usual back and forth argument) to conclude that $M^1 \cong M^2$. This implies that $|\K_\lambda^1| = |\K_\lambda^2|$. To see that $\K_\lambda^1 = \K_\lambda^2$ (that is, the orderings also coincide), observe that the AECs generated by $\K^1$ and $\K^2$ must be the same (Fact \ref{canon-aec} and Lemma \ref{generated-fragmented}), hence the orderings must coincide on all models there.
\end{proof}

\begin{defin}\LABEL{sat-class-def}
  Let $\K$ be an abstract class.

  \begin{enumerate}
  \item We write $\Ksatp{\lambda}$ for the abstract class $\K \, {\rest} \, \{M \in \K \mid M \text{ is } \lambda\text{-saturated}\}$ of $\lambda$-saturated models in $\K$ ordered by the restriction of the ordering of $\K$.\myindex{$\Ksatp{\lambda}$}
  \item We write $\Ksat$ for the abstract class $\K \, {\rest} \, \{M \in \K \mid M \text{ is saturated}\}$.\myindex{$\Ksat$}
  
  \item We write $\Kbrim$ for the abstract class with $|\K^{\rm{lim}}| = \{M \in \K \mid M \text{ is limit}\}$ ordered by $M \leap{\K^{\rm{lim}}} N$ if and only if $M \lea N$ and one of the following conditions holds:

    \begin{enumerate}
    \item $M = N$.
    
    \item $\|M\| < \|N\|$.
    
    \item $N$ is limit over $M$ (so $\|M\| = \|N\|$).
    \end{enumerate}\myindex{$\Kbrim$}
  \end{enumerate}
\end{defin}

It is known \cite[5.7]{categ-saturated-afml} that an AEC with amalgamation categorical in a high-enough cardinal will satisfy a weak version of tameness, and unions of chains of $\lambda$-saturated models will be $\lambda$-saturated there. This gives several examples of fragmented AECs and skeletons. Note that these results have a long history (detailed for example in \cite{Vas17}), with contributions of, among others, both the first and second author.

\begin{fact}[Structure of categorical AECs with amalgamation]\LABEL{tameness-ap}
  Let $\K$ be an AEC with arbitrarily large models. Let $\mu > \LS (\K)$. Assume that $\K$ is categorical in $\mu$ and $\K_{<\mu}$ has amalgamation and no maximal models.

  \begin{enumerate}
  \item\label{tameness-ap-1} If $\mu \ge \hanf{\LS (\K)}$, then there exists $\chi \in (\LS (\K), \mu)$ such that $\Ksat_{(\chi, \mu)}$ is $\chi$-tame.
  
  \item\label{tameness-ap-2} Let $\lambda \le \mu$. If $\cf{\lambda} \ge \LS (\K)^+$, then there exists $\chi \in (\LS (\K), \lambda)$ such that $\Ksat_{(\chi, \lambda)}$ is $\chi$-tame.
  
  \item\label{tameness-ap-3} For any $\lambda \in (\LS (\K), \mu]$, $\Ksatp{\lambda}$ is an AEC with L{\"o}wenheim-Skolem-Tarski number $\lambda$. In particular, $\Ksat_{(\LS (\K), \mu]}$ is a fragmented AEC with L{\"o}wenheim-Skolem-Tarski number $\LS (\K)^+$. Moreover, $\Ksat_{(\LS (\K), \, \mu]}$ is categorical in every $\lambda \in (\LS (\K), \mu)$ and it is a skeleton of $\K_{(\LS (\K), \mu]}$.
  \item\label{tameness-ap-4} Let $\Theta \coloneqq  [\LS (\K), \mu)$. Then:

      \begin{itemize}
      \item $\Kbrim_{\Theta}$ is a fragmented very weak AEC which is $[\LS (\K), \lambda)$-continuous and categorical in $\lambda$ for every $\lambda \in \Theta$.
      \item $\Kbrim_{\Theta}$ is a skeleton of $\K_{\Theta}$.
      \item $|\Kbrim_{(\LS (\K), \mu)}| = |\Ksat_{(\LS (\K), \mu)}|$ and $\Kbrim_{(\LS (\K), \mu)}$ is a sub-abstract class of $\Ksat_{(\LS (\K), \mu)}$.
      \end{itemize}
  \end{enumerate}
\end{fact}

Regarding categoricity transfers we will use the following two known results. The first can be traced back to the presentation theorem of the first author. See \cite[9.2]{Vas17} for a full proof.

\begin{fact}[Morley's omitting type theorem for AECs]\LABEL{morley-omitting}
  Let $\K$ be an AEC with amalgamation and let $\mu > \LS (\K)$. If every model in $\K_\mu$ is $\LS (\K)^+$-saturated, then there exists $\chi < \hanf{\LS (\K)}$ such that every model in $\K_{\ge \chi}$ is $\LS (\K)^+$-saturated.
\end{fact}

The second is due to the second author. It is a consequence of an omitting type theorem of the first author (see \cite{Vas17} for a history). 

\begin{fact}[{\cite[9.8]{Vas17}}]\LABEL{categ-unbounded-fact}
  Let $\K$ be an $\LS (\K)$-tame AEC with amalgamation, and arbitrarily large models. If $\K$ is categorical in \emph{some} $\mu > \LS (\K)$, then $\K$ is categorical in a proper class of cardinals.
\end{fact}

%%%%%%%%%%%%%%%%%%%%%%%%%%%%%%%%%%%%%%%%%%%%%%
\newpage

\section{Compact abstract elementary classes}\LABEL{compact-sec}

Compact AECs are an axiomatization of the AECs ``essentially below $\kappa$'' introduced by Boney \cite[2.10]{tamelc-jsl}. They encompass classes of models of an $\Ll_{\kappa, \omega}$ sentence, $\kappa$ a strongly compact, and more generally AECs closed under $\kappa$-complete ultraproducts, $\kappa$ a strongly compact cardinal below (yes, below!) the L{\"o}wenheim-Skolem-Tarski number. The reader who does not want to process yet another abstract definition can with little loss think of:

$(*)$ $\K_{\ge \kappa}$, where $\K$ is an AEC and $\kappa > \LS (\K)$ is a strongly compact cardinal\footnote{With a little bit of care, one could adapt the definitions to work also with \emph{almost} strongly compact cardinals, see \cite{btr-almost-compact-tams}. Since there are no obvious benefits, we have not bothered.}.

Why is ``$\kappa$-compact'' chosen? Restricting ourselves as in $(*)$ is natural as we consider only models of cardinality $\geq \kappa$ and in many cases $\LS(\K) << \kappa.$ This disappears for $\kappa$-compact. An additional point is that elementary classes become a special case when we allow $\kappa = \aleph_{0}.$

\begin{defin}\LABEL{compact-def}\myindex{compact AEC}
  Let $\K$ be an AEC and let $\kappa$ be a strongly compact cardinal (we allow $\kappa = \aleph_0$). We call $\K$ \emph{$\kappa$-compact} if:

  \begin{enumerate}
      \item[(A)] $\kappa \le \LS (\K)$ and $\K_{<\LS (\K)} = \emptyset$.
      
      \item[(B)] $\K$ is closed under $\kappa$-complete ultraproducts. More precisely, if $U$ is a $\kappa$-complete ultrafilter on some index set $I$ and $\seq{M_i : i \in I}$ is a sequence in $\K$, then

      \begin{enumerate}
          \item $\prod_{i \in I} M_i / U$ is in $\K$,

          \item if $\seq{N_i : i \in I}$ is another sequence in $\K$ and $M_i \lea N_i$ for all $i \in I$, then $\prod_{i \in I} M_i / U \lea \prod_{i \in I} N_i / U$,

          \item if $i \in I \Rightarrow M_{i} = M,$ then the canonical embedding of $M$ into $M^{I} / U = \prod_{i \in I} M_{i} / U$ is a $\leq_{\K}$-embedding.
      \end{enumerate}
  \end{enumerate}

  We call $\K$ \emph{compact} if it is $\kappa$-compact for some strongly compact cardinal $\kappa$.
\end{defin}

Immediately from the definition, we have:

\begin{remark}\LABEL{compact-rmk}
  Let $\K$ be a $\kappa$-compact AEC.

\begin{enumerate}
      \item $\K$ has no maximal models (just take an ultrapower).
      
      \item For any cardinal $\lambda$, $\K_{\ge \lambda}$ is a $\kappa$-compact AEC.
      
      \item The conclusion of \cite[4.3]{tamelc-jsl} says: $\K$ is closed under ultraproducts in a strong sense (for example, not only is the ultraproduct of members of $\K$ in $\K,$ but the canonical embeddings are themselves $\K$-embedding).
      
      \item[(4)] But (concerning \ref{compact-rmk}(3)).
    
      \underline{Example}:  Assume that $\kappa$ is a compact cardinal. We define $\K = (K, \leq_{\K})$ by: 
      
      \begin{enumerate}
          \item[(a)] $\tau_{\K} = \{ < \},$
          
          \item[(b)] $K$ is the class of linear orders with initial segments isomorphic to $(\kappa, <),$
          
          \item[(c)] $\leq_{\K} = \{(M, N) \colon M \subseteq N$ are linear orders from $\K$ such that $M \neq N \Rightarrow M \ncong (\kappa, <)\}.$
      \end{enumerate}
  \end{enumerate}
  
    Now check that $\K$ satisfies the requirements of being a $\kappa$-compact AEC, provided that we omit clause (B)(c) in \ref{compact-def}; and part (3) fail. 
\end{remark}

The two main examples are:

\begin{fact}\LABEL{compact-fact} \
  \begin{enumerate}
  \item Let $\kappa$ be a strongly compact cardinal (we allow $\kappa =\aleph_0$) and let $T$ be a theory in $\Ll_{\kappa, \omega}$. Let $\Phi$ be a fragment containing $T$. Let $\K$ be the class of models of $T$, ordered by $M \lea N$ if and only if $M \lee_{\Phi} N$. Then $\K_{\ge \kappa}$ is a $\kappa$-compact AEC with $\LS (\K_{\ge \kappa}) \le |\Phi| + |\tau (\Phi)| + \kappa$.
  \item Let $\K$ be an AEC and let $\kappa > \LS (\K)$ be a strongly compact cardinal. Then $\K_{\ge \kappa}$ is $\kappa$-compact.
  \end{enumerate}
\end{fact}

The following is an important property of compact AECs. Using the terminology of \cite[3.3]{tamelc-jsl}, it should be called ``fully $(<\kappa)$-tame and short over the empty set''. We prefer the shorter ``$(<\kappa)$-short'' here.

\begin{defin}\LABEL{r19}\myindex{short}
We say that an abstract class $\K$ is \emph{$(<\kappa)$-short} if for any $M_1, M_2 \in \K$, any ordinal $\alpha$, and any $\ba_\ell \in \fct{\alpha}{M_\ell}$, $\ell = 1,2$, if $\gtp (\ba_1 \, {\rest} \, I / \emptyset; M_1) = \gtp (\ba_2 \, {\rest} \, I / \emptyset; M_2)$ for any $I \in [\alpha]^{<\kappa}$, then $\gtp (\ba_1 / \emptyset; M_1) = \gtp (\ba_2 / \emptyset; M_2)$. We say that $\K$ is \emph{$\kappa$-short} if it is $(<\kappa^+)$-short. We say that $\K$ is \emph{short} if it is $(<\kappa)$-short for some $\kappa$.
\end{defin}
\begin{remark}[{\cite[3.4]{tamelc-jsl}}]\LABEL{short-local}
  If an abstract class $\K$ is $(<\kappa)$-short, then it is $(<\kappa)$-tame (put an enumeration of $M$ as part of both $\bb_1$ and $\bb_2$).
\end{remark}

\begin{fact}[{\cite[4.5]{tamelc-jsl}}]\LABEL{compact-short}
  If $\K$ is a $\kappa$-compact AEC, then $\K$ is $(<\kappa)$-short.
\end{fact}

Finally, we note that categorical compact AECs have amalgamation:

\begin{fact}\LABEL{compact-ap}
  Let $\K$ be a compact AEC. If $\K$ is categorical in some $\mu > \LS (\K)$, then $\K$ has amalgamation.
\end{fact}

\begin{proof}
  See \cite{Sh:E102}, a revised version of\footnote{The referee pointed out that \cite{kosh362} is about $\mathbb{L}_{\kappa, \omega}$ while we work in the more general context of AECs. This is true, but the fact is that all the relevant results of this paper have immediate analogs to AECs: all that the paper uses is that the class of models of and $\mathbb{L}_{\kappa, \omega}$-sentence has the structure of an AEC. One day, somebody may write an exposition of this beautiful paper and generalize it explicitly to AECs but for now we have to save trees and move on!.} \cite{kosh362}, $\K_{<\mu}$ has amalgamation. Now as in the proof of \cite[2.3]{baldwin-boney}, we can use sufficiently complete fine ultrafilters to push amalgamation up and get that $\K$ has amalgamation.
\end{proof}

Thus we obtain that categorical compact AECs already have quite a lot of structure:

\begin{cor}\LABEL{compact-unbounded}
  Let $\K$ be a compact AEC. If $\K$ is categorical in \emph{some} $\mu > \LS (\K)$, then $\K$ has amalgamation, no maximal models, is $(<\LS (\K))$-short, and is categorical in a proper class of cardinals.
\end{cor}

\begin{proof}
  Say $\K$ is $\kappa$-compact. By Fact \ref{compact-ap}, $\K$ has amalgamation. By Remark \ref{compact-rmk}, $\K$ has no maximal models. By Fact \ref{compact-short}, $\K$ is $(<\kappa)$-short, hence by Remark \ref{short-local} it is $\LS (\K)$-tame. Now apply Fact \ref{categ-unbounded-fact}.
\end{proof}

%%%%%%%%%%%%%%%%%%%%%%%%%%%%%%%%%%%%%%%%%%%%%%
\newpage

\section{Combinatorics of abstract classes}\LABEL{combinatorics-sec}

In addition to a few technical lemmas, the main result of this section is a generalization of the ``amalgamation from categoricity and weak diamond'' theorem of the first author \cite[I.3.8]{shelahaecbook}. The proof is essentially the same, but the axioms of AECs that were not used are dropped (as in the study of nice categories in \cite[\S4]{grsh174}), and an additional parametrization in the spirit of classification theory over a predicate \cite{class-predicate-i,class-predicate-ii} is introduced: we consider the $\phi$-amalgamation property to mean essentially that we can amalgamate while fixing the set defined by the formula $\phi$.

First, we recall a generally useful folklore fact about increasing continuous chains reflecting on a club:

\begin{lem}\LABEL{club-set}
  Let $\lambda$ be a regular uncountable cardinal. Let $\seq{A_i : i \le \lambda}$, $\seq{B_i : i \le \lambda}$ be increasing continuous sequences of sets such that:

  \begin{enumerate}
  \item $A_\lambda = B_\lambda$.
  \item For all $i < \lambda$, $|A_i| + |B_i| < \lambda$.
  \end{enumerate}

  Then $C \coloneqq  \{i < \lambda \mid A_i = B_i\}$ is club.
\end{lem}
\begin{proof}
  $C$ is clearly closed by continuity of the chains. To see unboundedness, let $i < \lambda$. We build $\seq{j_k : k < \omega}$ an increasing sequence of ordinals below $\lambda$ such that $j_0 = i$, and for all $k < \omega$, $B_{j_k} \subseteq A_{j_{k + 1}} $ and $A_{j_k} \subseteq B_{j_{k + 1}}$. This is straightforward using regularity of $\lambda$, $A_\lambda = B_\lambda$, $|A_j| + |B_j| < \lambda$ for all $j$, and the fact that the chains are increasing. Now by continuity of the chains, $j \coloneqq  \sup_{k < \omega} j_k$ is in $C$, as desired.
\end{proof}

We use this to prove a result about chain of structures that are not even necessarily ordered by substructure:

\begin{defin}\LABEL{f3}\myindex{weak substructure}\index{$M \subseteq^\ast N$|see {weak substructure}}
  For $M$ and $N$ $\tau$-structures, $M \subseteq^\ast N$ ($M$ is a \emph{weak substructure of $N$}) if $|M| \subseteq |N|,$ for all relations symbols $R$ in $\tau$, $R^M \subseteq R^N,$ and for all function symbols $f$ (including those of arity $0$), $f^{M} = f^{N} \, {\rest} \, M.$
\end{defin}

\begin{lem}\LABEL{club-struct}
  Let $\lambda$ be a regular uncountable cardinal. Let $\tau$ be a vocabulary of cardinality $< \lambda$ and let $\seq{M_i : i \le \lambda}$, $\seq{N_i : i \le \lambda}$ be $\subseteq^\ast$-increasing continuous sequences of $\tau$-structures. Assume:

  \begin{enumerate}
  \item $M_\lambda = N_\lambda$.
  \item $|\tau| < \lambda$.
  \item $\|M_i\| + \|N_i\| < \lambda$ for all $i < \lambda$.
  \end{enumerate}

  Then the set $C \coloneqq  \{i < \lambda \mid M_i = N_i\}$ is club.
\end{lem}
\begin{proof}
  Expanding $M_i$ and $N_i$ if necessary, we may assume without loss of generality that there is a unary predicate $P$ in $\tau$ such that $P^{M_i} = M_i$ and $P^{N_i} = N_i$ for all $i < \lambda$. For each $R \in \tau$, the set $C_R \coloneqq  \{i < \lambda \mid R^{M_i} = R^{N_i}\}$ is club by Lemma \ref{club-set}. Now observe that $C = \bigcap_{R \in \tau} C_R$.
\end{proof}

The following technical criteria for being resolvable (Definition \ref{resolv-def}) is used in the proof of Lemma \ref{isbrim-lem}.

\begin{lem}\LABEL{resolv-technical}
  Let $\K$ be an abstract class. Assume:

  \begin{enumerate}
      \item $\LS (\K) < \infty$, $\K_{<\LS (\K)} = \emptyset$.
      
      \item For any $\lambda \in \dom (\K)$, $\K$ is $[\LS (\K), \lambda)$-continuous.
      
      \item (Coherence for models of different sizes) For $M_0, M_1, M_2 \in \K$ with $M_0 \lea M_2$, $M_1 \lea M_2$, $|M_0| \subseteq |M_1|$, \emph{and} $\|M_0\| < \|M_1\| < \|M_2\|$, we have that $M_0 \lea M_1$.
      
      \item (Smoothness for big extensions) If $\seq{M_i : i \le \delta}$ is increasing and $N \in \K$ is such that $M_i \lea N$ for all $i < \delta$ and $\|M_\delta\| < \|N\|$, then $M_\delta \lea N$.
      
      \item (Resolvability for successors) If $\lambda \ge \LS (\K)$, $M \lea N$ are such that $M \in \K_\lambda$ and $N \in \K_{\lambda^+}$, then there exists $\seq{M_i : i < \lambda^+}$ increasing continuous in $\K_\lambda$ such that $M_0 = M$ and $\bigcup_{i < \lambda^+} M_i = N$.
  \end{enumerate}

  Then $\K$ is resolvable.
\end{lem}

\begin{proof}
  Let $M_0 \lea M$ be given with $\|M_0\| < \|M\|$. Write $\lambda \coloneqq  \|M_0\|$, $\mu \coloneqq  \|M\|$. We show by induction on $\mu$ that $M$ is resolvable over $M$. If $\mu = \lambda^+$, this is clause $(5)$ of the assumption. Assume now that $\lambda^+ < \mu$. For any such $\lambda, M_{0}, M$, if $\mu$ is a successor, say $\mu_0^+$, use that $\LS (\K) \le \lambda$ to pick $M_1$ such that $M_1 \lea M$, $|M_0| \subseteq |M_1|$, and $\lambda < \|M_1\| = \mu_0$. By coherence for models of different sizes, $M_0 \lea M_1$. By the induction hypothesis, there is a resolution $\seq{M_{0, i} : i < \delta_0}$ of $M_1$ over $M_0$. By assumption (5), there is also a resolution $\seq{M_{1, i}: i < \delta_1}$ of $M$ over $M_1$. Now concatenate these two resolutions.

  Thus we can assume that $\mu$ is limit. Let $\delta \coloneqq  \cf{\mu}$ and let $\seq{A_i : i < \delta}$ be an increasing sequence of sets such that $|A_i| < \mu$ for all $i < \delta$ and $\bigcup_{i < \mu} A_i = M$. We build $\seq{M_i : i < \delta}$ increasing continuous such that for all $i < \delta$:

  \begin{enumerate}
  \item $M_0 = M$.
  \item $A_i \subseteq |M_{i + 1}|$ for all $i < \delta$.
  \item $\|M_i\| < \|M_{i + 1}\| < \|M\|$.
  \end{enumerate}

  This is possible using the coherence and smoothness assumptions. This is enough. At the end, $\bigcup_{i < \delta} M_i = M$. We can then take a resolution of $M_{i + 1}$ over $M_i$ for each $i < \delta$ (using the induction hypothesis), and concatenate all these resolutions to obtain the final desired resolution.
\end{proof}

Instead of giving the definition of the weak diamond principle, we will use the following form of the weak diamond principle (the proof is the same as in \cite[6.1]{dvsh65}; or see the appendix of \cite{proper-and-imp}).

\begin{fact}\LABEL{theta-fact}
  Let $\lambda$ be an infinite cardinal and suppose there is an infinite cardinal $\lambda_0 < \lambda$ such that $2^{\lambda_0} = 2^{<\lambda} < 2^{\lambda}$. Then ($\lambda$ is regular uncountable and) for every sequence $\seq{f_\eta \in \fct{\lambda}{\lambda} : \eta \in \fct{\lambda}{2}}$, there exists $\eta \in \fct{\lambda}{2}$ such that the set $$ S_\eta \coloneqq  \{\delta < \lambda \mid \exists \nu \in \fct{\lambda}{2} : f_\eta \, {\rest} \, \delta = f_\nu \, {\rest} \, \delta, \eta \, {\rest} \, \delta = \nu \, {\rest} \, \delta, \eta \, {\rest} \, (\delta + 1) \neq \nu \, {\rest} \, (\delta + 1)\}
  $$

  is stationary.
\end{fact}

The following notions relativize the usual amalgamation, categoricity, and universality properties to a formula $\phi.$ The goal it to relativize and generalize the following result of the first author \cite[I.3.8]{shelahaecbook}: if an AEC $\K$ with $\lambda \geq \LS(\K)$ is categorical in $\lambda,$ has a universal model in $\lambda^{+},$ and $2^{\lambda} < 2^{\lambda^{+}},$ then $\K$ has amalgamation in $\lambda.$ This will be obtained from Theorem \ref{ap-diamond-thm} by setting $\K = \K^{\ast}, \lambda_{1} = \LS(\K) = \lambda, \lambda_{2} = \LS(\K)^{+}, \phi(x) \coloneqq  x \neq x$ and $\iota = 1.$  

\begin{defin}\LABEL{phi-props-def}
  Let $\K$ be an abstract class in a vocabulary $\tau = \tau (\K)$ and let $\phi (x)$ be a quantifier-free $\Ll_{\omega, \omega}(\tau)$-formula and for part (1)-(5), (9), let 
  % 2023-02-22 10:28  \; omitted by Saharon for part (1)-(5), (9), 
  $\iota \in \{ 1, 2\}.$ 
  
\begin{enumerate}
    \item For $M, N \in \K$, we say $M$ and $N$ are \emph{$\phi$-equal} if $\psi[M] = \phi[N]$ which means $\phi(M) = \phi(N)$ and $R^M \, {\rest} \, \phi (M) = R^N \, {\rest} \, \phi (N)$ for every relation and function symbols $R$ of $\tau$. In other words, the (partial) $\tau$-structures induced by $\phi$ on $M$ and $N$ are equal (if a function of $M$ takes an element out of $\phi(M),$ we require that its $M$ version also takes the element out, but not necessarily that the functions agree).\myindex{$\phi$-equal}\index{equal|see {$\phi$-equal}}
    
    \item For $\iota \in \{ 1, 2 \}$ a \emph{$\phi$-span$_{\iota}$} is a triple $(M_0, M_1, M_2)$ such that $M_0 \lea M_\ell$, $\ell = 1,2$, $\phi(M_{1}) = \phi(M_{2})$ and $\iota = 1 \Rightarrow \phi(M_{\ell}) = \phi(M_{0})$ and $\iota = 2 \Rightarrow \phi[M_{\ell}] \neq \phi[M_{0}]$.\myindex{$\phi$-span}\index{span|see {$\phi$-span}}
    
    \item A \emph{$\phi$-amalgam}$_{\iota}$ of a $\phi$-span$_{\iota}$ $(M_0, M_1, M_2)$ is a  triple $(N, f_1, f_2)$ such that $N \in \K$ and $f_\ell : M_\ell \xrightarrow[M_0]{} N$ are such that $f_1 \, {\rest} \, \phi (M_1) = f_2  \, {\rest} \, \phi (M_2)$ and $f_{1} \, {\rest} \, M_{0} = f_{2} \, {\rest} \, M_{0}.$\myindex{$\phi$-amalgam}\index{amalgam|see {$\phi$-amalgam}}
    
    \item We say that $M$ is a \emph{$\phi$-amalgamation$_{\iota}$ base (in $\K$)} if $M \in \K$ and every $\phi$-span$_{\iota}$ $(M, M_1, M_2)$ has a $\phi$-amalgam$_{\iota}$.\index{$\phi$-amalgamation base}\myindex{amalgamation base|see {$\phi$-amalgamation base}}
    
    \item We say that $\K$ has the \emph{$\phi$-amalgamation$_{\iota}$ property} if every $M \in \K$ is a $\phi$-amalgamation$_{\iota}$ base.\myindex{$\phi$-amalgamation property}\index{amalgamation property|see {$\phi$-amalgamation property}}
    
      \item We say that $N \in \K_{\lambda}$ is \emph{$\phi$-universal}$_{\iota}$ if whenever $M \in \K_{\lambda}$ and $f: \phi[M] \cong \phi[N]$, there exists $g: M \rightarrow N$ which extends $f$.\myindex{$\phi$-universal}\index{universal|see {$\phi$-universal}}
      
      \item We say that $\K$ has \emph{$\phi$-uniqueness}$_{\iota}$ if whenever $M_0, M_1, M_2 \in \K$ are such that $M_0 \lta M_\ell$, $\ell = 1,2$, there is $f: \phi[M_1] \cong_{\phi(M_0)} \phi[M_2]$ and $f \, {\rest} \, M_{0} = \id$.\myindex{$\phi$-uniqueness}
      
      \item We say that $\K_{\lambda}$ is \emph{$\phi$-categorical}$_{\iota}$ if whenever $M, N \in \K_{\lambda}$, we have that $\phi [M] \cong \phi [N]$.\myindex{$\phi$-categorical}\index{categorical|see {$\phi$-categorical}}
      
      \item $\K$ is $([\lambda_{1}, \lambda_{2}), \phi, \delta)$-continuous \underline{when} if $\langle M_{\alpha}: \alpha < \delta \rangle$ is $\leq_{\K_{[\lambda_{1}, \lambda_{2})}}$-increasing continuous \underline{then} $M_{\delta} = \bigcup_{\alpha} M_{\alpha} \in \K.$ Omitting $\delta$ means ``for every limit ordinal $\delta$''. 
  \end{enumerate}
\end{defin}

\begin{remark}\LABEL{f6}\
  \begin{enumerate}
      \item Letting $\iota = 1$ and setting $\phi (x)$ to be $x \neq x$, we get back the usual definitions of the amalgamation property and of a universal model.
      
      \item Setting $\phi (x)$ to be $x = x$ in Definition \ref{phi-props-def}(8) we get back the usual definition of categoricity.
      \end{enumerate}
\end{remark}

\begin{thm}\LABEL{ap-diamond-thm}
  Let $\K$ and $\K^\ast$ be abstract classes such that $|\K| = |\K^\ast|$ and $M \lea N$ implies $M \leap{\K^\ast} N$. Let $\phi$ be a quantifier-free $\Ll_{\omega, \omega} (\tau (\K))$-formula, $\iota \in \{1, 2\}$ and let $\Theta = [\lambda_1, \lambda_2]$ be an interval of cardinals. Assume $2^{\lambda_1} = 2^{<\lambda_2} < 2^{\lambda_2}$.
  
  1) If $\iota = 1$ and

  \begin{enumerate}
      \item[(A)] $\K$ is $[\lambda_1, \lambda_2)$-continuous.
      
      \item[(B)] There is a $\phi$-universal$_{\iota}$ model in $\K_{\lambda_2}^\ast$.
      
      \item[(C)] $\K_{\lambda_2}$ is $\phi$-categorical$_{\iota}$.
      
      \item[(D)] For any $\lambda \in [\lambda_1, \lambda_2)$:
      
        \begin{enumerate}
            \item $\K_\lambda$ has $\phi$-uniqueness$_{\iota}$.
            
            \item\label{phi-ap-assumption} For any $\phi$-span$_{\iota}$ $\bM = (M_0, M_1, M_2)$ in $\K_\lambda$, if $\bM$ has a $\phi$-amalgam$_{\iota}$ in $\K^\ast$, then $\bM$ has a $\phi$-amalgam$_{\iota}$ in $\K_\lambda$.
        \end{enumerate}
  \end{enumerate}
  
  \underline{Then} for any $M \in \K_{[\lambda_1, \lambda_2)}$, there exists $N \in \K_{[\lambda_1, \lambda_2)}$ such that $M \lea N$ and $N$ is a $\phi$-amalgamation$_{\iota}$ base in $\K_{\|N\|}$.
  
  2) If $\iota = 2,$ we get the conclusion of part (1) if (A), (B), (C) there holds and
  
  \begin{enumerate}
      \item[(D)$_{2}$] if $M_{\ast} \in \K_{[\lambda_{1}, \lambda_{2})}$ there is $\bar{N}, N$ such that: 
      
      \begin{enumerate}
          \item[(a)] $N \in \{ \phi(M) \colon M \in K_{\lambda_{2}} \},$
          
          \item[(b)] $\bar{N} = \langle N_{i} \colon i < \lambda_{2} \rangle$ is a $\subsetneq$-increasing sequence of $\tau(\phi(M))$-models, 
          
          \item[(c)] $\bigcup_{i} N_{i} = N $ and $i < \lambda_{2} \Rightarrow \Vert N_{i} \Vert < \lambda_{2},$
          
          \item[(d)] $N_{0} = \phi(M_{\ast}),$
          
          \item[(e)] if $i < \lambda_{2}, M_{0} \in \K_{[\lambda_{1}, \lambda_{2})}, \phi(M_{0}) = N_{i}$ and $M_{0} \cap N = N_{i},$ then there are $M_{1}, M_{2}$ such that $(M_{0}, M_{1}, M_{2})$ is $\phi$-span$_{\iota}$ and $\phi(M_{\ell}) = N_{i +1}$ for $\ell = 1, 2.$ 
      \end{enumerate}
  \end{enumerate}
\end{thm}

\begin{proof}\
  
  1) Suppose $M$ is a counter-example. We shall build an increasing continuous tree $\seq{M_\eta : \eta \in \fct{\le \lambda_2}{2}}$ such that for all $\eta \in \fct{\leq \lambda_2}{2}$:

  \begin{enumerate}
      \item[$\bullet_{1}$] $M_{\seq{}} = M$.
      
      \item[$\bullet_{2}$] $M_\eta \in \K_{\lambda + \rm{lg}(\eta)}$, where $\lambda = \Vert M \Vert.$
      
      \item[$\bullet_{3}$]\LABEL{ord-req} $|M_\eta| \subseteq \lambda_2$.
      
      \item[$\bullet_{4}$]\LABEL{phi-req} if $\iota = 1,$ then $\phi (M_\eta) = \phi (M_\nu)$ for all $\nu \in \fct{\leq \ell (\eta)}{2},$ if $\iota = 2$ then $\phi(M_{\eta^{\smallfrown} \langle 0 \rangle}) = \phi(M_{\eta^{\smallfrown} \langle 1 \rangle}).$ 
      
      \item[$\bullet_{5}$]\LABEL{ap-fail-req} if $\iota = 2,$ then in $\K_{\lambda_1 + \vert \lg (\eta) \vert}$, $(M_\eta, M_{\eta \smallfrown 0}, M_{\eta \smallfrown 1})$ is a $\phi$-span$_{\iota}$ that has no $\phi$-amalgam$_{\iota}$.
      
      \item[$\bullet_{6}$] $\langle M_{\eta \, {\rest} \, \varepsilon}: \varepsilon \leq \rm{lg}(\eta) \rangle$ is $\leq_{\K}$-increasing continuous.
  \end{enumerate}

    We choose $M_{\eta}$ for $\eta \in {}^{\alpha} 2$  by induction on $\alpha \leq \lambda_{2}:$
    
    \underline{Case 1}: $\alpha = 0.$ Let $M_{\langle \rangle} = M. $
    
    \underline{Case 2}: $\alpha = \beta + 1$ and $\eta \in {}^{\beta}2, M_{\eta}$ cannot be a $\phi$-amalgamation$_{\iota}$ base (by the choice of $M = M_{\langle \rangle}$) so by Definition \ref{phi-props-def}(4) there is a $\phi$-span$_{\iota}$ $(M_{\eta}, M_{\eta ^{\smallfrown} 0}, M_{\eta^{\smallfrown} 1})$ which has no $\phi$-amalgamation  (for $\K$) and $\Vert M_{\eta ^{\smallfrown} \ell} \Vert = \Vert M_{\eta} \Vert = \lambda + \vert \rm{lg}(\eta) \vert.$
    
    It follows that $(M_{\eta}, M_{\eta^{\smallfrown} 0}, M_{\eta^{\smallfrown} 1})$ has no $\phi$-amalgamation$_{\iota}$ in $\K^{*},$ also $\phi(M_{\eta^{\smallfrown} \ell}) = \phi(M_{\eta}) = \phi(M_{\langle \rangle})$ by Definition \ref{phi-props-def}(3). Renaming without loss of generality $\vert M_{\eta^{\smallfrown} 1} \vert$ is an ordinal so  $< \lambda_{2}. $
    
    \underline{Case 3}: $\alpha$ limit ordinal.
    
    For $\eta \in {}^{\alpha}2$ let $M_{\eta} = \bigcup \{ M_{\eta \, {\rest} \, \beta}: \beta < \alpha \},$ note $M_{\eta} \in K_{\lambda + \vert \rm{lg}(\eta) \vert}$ by  clause (1) of the assumption, and also $\phi(M_{\eta}) = \bigcup \{ M_{\eta \, {\rest} \, \beta}): \beta < \alpha \} = \phi(M_{\langle \rangle})$ by the induction hypothesis. So we have carried the induction.  
    
    %This is possible: use the assumption that the conclusion fails to take care of requirement (\ref{ap-fail-req}), and use $\phi$-uniqueness and some renaming to take care of requirements (\ref{ord-req}) and (\ref{phi-req}). This is enough: it is easy to check that we must have that $M_{\eta} \lta M_{\eta \smallfrown \ell}$ for all $\ell \in 2$ and $\eta \in \fct{<\lambda_2}{2}$. Thus $M_{\eta} \in \K_{\lambda_2}$ for each $\eta \in \fct{\lambda_2}{2}$.

  Let $A$ be the $\tau (\K)$-structure induced by $\phi (M_\eta)$ for some (or any, see requirement $\bullet_{4})$ $\eta \in \fct{\lambda_2}{2}$. Let $\sea \in \K_{\lambda_2}^\ast$ be $\phi$-universal in $\K_{\lambda_2}^\ast$. By $\phi$-categoricity (that is, clause (c) of the claim), pick $h: A \cong \phi (N)$. For each $\eta \in \fct{\lambda_2}{2}$, pick $g_\eta : M_\eta \rightarrow \sea$ a $\K^\ast$-embedding extending $h$. We have that $|M_\eta| \subseteq \lambda_2$, so extend each $g_\eta$ arbitrarily to $f_\eta \in \fct{\lambda_2}{\lambda_2}$.

  By Fact \ref{theta-fact}, there exists $\eta \in \fct{\lambda_2}{2}$ such that the set $S_\eta$ described there is stationary. Let $C_\eta \coloneqq  \{\delta < \lambda_2 \mid |M_{\eta \, {\rest} \, \delta}| \subseteq \delta\}$. This is a club (by Lemma \ref{club-set} applied to $\seq{M_{\eta \, {\rest} \, \delta} : \delta < \lambda_2}$, $\seq{M_{\eta \, {\rest} \, \delta} \cap \delta : \delta < \lambda_2}$), so let $\delta \in S_\eta \cap C_\eta$. We then must have that there is $\nu \in \fct{\lambda_2}{2}$ such that $\eta_0 \coloneqq  \eta \, {\rest} \, \delta = \nu \, {\rest} \, \delta$, $\eta \, {\rest} \, (\delta + 1) \neq \nu \, {\rest} \, (\delta + 1)$, and (in particular), $g_\eta$ and $g_\nu$ agree on $M_{\eta_0}$. Moreover, $g_\eta$ and $g_\nu$ also extend $h$, so must in particular agree on $\phi (M_{\eta \, {\rest} \, (\delta + 1)})$. This means that $(N, g_\eta \, {\rest} \, M_{\eta \, {\rest} \, (\delta + 1)}, g_\nu \, {\rest} \, M_{\nu \, {\rest} \, (\delta + 1)})$ is a $\phi$-amalgam of $(M_\eta, M_{\eta \, {\rest} \, (\delta + 1)}, M_{\nu \, {\rest} \, (\delta + 1)})$ in $\K^\ast$. By assumption (D)(b), this means there is such a $\phi$-amalgam in $\K_{\lambda_1 +  \vert \delta \vert}$, a contradiction to requirement $\bullet_{5}$. 
  
  2) Similarly. 
\end{proof}

% \begin{remark}
%   Setting $\K = \K^\ast$, $\lambda_1 = \LS (\K) = \lambda$, $\lambda_2 = \LS (\K)^+$, $\phi (x) \coloneqq  x \neq x$, we get back the following result of the first author \cite[I.3.8]{shelahaecbook}: if an AEC $\K$ with $\lambda \coloneqq  \LS (\K)$ is categorical in $\lambda$, has a universal model in $\lambda^+$, and $2^{\lambda} < 2^{\lambda^+}$, then $\K$ has amalgamation in $\lambda$.
% \end{remark}

%%%%%%%%%%%%%%%%%%%%%%%%%%%%%%%%%%%%%%%%%%%%%%
\newpage 

\section{Good frames}\LABEL{good-frames-sec}

Good frames are a local notion of ``bare bone'' superstability, introduced by the first author \cite[Chapter II]{shelahaecbook}. Essentially, an AEC $\K$ has a \emph{good $\lambda$-frame} if it looks superstable in $\lambda$ in the sense that $\K_\lambda$ has some reasonable structural properties (like amalgamation), and there is a forking-like notion for types of singletons over models. 

In this section, we give the definition we will use and state some fundamental results (most of them known or folklore) about good frames.

The definition we give here allows good frames in several cardinals (as in \cite[2.21]{ss-tame-jsl}) but using fragmented AECs (so also allowing ``shrinking frames'', \cite[Appendix A]{Vas17}). The frames in this paper will always be type-full (i.e.\ all nonalgebraic types will be basic), so we will drop the adjective and ignore basic types in the definition.

\begin{defin}\LABEL{good-frame-def}\index{good frame}\myindex{$\s$|see {good frame}} 
  A \emph{(type-full) good frame} (or a semi good frame) is a pair $\s = (\K, F) = (\K_{\s}, F_{\s})$, where:

  \begin{enumerate}
  \item $\K$ is a fragmented AEC (or a semi fragmented AEC) such that:
    \begin{enumerate}
        \item $\K \neq \emptyset$.
        
        \item $\K_{< \LS (\K)} = \emptyset$.
        
        \item\label{frame-aec-3} $\K$ has amalgamation, joint embedding, and no maximal models.
        
        \item\label{frame-aec-4} For every $M \in \K$, there is $N \in \K$ such that $N$ is universal over $M$ and $\|N\| = \|M\|$.
    \end{enumerate}

  \item\myindex{does not fork over}\myindex{forking} $F$ is a binary relation taking as input pairs $(p, M)$, where $p$ is an orbital type and $M$ is a model in $\K$. We write \emph{$p$ does not fork over $M$} (or \emph{$p$ does not $\s$-fork over $M$}) instead of $F (p, M)$ and require that $F$ satisfies the following:

    \begin{enumerate}
    \item If $p$ does not fork over $M$, then $p \in \gS (N)$ for some $N \gea M$.
    \item Invariance: if $f: N \cong N'$ and $p \in \gS (N)$ does not fork over $M$, then $f (p)$ does not fork over $f[M]$.
    \item Monotonicity: if $p \in \gS (N)$ does not fork over $M$ and $M \lea M' \lea N$, then $p \, {\rest} \, M'$ does not fork over $M$ and $p$ does not fork over $M'$.
    \item Disjointness: if $p \in \gS (N)$ does not fork over $M$, then $p \, {\rest} \, M$ is algebraic if and only if $p$ is algebraic.
    \item\label{frame-5} Extension: for any $M \lea N$ and any $p \in \gS (M)$ there exists $q \in \gS (N)$ such that $q$ extends $p$ and $q$ does not fork over $M$.
    \item\label{frame-6} Uniqueness: for any $M \lea N$ and any $p, q \in \gS (N)$, if both $p$ and $q$ do not fork over $M$ and $p \, {\rest} \, M = q \, {\rest} \, M$, then $p = q$.
    
    \item Local character: If $\seq{M_i : i \le \delta}$ is increasing continuous in $\K$ and $p \in \gS (M_\delta)$, then there exists $i < \delta$ such that $p$ does not fork over $M_i$.
    \item\label{frame-8} Symmetry: If $p = \gtp (a / N; N')$ does not fork over $M$ and $b \in |N|$, then there exists $M', N'' \in \K$ such that $N' \lea N''$, $M \lea M'$, $a \in M'$, and $\gtp (b / M'; N'')$ does not fork over $M$.
    \end{enumerate}
  \end{enumerate}
  \myindex{$\K_{\s}$}\myindex{categorical good frame}\myindex{domain of a good frame}
  We write $\K_{\s}$ for the class of the frame $\s$. We say that $\s$ is \emph{on $\K^\ast$} if $\K_{\s} = \K^\ast$. The \emph{domain} of a good frame $\s$ is the domain of $\K_{\s}$ (see Definition \ref{dom-def}). For $\lambda \in \dom (\s)$, we say that $\s$ is \emph{categorical in $\lambda$} if $\K_{\s}$ is categorical in $\lambda$. We say that $\s$ is \emph{categorical} if it is categorical in \emph{all} $\lambda \in \dom (\s)$. We also define restrictions to smaller classes of models such as $\s_{\lambda}$ in the natural way. We say that $\s$ is a \emph{good $\lambda$-frame} if $\s = \s_\lambda$. 
\end{defin}

We will use the following construction of a good frame:

\begin{fact}\LABEL{tame-frame-constr}
  Let $\K$ be an AEC with arbitrarily large models. Let $\mu > \LS (\K)$ and let $\theta \ge \mu$. Assume that $\K_{<\theta}$ has amalgamation and no maximal models. Assume further that $\K$ is categorical in $\mu$ and $\Ksat_{(\LS (\K), \theta)}$ is $\LS (\K)$-tame. Then there is a (categorical) good frame on $\Ksat_{(\LS (\K), \theta)}$.
\end{fact}

\begin{proof}
  By \cite[5.7(1)]{categ-saturated-afml}, $\K$ is $\LS (\K)$-superstable and has $\LS (\K)$-symmetry. The result now follows from the proof of \cite[A.3]{Vas17}.
\end{proof}

\begin{remark}\LABEL{good-frame-local-rmk}
  Combining local character, transitivity, and uniqueness, we have that whenever $\s$ is a good frame, then $\K_{\s}$ is $\min (\dom (\s))$-tame.
\end{remark}

Regarding disjointness, it follows from the other properties if the frame is categorical: 

\begin{lem}\LABEL{disj-lem}
  If $\s$ satisfies all the properties of good frames, except perhaps disjointness, and $\s$ is categorical, then $\s$ satisfies disjointness as well.
\end{lem}
\begin{proof}
  By the conjugation property \cite[III.1.21]{shelahaecbook} (whose proof never uses disjointness).
\end{proof}

We now state and prove canonicity of the framework. First, frames with the same restriction in their low cardinals are the same:

\begin{lem}\LABEL{canon-lem}
  Let $\s$ and $\ts$ be good frames with $\K_{\s} = \K_{\ts}$. Let $\lambda \coloneqq  \min (\dom (\s))$. If $\s_\lambda = \ts_\lambda$, then $\s = \ts$.
\end{lem}
\begin{proof}
  Let $\K \coloneqq  \K_{\s} = \K_{\ts}$. Let $M \lea N$ be in $\K$ and let $p \in \gS (N)$. Assume that $p$ does not $\s$-fork over $M$. We show that $p$ does not $\ts$-fork over $M$, and the converse is symmetric. First, by local character and transitivity there exists $M_0 \in \K_\lambda$ such that $p$ does not $\s$-fork over $M_0$. In particular (by monotonicity), for every $N_0 \in \K_\lambda$ with $M_0 \lea N_0 \lea N$, $p \, {\rest} \, N_0$ does not $\s$-fork over $M_0$. Since $\s_\lambda = \ts_\lambda$, $p \, {\rest} \, N_0$ does not $\ts$-fork over $M_0$. Now pick $N_0' \in \K_\lambda$ such that $p$ does not $\ts$-fork over $N_0'$ and (by monotonicity), enlarge it so that $M_0 \lea N_0'$. Then by transitivity, $p$ does not $\ts$-fork over $M_0$, hence (by monotonicity) over $M$, as desired.
\end{proof}
\begin{remark}\LABEL{canon-lem-rmk}
  Disjointness and symmetry are not used in the proof. Regarding local character, we only use that any type does not fork over a model of size $\lambda$.
\end{remark}

\begin{fact}[Canonicity of categorical good frames]\LABEL{good-frame-canon}
  Let $\s$ and $\ts$ be categorical good frames with $\dom (\s) = \dom (\ts)$. Let $\lambda \coloneqq  \min (\dom (\s))$. If $\left(\K_{\s}\right)_{\lambda} = \left(\K_{\ts}\right)_\lambda$, then $\s = \ts$.
\end{fact}
\begin{proof}
  By \cite[9.7]{indep-aec-apal}, $\s_\lambda = \ts_\lambda$. By canonicity of categorical fragmented AECs (Theorem \ref{fragment-canon}), $\K_{\s} = \K_{\ts}$. By Lemma \ref{canon-lem}, $\s = \ts$.
\end{proof}

Categoricity of a good frame may seem a strong assumption. However, we have:

\begin{fact}\LABEL{categ-subframe}
  Let $\s$ be a good frame and let $\lambda \coloneqq  \min (\dom (\s))$. If $\s$ is categorical in $\lambda$, then there is a categorical good frame $\ts$ with $\dom (\ts) = \dom (\s)$ and $\left(\K_{\ts}\right)_\lambda = \left(\K_{\s}\right)_\lambda$.
\end{fact}

\begin{proof}
  By the proof of \cite[A.2]{Vas17}: that is, just restrict to saturated models (note that this gives fragmented AEC, not necessarily an AEC, but this is allowed by the above definition of a good frame) 
\end{proof}

We recall the following result about categoricity transfers in good frames:

\begin{discussion}\LABEL{frame-categ-succ}\

    (1) The following assertion appears as a claim in earlier versions, but the referee point out that saying \cite[Theorem A.9]{Vas17} is not convincing as there $\K$ is not fragmented. As we prefer not to resolve it, we change the proof of \ref{categ-excellence-lem}, the only place it was used.

    (2) The assertion was: 
    
    Assume that $\s$ is a good frame with fragmented AEC $\K$. Let $\Theta \coloneqq  \dom (\K)$, let $\lambda \coloneqq  \min (\Theta)$, and let $\mu \in \Theta$. Assume that $\K$ is categorical in $\lambda$. 
    
    Let $\K^\ast$ be the AEC generated by $\K_{\mu}$. If $\K^\ast$ is categorical in $\mu^+$, it follows that:

    \begin{enumerate}
        \item[(A)] $\K$ is $\theta$-continuous for every $\theta \in \Theta$.
        
        \item[(B)] $\K$ is categorical in every $\mu' \in \Theta$.
    \end{enumerate}

    (3) The asserted proof was: we argue by Fact \ref{categ-subframe}, we might as well assume that $\s$ is categorical. We are then in the setup of \cite[Theorem A.9]{Vas17}, which gives the result we want.
\end{discussion}

We will also use the following upward frame transfer:

\begin{defin}[{\cite[II.2.4,II.2.5]{shelahaecbook}}]\LABEL{up-def}\myindex{$\s^{up}$}
  Let $\s$ be a good $\lambda$-frame and let $\K$ be the AEC generated by $\K_\s$. We let $\s^{\rm{up}} \coloneqq  (\K, F)$, where $F$ is the following binary relation on pairs $(p, M)$, with $p$ is an orbital type and $M$ is a model in $\K$: $F (p, M)$ if and only if $p \in \gS (N)$ for some $M \lea N$, and there exists $M_0 \in \K_{\s}$ such that $M_0 \lea M$ and for all $N_0 \in \K_{\s}$ with $M_0 \lea N_0 \lea N$, $p \, {\rest} \, N_0$ does not $\s$-fork over $M_0$.
\end{defin}

\begin{fact}\LABEL{frame-ext}
  Let $\s$ be a good $\lambda$-frame.

  \begin{enumerate}
      \item\cite[II.2.11]{shelahaecbook} $\s^{\rm{up}}$ satisfies all the axioms from the definition of a type-full good frame, except for (\ref{frame-aec-3}),(\ref{frame-aec-4}), (\ref{frame-5}), (\ref{frame-6}), and (\ref{frame-8}) in Definition \ref{good-frame-def}.
      
      \item\cite[6.9]{tame-frames-revisited-jsl} If $\K_{\s^{\rm{up}}}$ is $\lambda$-tame and has amalgamation, then $\s^{\rm{up}}$ is a good frame.
  \end{enumerate}
\end{fact}

%%%%%%%%%%%%%%%%%%%%%%%%%%%%%%%%%%%%%%%%%%%%%%
\newpage

\section{Two-dimensional independence notions}\LABEL{twodim-sec}

While good frames describe a forking-like relation for types of \emph{singletons} over models, two-dimensional independence notions describe a forking-like relations for types of models over models. At that point, it seems more convenient to think of such a relation as a $4$-ary relation on squares of models\footnote{The referee asks: why does 4-ary correspond to $2$-dimensional? For the same reason that a square is two-dimensional despite having four vertices: $4 = 2^{2}.$ More explicitly, $\nf$ is a $4$-place relation, and $\nf(M_{0}, M_{1}, M_{2}, M_{3})$ it implies $M_{0} \leq_{\K} M_{\ell} \leq_{\K} M_{3}$ (and $M_{1} \cap M_{2} = M_{0}$). But we can express this as ``$M_{1}, M_{2}$ are independent over $M_{0}$ inside $M_{3}$''saying it in this form, it appear as a $2$-place relation. Alternatively we can write ``$\bar{N} = \langle N_{u}: u \in \mathcal{P}(n) \rangle$ is independent'' where $N_{\emptyset} = M_{0}, N_{\{ 0 \}} = M_{1}, N_{ \{ 1 \} } = M_{2}, N_{\{ 0, 1 \}} = M_{3};$ in this form the $2$ appear in $\mathcal{P}(2)$ and later generalized to $\langle N_{u}: u \in \mathcal{P}(n) \rangle$ and is called $n$-dimensional and is used to analyze models in $\lambda^{+n}$.} . Squares in nonforking amalgamation form the simplest nontrivial example of an independent system (a concept defined in the next section).

This section studies abstract frameworks for two-dimensional independence, while the next section will look at higher dimensions. We state the important properties of two-dimensional independence relations. We look also at their canonicity (Theorem \ref{twodim-canon}). Two key questions are when a two-dimensional independence relation can be built from a good frame, and when a two-dimensional independence relation for models of size $\lambda$ implies the existence of a good $\lambda^+$-frame (these themes are present already in \cite[II, III]{shelahaecbook}). Frames that are well-behaved in these respects are called \emph{extendible} (Definition \ref{extendible-def}). A frame that can be extended $\omega$-many steps is called \emph{$(<\omega)$-extendible}. These correspond (but are slightly more convenient to work with than) the $\omega$-successful good frames in \cite[III]{shelahaecbook}. We give conditions under which such frames exist, both in the compact (Fact \ref{compact-very-good}) and non-compact (Fact \ref{extendible-wd}) cases. For the latter result, the weak diamond (a weakening of the generalized continuum hypothesis) is assumed.

\begin{defin}\LABEL{relation-def}\myindex{two-dimensional independence relation}\index{two-dimensional independence notion|see {two-dimensional independence relation}}\index{$\nf$|see {two-dimensional independence notion}}
  Let $\K$ be an abstract class. A \emph{two-dimensional independence relation (or notion) on $\K$} is a $4$-ary relation $\nf$ on $\K$ satisfying:

  \begin{enumerate}
  \item[(A)] $\nf (M_0, M_1, M_2, M_3)$ implies $M_0 \lea M_\ell \lea M_3$ for $\ell = 1,2$. We may write $\nfs{M_0}{M_1}{M_2}{M_3}$ instead of $\nf (M_0, M_1, M_2, M_3)$.
  
  \item[(B)] If $M_0 \lea M_\ell \lea M_3$, $\ell = 1,2$ and $f: M_3 \rightarrow M_3'$ is a $\K$-embedding, then $\nfs{M_0}{M_1}{M_2}{M_3}$ if and only if $\nfs{f[M_0]}{f[M_1]}{f[M_2]}{M_3'}$.
  
  \item[(C)] Monotonicity: if $\nfs{M_0}{M_1}{M_2}{M_3}$ and $M_0 \lea M_1' \lea M_1$, then $\nfs{M_0}{M_1'}{M_2}{M_3}$.
  
  \item[(D)] Disjointness: if $\nfs{M_0}{M_1}{M_2}{M_3}$, then $M_1 \cap M_2 = M_0$.
  
  \item[(E)] Symmetry: if $\nfs{M_0}{M_1}{M_2}{M_3}$, then $\nfs{M_0}{M_2}{M_1}{M_3}$.
  
  \item[(F)] Transitivity: if $\nfs{M_0}{M_1}{M_2}{M_3}$ and $\nfs{M_2}{M_3}{M_4}{M_5}$, then $\nfs{M_0}{M_1}{M_4}{M_5}$.
  
  \item[(G)] Extension: whenever $M_0 \lea M_\ell$, $\ell = 1,2$, there exists $M_3 \in \K$ and $f_\ell : M_\ell \xrightarrow[M_0]{} M_3$ such that $\nfs{M_0}{f_1[M_1]}{f_2[M_2]}{M_3}$.
  
  \item[(H)] Uniqueness: whenever $\nfs{M_0^\ell}{M_1^\ell}{M_2^{\ell}}{M_3^{\ell}}$ for $\ell = 0,1$ and $f_k : M_k^0 \cong M_k^1$, $k < 3$ are such that $f_0 \subseteq f_1$, $f_0 \subseteq f_2$, then there exists $M_3^2 \in \K$ with $M_3^1 \lea M_3^2$ and and $f_3 : M_3^0 \rightarrow M_3^2$ such that $f_k \subseteq f_3$ for all $k < 3$.
  \end{enumerate}
\end{defin}

\begin{defin}\LABEL{s4}\myindex{respects $\s$}
  Let $\nf$ be a two-dimensional independence notion on $\K$. If $\s$ is a good frame on $\K$, we say that $\nf$ \emph{respects $\s$} if whenever $\nfs{M_0}{M_1}{M_2}{M_3}$ and $a \in |M_1|$, we have that $\gtp (a / M_2; M_3)$ does not $\s$-fork over $M_0$.
\end{defin}

As in Lemma \ref{disj-lem}, note that disjointness is not really needed in some cases:

\begin{lem}\LABEL{disj-lem-2}
  Assume that $\nf$ satisfies all the properties of a two-dimensional independence notion on $\K$, except perhaps for disjointness. Assume that $\nf$ respects a good frame $\s$ on $\K$. Then $\nf$ satisfies disjointness.
\end{lem}
\begin{proof}
  Because $\s$ satisfies disjointness.
\end{proof}

It is sometimes useful to extend $\nf$ to take sets on the left and right hand sides. This is the content of the next definition, versions of which were already considered by both the first author \cite[II.6.35]{shelahaecbook} and the second author (in joint work with Boney, Grossberg, and Kolesnikov) \cite{bgkv-apal}. See also the recent work of  Lieberman, Rosick{\'y}, and the second author \cite[8.2]{indep-categ-v3}.

\begin{defin}\LABEL{s7}\myindex{does not fork over}\myindex{forking}
  Let $\nf$ be a two-dimensional independence relation on $\K$.

  \begin{enumerate}
  \item\myindex{$\nfm$} Define a $4$-ary relation $\nfm$ as follows: $\nfcl{M}{A}{B}{N}$ holds if and only if $M \lea N$, $A, B \subseteq |N|$, and there exists $M_1, M_2, M_3 \in \K$ such that $N \lea M_3$, $A \subseteq |M_1|$, $B \subseteq |M_2|$, and $\nfs{M}{M_1}{M_2}{M_3}$.

  \item For $p$ an orbital type, we say that $p$ \emph{does not fork over $M$} if $M \in \K$, and there exists $b, A, M_3$ such that $p = \gtp (b / A; M_3)$ and $\nfcl{M}{b}{A}{M_3}$.
  \item Define a binary relation $F = F(\nf)$ taking as input pairs $(p, M)$, where $p$ is an orbital type and $M$ is a model in $\K$ as follows: $p F M$ if $p \in \gS (N)$ for some $N \gea M$ and $p$ does not fork over $M$. We let $\s (\nf) \coloneqq  (\K, F (\nf))$.\myindex{$\s (\nf)$}
  \end{enumerate}
\end{defin}

Properties of $\nf$ generalize to $\nfm$ as follows:

\begin{fact}[{\cite[8.4,8.5]{indep-categ-v3}}]\LABEL{nfm-props}
  Let $\nf$ be a two-dimensional independence relation on $\K$.
  
  \begin{enumerate}
  \item Let $M_0 \lea M_\ell \lea M_3$ for $\ell = 1, 2$. Then $\nfs{M_0}{M_1}{M_2}{M_3}$ if and only if $\nfcl{M_0}{M_1}{M_2}{M_3}$.
  
  \item\label{nfcl-k-embed} (Preservation under $\K$-embeddings) Given $M_0 \lea M_3$, $A, B \subseteq \vert M_3 \vert$, and $f: M_3 \rightarrow N_3$, we have that $\nfcl{M_0}{A}{B}{M_3}$ if and only if $\nfcl{f[M_0]}{f[A]}{f[B]}{N_3}$.
  
  \item (Monotonicity) If $\nfcl{M_0}{A}{B}{M_3}$ and $A_0 \subseteq A$, $B_0 \subseteq B$, then $\nfcl{M_0}{A_0}{B_0}{M_3}$.
  
  \item (Normality) $\nfcl{M_0}{A}{B}{M_3}$ if and only if $\nfcl{M_0}{A M_0}{B M_0}{M_3}$.
  
  \item (Base monotonicity) If $\nfcl{M_0}{A}{B}{M_3}$, $M_0 \lea M_2 \lea M_3$, and $ \vert M_{2} \vert \subseteq B$, then $\nfcl{M_2}{A}{B}{M_3}$.
  
  \item (Extension) Whenever $M \lea N$ and $p \in \gS^{<\infty} (M)$, there exists $q \in \gS^{<\infty} (N)$ extending $p$ such that $q$ does not fork over $M$.
  
  \item (Symmetry) $\nfcl{M}{A}{B}{N}$ holds if and only if $\nfcl{M}{B}{A}{N}$ holds.
  
  \item (Uniqueness) Given $p, q \in \gS^{<\infty} (B; N)$ with $M \lea N$ and $\vert M \vert \subseteq B \subseteq \vert N \vert$, if $p \, {\rest} \, M = q \, {\rest} \, M$ and $p$, $q$ do not fork over $M$, then $p = q$.
  
  \item (Transitivity) If $M_0 \lea M_2 \lea M_3$, $\nfcl{M_0}{A}{M_2}{M_3}$ and $\nfcl{M_2}{A}{B}{M_3}$, then $\nfcl{M_0}{A}{B}{M_3}$.
  \end{enumerate}
\end{fact}

In a general two-dimensional independence relation $\nf$, $\s (\nf)$ may not induce a good frame (because e.g.\ such a relation also exists in strictly stable first-order theories). We call the ones that do (and satisfy a few more convenient properties) \emph{good}:

\begin{defin}\LABEL{good-def}\myindex{good two-dimensional independence notion}
  A two-dimensional independence notion $\nf$ on $\K$ is \emph{good} if it satisfies the following properties:

  \begin{enumerate}
  \item $\s (\nf)$ is a good frame on $\K$ (in particular, $\K$ is a fragmented AEC).
  \item Long transitivity: if $\delta$ is a limit ordinal, $\seq{M_i : i \le \delta}$, $\seq{N_i : i \le \delta}$ are increasing continuous in $\K$ and $\nfs{M_i}{N_i}{M_{i + 1}}{N_{i + 1}}$ for all $i < \delta$, then $\nfs{M_0}{N_0}{M_\delta}{N_\delta}$.
  \item Local character: if $M \lea N$ and $A \subseteq |N|$, there exists $M_0 \lea N_0$ such that $\nfs{M_0}{N_0}{M}{N}$, $A \subseteq |N_0|$, and $\|N_0\| \le |A| + \LS (\K)$.
  \end{enumerate}
\end{defin}

Note that the local character property of two-dimensional independence notions is vacuous in case the relation is on $\K_\lambda$. In this case, the following replacement is useful (this is related to the definition of successful $\goodp$ in \cite[III]{shelahaecbook}, see \cite[2.14]{counterexample-frame-v4-toappear}; the ``reflects down'' terminology appears in \cite[3.7(2)]{Vas17}):

\begin{defin}\LABEL{reflects-down-def}\myindex{reflects down}
  A good two-dimensional independence notion $\nf$ on $\K$ \emph{reflects down} if for any $\lambda \in \dom (\K)$ and any two $\leq_{\K}$-increasing continuous chains $\seq{M_i : i < \lambda^+}$, $\seq{N_i : i < \lambda^+}$ in $\K_{\lambda}$ and $M_{i} \leq_{\K} N_{i}$ for $i < \lambda^{+}$, there is a club $C \subseteq \lambda^+$ such that for any $i < j$ in $C$, $\nfs{M_i}{N_i}{M_j}{N_j}$ (note that this implies $M_{i} \leq_{\K} N_{i}$). 
\end{defin}

\begin{remark}\LABEL{lc-reflects-down}
  If $\nf$ is a good two-dimensional independence notion on $\K$ and $\lambda, \lambda^+ \in \dom (\K)$, then by local character $\nf \, {\rest} \, \K_\lambda$ reflects down.
\end{remark}

We will use the following very useful fact about good two-dimensional independence notions: a union of independent limit squares is limit, in the following sense:

\begin{fact}[{\cite[II.6.29]{shelahaecbook}}]\LABEL{brimmed-union-fact}
  Let $\nf$ be a good two-dimensional independence notion on $\K$. Let $\seq{M_i : i \le \delta}$, $\seq{N_i : i \le \delta}$ be increasing continuous in $\K$ such that $\nfs{M_i}{N_i}{M_{i + 1}}{N_{i + 1}}$ for all $i < \delta$. If $N_{i + 1}$ is limit over $M_i \cup N_i$ (see Definition \ref{sat-defs}(\ref{brimmed-set-def})) for all $i < \delta$, then $N_\delta$ is limit over $M_\delta \cup N_0$.
\end{fact}

Regarding limit models, we also have that one can resolve them in a nice way:

\begin{fact}[{\cite[III.1.17]{shelahaecbook}}]\LABEL{brimmed-reflects}
  Let $\nf$ be a good two-dimensional independence notion on $\K$. Let $\lambda \in \dom (\K)$ be such that $\lambda^+ \in \dom (\K)$. Let $M, N \in \K_{\lambda^+}$ be limit such that $N$ is limit over $M$. Then there exists $\seq{M_i : i < \lambda^+}, \seq{N_i : i < \lambda^+}$ increasing continuous resolutions of $M$ and $N$ in $\K_\lambda$ such that $M_i$ is limit and $N_i$ is limit over $M_i$ for all $i < \lambda^+$.
\end{fact}

Being good and reflecting down is useful, but one sometimes wants more than the long transitivity property. This is the content of the next definition:

\begin{defin}\LABEL{very-good-twodim-def}\myindex{very good two-dimensional independence notion}\myindex{strong continuity (for a two-dimensional independence relation)}
  A two-dimensional independence notion $\nf$ on $\K$ is \emph{very good} if it is good, resolvable, reflects down and satisfies in addition \emph{strong continuity}: whenever $\seq{M_\ell^i : i \le \delta}$ are increasing continuous in $\K$, $\ell < 4$, and $\nfs{M_0^i}{M_1^i}{M_2^i}{M_3^i}$ for all $i < \delta$, then $\nfs{M_0^\delta}{M_1^\delta}{M_2^\delta}{M_3^\delta}$.
\end{defin}

We now proceed to show that very good categorical two-dimensional independence notions are canonical. This is essentially \cite{bgkv-apal}, but since the setup here is not as global as there, we use a slightly different road.

\begin{lem}\LABEL{twodim-canon-lem}
  If $\nf^1$ and $\nf^2$ are very good two-dimensional independence notions on $\K$ and $\nf^1 \, {\rest} \, \K_{\lambda} = \nf^2 \, {\rest} \, \K_{\lambda}$ for all $\lambda \in \dom (\K)$, then $\nf^1 = \nf^2$.
\end{lem}

\begin{proof}
  Assume that $\nf^1 (M_0, M_1, M_2, M_3)$. We show that $\nf^2 (M_0, M_1, M_2, M_3)$, and the converse will be symmetric. We proceed by induction on $\|M_3\|$. If $\|M_0\| = \|M_3\|$, then by assumption $\nf^2 (M_0, M_1, M_2, M_3)$. Assume now that $\|M_0\| < \|M_3\|$. Let $\delta \coloneqq  \|M_3\|$. For $\ell = 1,2,3$, build $\seq{M_\ell^i : i \le \delta}$ increasing continuous such that for all $i \le \delta$:

  \begin{enumerate}
  \item $\|M_\ell^i\| = \|M_0\| + |i|$ for $\ell = 1,2,3$.
  
  \item $\nf^1(M_{0}, M_{1}^{i}, M_{2}^{i}, M_{3}^{i}).$
  
  \item $M_\ell^\delta = M_\ell$ for $\ell = 1,2,3$.
  \end{enumerate}

  This is possible using monotonicity and the fact that $\K$ is a fragmented AEC. This is enough: by the induction hypothesis, $\nf^1 (M_0, M_1^i, M_2^i, M_3^i)$ for all $i < \delta$. By strong continuity, $\nf^2 (M_0, M_1, M_2, M_3)$, as desired.
\end{proof}

\begin{fact}\LABEL{twodim-canon-fact}
  Let $\K$ be fragmented AEC and let $\lambda \ge \LS (\K)$. Let $\nf^1$ and $\nf^2$ be two good two-dimensional independence notions on $\K_\lambda$ which reflect down. If $\K$ is categorical in $\lambda$, then $\nf^1 = \nf^2$.
\end{fact}
\begin{proof}
  Let $\s^\ell \coloneqq  \s (\nf^\ell)$. By canonicity of categorical good frames (Fact \ref{good-frame-canon}), $\s^1 = \s^2$, so write $\s \coloneqq  \s^1$. By \cite[3.11]{Vas17}, $\s$ has the existence property for uniqueness triples (see \cite[II.5.3(3)]{shelahaecbook}). By \cite[II.6.3(3)]{shelahaecbook}, $\nf^1 = \nf^2$.
\end{proof}

\begin{thm}[Canonicity]\LABEL{twodim-canon}
  If $\nf^1$ and $\nf^2$ are two very good two-dimensional independence notions on $\K$ and $\K$ is categorical in every $\lambda \in \dom (\K)$, then $\nf^1 = \nf^2$.
\end{thm}
\begin{proof}
  By Fact \ref{twodim-canon-fact}, $\nf^1 \, {\rest} \, \K_\lambda = \nf^2 \, {\rest} \, \K_\lambda$ for every $\lambda \in \dom (\K)$. Now apply Lemma \ref{twodim-canon-lem}.
\end{proof}

The next topic is when (very) good two-dimensional independence notions exist. It is natural to construct them from good frames, since we already have sufficient conditions for the existence of good frames (Fact \ref{tame-frame-constr}). A good frame that can be extended to a good two-dimensional independence relation which reflects down will be called \emph{extendible}. See \ref{splus-fact} for another justification of the name. 

% (this is essentially equivalent to ``successful $\goodp$'' in \cite[III.1]{shelahaecbook}, but has a shorter definition):

\begin{defin}\LABEL{extendible-def}\myindex{extendible}
  We say a good frame $\s$ is \emph{extendible} if there is a good two-dimensional independence notion $\nf$ on $\K_{\s}$ which reflects down. We say that $\s$ is \emph{very good} if in addition we can find such an $\nf$ which is also very good.
\end{defin}

We will use the following sufficient condition for a good frame to be extendible:

\begin{fact}\LABEL{extendible-wd}
  Let $\s$ be a good frame and let $\lambda \in \dom (\s)$. If $2^{\lambda} < 2^{\lambda^+}$, $\s$ is categorical in $\lambda$, and $\lambda^+ \in \dom (\s)$, then $\s_\lambda$ is extendible.
\end{fact}
\begin{proof}
  Since $\lambda^{+} \in \dom(\s),$ every model in $\K_{\lambda^{+}}$ has a universal model over it. Thus by \cite[E.8]{Vas17} (the main idea of the argument is due to the first author, see \cite[p.~798]{shelahaecbook}), $\s_\lambda$ has what is called the existence property for uniqueness triples. By \cite[II.6.34]{shelahaecbook}, there is a two-dimensional independence relation $\nf$ on $\K_\lambda$ which satisfies long transitivity and respects $\s_\lambda$. Now, any time a type does not $\s (\nf)$-fork, this implies that it does not $\s_\lambda$-fork as $\nf$ respects $\s_\lambda$. But we know that $\s$ satisfies uniqueness and $\s (\nf)$ satisfies extension, see Fact \ref{nfm-props}, so by the proof of \cite[4.1]{bgkv-apal}, $\s = \s (\nf)$. We have shown that $\nf$ is good (since we only consider it as a relation on $\K_\lambda$, local character is not relevant).

  It remains to see that $\nf$ reflects down. Since $\s$ is a good frame and $\lambda^+ \in \dom (\s)$, we have that $\left(\K_{\s}\right)_{[\lambda, \lambda^+]}$ is $\lambda$-tame (Remark \ref{good-frame-local-rmk}). Thus we can apply \cite[7.15]{jarden-tameness-apal} to obtain that $\nf$ reflects down, as desired.
\end{proof}

Note also that being extendible implies that the frame itself can be extended (this is the main idea of the end of Chapter II in \cite{shelahaecbook}):

\begin{fact}\LABEL{splus-fact}
  If $\s$ is an extendible categorical good $\lambda$-frame, then there exists a (unique) categorical good frame $\ts$ such that $\dom (\ts) = \{\lambda, \lambda^+\}$ and $\ts_\lambda = \s_\lambda$.
\end{fact}

\begin{proof}
  Let $\K$ be the AEC generated by $\K_{\s}$. By \cite[III.1.6(2)]{shelahaecbook} (or see \cite[10.1.9]{jrsh875}), there is a good $\lambda^+$-frame $\s^+$ on $\Ksat_{\lambda^+}$. Since by \cite[III.1.10]{shelahaecbook}, $\Ksat_{\lambda^+}$ is $\lambda$-tame, nonforking in $\s^+$ must be generated by nonforking in $\s$ (i.e.\ $\s^+ = \s^{up} \, {\rest} \, \Ksat_{\lambda^+}$). Thus letting $\ts \coloneqq  \s^{up} \, {\rest} \, (\Ksat_{\lambda^+} \cup \K_\lambda)$, we get the result.
\end{proof}

\begin{defin}\LABEL{splus-def}\myindex{$\s^+$}
  For $\s$ an extendible categorical good $\lambda$-frame, we write $\s^+$ for $\ts_{\lambda^+}$, where $\ts$ is as given by Fact \ref{splus-fact}.
\end{defin}

Since we now know how to extend a frame to the next cardinal, it makes sense to define when one can do this successively:

\begin{defin}\LABEL{s10}[{\cite[III.1.12]{shelahaecbook}}]\myindex{$\s^{+n}$}\myindex{$(<\omega)$-extendible}\myindex{$n$-extendible}
  Let $\s$ be a categorical good $\lambda$-frame. We define by induction on $n < \omega$ what it means for $\s$ to be $n$-extendible as well as a good $\lambda^{+n}$-frame $\s^{+n}$ as follows:

  \begin{enumerate}
  \item $\s$ is always $0$-extendible and $\s^{+0} = \s$.
  \item $\s$ is $(n + 1)$-extendible if it is $n$-extendible and $\s^{+n}$ is extendible.
  \item If $\s$ is $(n + 1)$-extendible, let $\s^{+(n + 1)} \coloneqq  \left(\s^{+n}\right)^+$.
  \end{enumerate}

  We say that $\s$ is \emph{$(<\omega)$-extendible} if $\s$ is $n$-extendible for all $n < \omega$.
\end{defin}

When a good $\lambda$-frame is $(<\omega)$-extendible, then after taking its successor a few times, it becomes very good. Moreover after this is done one can ``connect'' forking between the cardinals, getting a good $[\lambda, \lambda^{<\omega})$-frame:

\begin{fact}\LABEL{very-good-facts}
  Let $\s$ be a categorical good $\lambda$-frame.
  
  \begin{enumerate}
  \item\label{very-good-1} If $\s$ is $\omega$-successful and $\goodp$ (in the sense of \cite[III.1]{shelahaecbook}), then $\s$ is $(<\omega)$-extendible.
  
  \item\label{very-good-2} If $\s$ is $4$-extendible, then $\s^{+3}$ is very good.
  
  \item If $\s$ is very good and $2$-extendible, then $\s^+$ is very good.
  
  \item\label{very-good-4} If $\s$ is very good and $(<\omega)$-extendible, then there exists a (unique) very good categorical frame $\ts$ such that $\dom (\ts) = [\lambda, \lambda^{+\omega})$ and $\ts_{\lambda^{+n}} = \s^{+n}$ for all $n < \omega$. 
  \end{enumerate}
\end{fact}

\begin{proof} \
  \begin{enumerate}
  \item Essentially follows from \cite[2.14]{counterexample-frame-v4-toappear}.
  
  \item By \cite[III.8.19]{shelahaecbook} (see the proof of \cite[12.14]{indep-aec-apal}). The four levels of extendibility are used to get a ``better'' frame each time we go from $\s^{+ k}$ to $\s^{+(k+1)}.$ Strong continuity is the hard part to obtain in this process, and Shelah`s proof uses quite sophisticated orthogonality calculus. 
  
  \item Also by \cite[III.8.19]{shelahaecbook}.
  \item Let $\ts \coloneqq  \s^{up} \, {\rest} \, (\K_{\s} \cup \Ksat_{[\lambda^+ ,\lambda^{+\omega})})$. As in the proof of Fact \ref{splus-fact}, we have enough tameness for nonforking to connect, so this works (see also \cite[Appendix A]{Vas17}).
  \end{enumerate}
\end{proof}

While in general, it is not known how to build extendible good frames without using the weak diamond (Fact \ref{extendible-wd}), in the compact case it is possible:

\begin{fact}\LABEL{compact-very-good}
  Let $\K$ be a compact AEC. If $\K$ is categorical in some $\mu > \LS (\K)$, then there exists a (unique) categorical very good frame on $\Ksat_{\ge \LS (\K)^{+6}}$.
\end{fact}
\begin{proof}
  Let $\kappa$ be such that $\K$ is $\kappa$-compact. By Corollary \ref{compact-unbounded}, $\K$ has amalgamation, no maximal models, is $(<\kappa)$-short, and is categorical in a proper class of cardinals. By essentially the main result of \cite{indep-aec-apal}, and more precisely by \cite[A.16]{ap-universal-apal}, $\Ksatp{\lambda}$ is what is called there fully good. Here, we have set $\lambda \coloneqq  \left(\LS (\K)^{<\kappa}\right)^{+5}$. Since $\kappa$ is strongly compact, we have by a result of Solovay (see the proof of \cite[20.8]{jechbook}) that $\LS (\K)^{<\kappa} \le \LS (\K)^+$ (of course this also holds if $\kappa = \aleph_0$). Thus $\lambda \le \LS (\K)^{+6}$. Checking the definition of fully good in \cite[8.4]{indep-aec-apal}, we see that this implies that $\Ksat_{\ge \LS (\K)^{+6}}$ carries a very good two-dimensional independence relation, as desired.
\end{proof}

%%%%%%%%%%%%%%%%%%%%%%%%%%%%%%%%%%%%%%%%%%%%%%
\newpage

\section{Multidimensional independence}\LABEL{multidim-sec}

We define here the main notion of this paper: multidimensional independence relations. This section contains mostly definitions and easy lemmas. A result of importance is how to build multidimensional independence relations from the two-dimensional relations (Theorem \ref{constr-thm-0}).

\subsection{Systems}

We start by defining what is meant by a system of models.

\begin{defin}\LABEL{e0}\myindex{system}
  For $\K$ an abstract class and $I = (I, \le)$ a partial order, an \emph{$(I, \K)$-system} is a sequence $\m = \seq{M_u : u \in I}$ such that $u \le v$ implies $M_u \lea M_v$. A \emph{$\K$-system} is an $(I, \K)$-system for some $I$. For $\m$ a $\K$-system, we write $I (\m)$ for the unique $I$ such that $\m$ is an $(I, \K)$-system. When $\K$ is clear from context, we omit it from all the above definitions.
\end{defin}

In practice, $I$ is quite often $\mathcal{P}^{-}(n),$ or $\mathcal{P}(n).$ Nevertheless to easily carry out induction in the study of the combinatorial properties of these systems, we need to consider general orders. 

We will often be interested in systems where all extensions are strict. In fact, a strengthening of this, being proper, is more useful (one can see it as an abstract version of Definition \ref{sat-defs}(\ref{brimmed-set-def}), when we take $A$ to consist of the union of several models):

\begin{defin}\LABEL{proper-def}\myindex{proper}
  Let $\m = \seq{M_u : u \in I}$ be an $I$-system. Let $\K^\ast$ be a skeleton of $\K$ (the reader can think of $\K^\ast = \K$ at first reading but the interesting case is $K^{*},$ the class of saturated models, in a good frame  and $M <_{\K^{*}} N$ iff $M \leq_{\K} N$ and $N$ limit over M).

  \begin{enumerate}
  \item Let $u \in I$. We say that $u$ is \emph{$\K^\ast$-proper in $\m$} if $M_u \in \K^\ast$ and there exists $N \in \K^\ast$ such that:
    \begin{enumerate}
    \item For all $v \in I$, $v < u$ implies $M_v \lea N \ltap{\K^\ast} M_u$ (note that the last inequality is strict!).
    
    \item For all $v \in I$, if $v < u$ and $M_v \in \K^\ast$, then $M_v \leap{\K^\ast} N$.
    \end{enumerate}

  \item For $I_0$ a sub-order of $I$, we say that $\m$ is \emph{$(I_0, \K^\ast)$-proper} if every $u \in I_0$ is $\K^\ast$-proper in $\m$. When $\K^\ast = \K,$ we omit it and when $I_{0} = I$ we may omit it.
  \end{enumerate}
  
\end{defin}

Note that a system where all extensions are strict may not be proper:

\begin{example}\LABEL{e4}
  Let $\K$ be the AEC of all infinite sets, ordered by subset. Let $A \subseteq B$ be two countably infinite sets with $B \backslash A$ infinite. Partition $B \backslash A$ into two non-empty sets $C_1$ and $C_2,$ $C_{1}$ finite so $C_{2}$ infinite and let $A_\ell \coloneqq  A \cup C_\ell$. Then the $\Ps (2)$-system $(A, A_1, A_2, B)$ is not proper but all the extensions are strict.
\end{example}

The next two lemmas are easy properties of proper systems: being proper does not depend on the exact indexing set, and proper systems can be built.

\begin{lem}\LABEL{proper-subsys}
  Let $I_0 \subseteq I_1 \subseteq  I$ all be partial orders and let $\m$ be an $I$-system. Let $\K^\ast$ be a skeleton of $\K$. If $\m$ is $(I_0, \K^\ast)$-proper, then $\m \, {\rest} \, I_1$ is $(I_0, \K^\ast)$-proper.
\end{lem}

\begin{proof}
  Straightforward.
\end{proof}

\begin{lem}\LABEL{skel-sys-lem}
  Let $\K^\ast$ be a skeleton of $\K$ (Definition \ref{skel-def}). Let $\m = \seq{M_u : u \in I}$ be an $I$-system with $I$ a finite partial order and let $v$ be maximal in $I$. If $\K$ has no maximal models, then there exists $N \in \K$ such that $M_{v} \lea N$ and $v$ is $\K^\ast$-proper in $\m \, {\rest} \, (I \backslash \{v\}) \smallfrown \seq{N}$ (restriction of a system means what is expected and is defined right after this lemma).  
\end{lem}

\begin{proof}
  Using that $\K$ has no maximal models, let $M \in \K$ be such that $M_v \lta M$. Now apply Lemma \ref{skel-order-lem} with $M$, $\seq{M_i : i < n}$ there standing for $M$, $\{M_u \in \K^\ast \mid u < v\}$ here.
\end{proof}

It is now time to define how systems can relate to each other, in particular how they can be isomorphic, extensions, etc.

\begin{defin}\LABEL{sys-ext-def}\
  \begin{enumerate}
  \item Let $I \subseteq J$ be partial orders and let $\m = \seq{M_u : u \in J}$ be a $J$-system. Then $\m \, {\rest} \, I \coloneqq  \seq{M_u : u \in I}$.\myindex{$\m \, {\rest} \, I$}
  
  \item\label{sys-ext-disj} For $\m_\ell= \seq{M_u^\ell : u \in I}$, $\ell = 1,2$ both $I$-systems, say $\m_1 \lea \m_2$ if $M_u^1 \lea M_u^2$ for all $u \in I$. We say that $\m_2$ is a \emph{disjoint} extension of $\m_1$ (written $\m_1 \lea^d \m_2$) if $\m_1 \lea \m_2$ and $M_u^2 \cap M_v^1 \subseteq M_u^1$ when $u \leq_{I} v$. \myindex{$\lea^d$}
  
  \item \cite[III.12.14(4)]{shelahaecbook} For $k < \omega$, $\m_\ell = \seq{M_u^\ell : u \in I}$, $\ell < k$ all $I$-systems with $\m_0 \lea \m_1 \lea \ldots \lea \m_{k - 1}$, let $\m \coloneqq  \m_0 \ast \m_1 \ast \ldots \ast \m_{k - 1}$ be the system defined as follows:\index{$\m_1 \ast \m_2$}

    \begin{enumerate}
    \item It is indexed by the partial order $J \coloneqq  I \times k$ ordered lexicographically, so $\m = \seq{M_u : u \in J}$.
    \item For all $u \in I$ and $i < k$, $M_{(u, i)} = M_u^i$.
    \end{enumerate}
    
  \item\label{sys-embed} For $\m_\ell = \seq{M_u^\ell : u \in I}$, $\ell = 1,2$ both $I$-systems, we say $f$ is a \emph{system embedding} from $\m_1$ to $\m_2$ and write $f: \m_1 \rightarrow \m_2$ if $f = \seq{f_u :u \in I}$, for all $u \in I$, $f_u$ is a $\K$-embedding from $M_u^1$ to $M_u^2$, and $u \le v$ implies $f_u \subseteq f_v$. \myindex{system embedding}
  
  \item For $\m_\ell = \seq{M_u^\ell : u \in I}$, $\ell = 1,2$ both $I$-systems, we say $f$ is an \emph{isomorphism} from $\m_1$ to $\m_2$ and write $f: \m_1 \cong \m_2$ if $f = \seq{f_u : u \in I}$ is a system embedding from $\m_1$ to $\m_2$ and for all $u \in I$, $f_u$ is an isomorphism from $M_u^1$ onto $M_u^2$. \myindex{isomorphism of systems}
  
  \item\label{sys-iso} For $\pi: I_1 \cong I_2$ an isomorphism of partial orders, $\m = \seq{M_u : u \in I_1}$ an $I_1$-system, let $\pi (\m)$ denote the $I_2$-system $\seq{M_{\pi^{-1} (v)} : v \in I_2}$.\myindex{$\pi (\m)$}
  
  \item\label{sys-disj} A system $\m = \seq{M_u : u \in I}$ is called \emph{disjoint} if whenever $u, v, w, u^\ast \in I$ are such that $u \le v \le u^\ast$ and $u \le w \le u^\ast$ and $v \cap w = u,$ then $M_u = M_v \cap M_w$.\myindex{disjoint system}
  
  \item A system $\m = \seq{M_u : u \in I}$ is called \emph{fully disjoint} if whenever $u, v, w \in I$ are such that $u \le v, v \cap w = u$ and $u \le w$, then $M_u = M_v \cap M_w$.\myindex{fully disjoint system}
  
%  \item We say that $N \in \K$ is \emph{prime}\svnote{Needed?} over a system $\m = \seq{M_u :u \in I}$ if $M_u \lea N$ for all $u \in I$ and whenever $N' \in \K$ is such that $M_u \lea N'$ for all $u \in I$, there exists a $\K$-embedding $f: N \rightarrow N'$ which fixes $M_u$ for all $u \in I$.
  \end{enumerate}
\end{defin}

\begin{remark}\LABEL{e6}
  Every disjoint system is isomorphic to a fully disjoint system.
\end{remark}

It is often useful to code systems as an abstract class:

\begin{defin}\LABEL{sys-voc-def}\myindex{vocabulary of systems}\myindex{$\tau^I$}\myindex{$\K^I$}
  Let $\K$ be an abstract class and let $I$ be a partial order. The \emph{vocabulary of $(\K, I)$-systems} is $\tau^I \coloneqq  \tau (\K) \cup \{P_i : i \in I\}$, where each $P_i$ is a new unary predicate. The \emph{abstract class of $(\K, I)$-systems} is the abstract class $\K^I = (K^I, \leap{\K^I})$ defined as follows:

  \begin{enumerate}
  \item $K^I$ consists of all the $\tau^I$-structures $M$ such that $\seq{M^{P_i} : i \in I}$ forms a fully disjoint $(\K, I)$-system (we see $M^{P_i}$ as a $\tau (\K)$-structure). We identify the elements of $K^I$ with the corresponding systems.
  \item $\m_1 \leap{\K^I} \m_2$ if and only if $\m_1 \lea^d \m_2$ (see Definition \ref{sys-ext-def}(\ref{sys-ext-disj})).
  \end{enumerate}
\end{defin}

\subsection{Multidimensional independence relations}

Multidimensional independence relations consist of systems indexed by a certain class of semilattices. For reasons that will become apparent, we require that this class has a certain amount of closure:

\begin{defin}\LABEL{closed-def}\myindex{closed class of semilattice}
  A class $\Iis$ of semilattices is called \emph{closed} if:

  \begin{enumerate}
  \item[(A)] $\Iis$ is closed under isomorphisms.
  
  \item[(B)] $\Iis$ is closed under taking initial segments.
  
  \item[(C)] For any $I \in \Iis$ and any $u \in I \backslash \{\bot\}$, $[u, \infty)_I \times \{0,1\} \in \Iis$.
  
  \item[(D)] (Follows) If $I_{0} \times \{ 0, 1 \} \in \mathcal{I},$ then $I_{0} \in \mathcal{I}$ (it is isomorphic to an initial segment of $\mathcal{I}_{0} \times \{ 0, 1 \}$). 
  \end{enumerate}
\end{defin}

The main examples of a closed class of semilattice are:

\begin{lem}\LABEL{closed-lem}
  Let $n < \omega$ and let $\Iis$ be the class of all initial semilattices isomorphic to an initial segments of $\Ps (n)$. Then $\Iis$ is closed.
\end{lem}
\begin{proof}
  The third condition of Definition \ref{closed-def} is the only non-trivial one. Let $I \in \Iis$, and identify it with an initial segment of $\Ps (n)$. Let $u \in I \backslash \{\bot\}$ (so $u \neq \emptyset$). Then $[u, \infty)_I$ is the set of all subsets of $\Ps (n)$ extending $u$. As a semilattice, it is isomorphic to $\Ps (n \backslash u) \cong \Ps (n - |u|)$. Thus $[u, \infty)_I \times \{0, 1\} \cong \Ps (n - |u|) \times \Ps (1) \cong \Ps (n - |u| + 1)$, which since $|u| \ge 1$ is an initial segment of $\Ps (n)$.
\end{proof}

A multidimensional independence relation consists of an abstract class $\K$, a closed class of finite semilattices, and a class of systems indexed by these semilattices. Several properties are required, akin to the requirements in the definition of a two-dimensional independence relation. A complication arises in the Monotonicity properties: there are several and may not always seem natural (it may help to jump ahead and come back once the properties are used; for example a good hint to understand the definition of Monotonocity $4$ is to looks at the proof of Lemma \ref{trans-lem}). Nevertheless they will hold for the relation we construct and are all used in the next section. Note that we do not know whether a multidimensional independence relations are canonical (i.e.\ whether under reasonable conditions, also involving the existence and uniqueness properties to be defined later, there can be at most one satisfying this list of axiom).

\begin{defin}\LABEL{multi-def}\myindex{multidimensional independence relation}\index{multidimensional independence notion|see {multidimensional independence relation}}\index{$\is$|see {multidimensional independence relation}}\index{$\K_{\is}$|see {multidimensional independence relation}}\index{$\NF_{\is}$|see {multidimensional independence relation}}\index{$\Iis_{\is}$|see {multidimensional independence relation}}
  
  A \emph{multidimensional independence relation (or notion)} is a triple $\is = (\K, \Iis, \NF)$, where $\K = \K_{\is}$ is an abstract class, $\Iis = \Iis_{\is}$ is a class of semilattices, and every element of $\NF = \NF_{\is}$ is a $(\K, I)$-systems with $I \in \mathcal{I}.$ We call the members of $\NF$  \emph{independent systems}, and write $\NF(\m)$ instead of $\m \in \NF.$  We require the following properties:

  \begin{enumerate}
  \item\label{multi-1} $\Iis$ is a \emph{closed} class of \emph{finite} semilattices.
  \item Invariance: If $\m$ and $\m'$ are systems and $f: \m \cong \m'$, then $\NF (\m)$ if and only if $\NF (\m')$.
  \item Nontriviality: $\NF \neq \emptyset$ and for every $I \in \Iis$ and every $M \in \K$, there exists an independent $I$-system containing $M$.
  \item Disjointness: Every independent system is disjoint (see Definition \ref{sys-ext-def}(\ref{sys-disj})).
  \item Symmetry: If $I_1, I_2 \in \Iis$, $\pi: I_1 \cong I_2$ is an isomorphism, and $\m$ is an $I_1$-system, then $\NF (\m)$ if and only if $\NF (\pi (\m))$, where $\pi (\m)$ is the $I_2$-system naturally induced by $\pi$ (see Definition \ref{sys-ext-def}(\ref{sys-iso})).
  \item Transitivity: If $\m_0 \lea \m_1 \lea \m_2$, $\m_0 \ast \m_1$ is independent, and $\m_1 \ast \m_2$ is independent, then $\m_0 \ast \m_2$ is independent.
\item Monotonicity 1: If $\m$ is an $I$-system, $\NF (\m)$, and $J \subseteq I$ is such that $J \in \Iis$, then $\NF (\m \, {\rest} \, J)$.
\item Monotonicity 2: If $\m = \seq{M_u : u \in I}$ is an $I$-system, $v$ is a maximal element in $I$, and $M_v \lea N$, then $\m$ is independent if and only if $\seq{M_u : u \in I \backslash \{v\}} \smallfrown \seq{N}$ is independent.
\item Monotonicity 3: If $\m$ is an $I$-system, then $\m$ is independent if and only if for any maximal $v \in I$, $\m \, {\rest} \, \{u \in I : u \le v\}$ is independent.
\item Monotonicity 4: Let $I_1, I_2 \subseteq I$ all be in $\Iis$. Let $\m$ be an $I$-system. If:
  \begin{enumerate}
  \item $I_1$ is an initial segment of $I$.
  \item $I_1 \cap I_2$ is cofinal in $I_1$.
  \item $I_1 \cup I_2 = I$.
  \item $\m \, {\rest} \, I_1$ and $\m \, {\rest} \, I_2$ are independent.
  \end{enumerate}

  Then $\m$ is independent.

\item Monotonicity 5: Let $I \in \Iis$, let $u \in I \backslash \{\bot\}$, and let $I_0 \coloneqq  [u, \infty)_I$. Let $\m$, $\m'$ be $I$-systems, $\m_0 \coloneqq  \m \, {\rest} \, I_0$ be independent, and $\m_0' \coloneqq  \m' \, {\rest} \, I_0$ be such that $\m_0 \lea \m_0'$ and $\m_0 \ast \m_0'$ is independent. Assume that $\m \, {\rest} \, (I \backslash I_0) = \m' \, {\rest} \, (I \backslash I_0)$. Then $\m$ is independent if and only if $\m'$ is independent.
  \end{enumerate}
\end{defin}

The natural restrictions can be defined on a multidimensional independence relation:

\begin{defin}\LABEL{e14}\myindex{$\is_\lambda$}\myindex{$\is \, {\rest} \, \Jis$}\myindex{$\is \, {\rest} \, \K^\ast$}
  Let $\is = (\K, \Iis, \NF)$ be a multidimensional independence relation.

  \begin{enumerate}
  \item For $\Jis \subseteq \Iis$ a class of partial orders, we let $\is \, {\rest} \, \Jis$ denote the multidimensional independence relation $\js = (\K, \Jis, \NF \, {\rest} \, \Jis)$, where $\NF \, {\rest} \, \Jis$ denotes the class of systems in $\NF$ that are $J$-systems for $J \in \Jis$.
  \item For $\K^\ast$ a sub-abstract class of $\K$, let $\is \, {\rest} \, \K^\ast$ denote the multidimensional independence relation $\js = (\K^\ast, \Iis, \NF \, {\rest} \, \K^\ast)$, where $\NF \, {\rest} \, \K^\ast$ denotes the class of systems in $\NF$ that are also $\K^\ast$-systems. We may write $\is_\lambda$ instead of $\is \, {\rest} \, \K_\lambda$.
  \end{enumerate}
\end{defin}

The next lemma shows that, under some closure conditions on $\Iis$, transitivity follows from the other axioms:

\begin{lem}\LABEL{trans-lem}
  Let $\is$ be a multidimensional independence relation except transitivity is not assumed. Let $\m_0, \m_1, \m_2$ be $I$-systems with $I \in \Iis_{\is}$. If $\m_0 \leap{\K_{\is}} \m_1 \leap{\K_{\is}} \m_2$, $\m_0 \ast \m_1$ is independent, $\m_1 \ast \m_2$ is independent, and $I \times 3 \in \Iis_{\is}$, then $\m_0 \ast \m_1 \ast \m_2$ is independent. In particular, if $\Iis_{\is}$ is closed under products the transitivity axiom follows from the others.
\end{lem}

\begin{proof}
  Say the $\m_\ell$'s are $I$-systems.  Let $J \coloneqq  I \times 3$. Recalling Definition \ref{sys-ext-def}(3) let  $\m \coloneqq  (\m_0 \ast \m_1 \ast \m_2)$. We check that $\m$ is independent by using monotonicity 4: let $I_1 \coloneqq  I \times \{0, 1\}$ and let $I_2 \coloneqq  I \times \{1, 2\}$. We have that $\m \, {\rest} \, I_{1} \ast \m_{1}$ and $\m \, {\rest} \, I_{2} = \m_{1} \ast \m_{2}.$  Both are independent by assumption. Moreover, $I_1 \cup I_2 = J$, $I_1$ is an initial segment of $J$, and $I_1 \cap I_2 = I \times \{1\}$ is cofinal in $I_1$. Thus monotonicity 4 implies that $\m$ is independent. The ``in particular'' part follows from monotonicity 1.
\end{proof}

One wants to order systems so that if $\m_2$ extends $\m_1$, then the entire diagram that this forms is independent. This is formalized by looking at $\m_1 \ast \m_2$:

\begin{defin}\LABEL{e17}\myindex{$\leap{\is}$}\myindex{$\K_{\is, I}$}
  Let $\is$ be a multidimensional independence relation. We define a binary relation $\leap{\is}$ on $\NF_{\is}$ as follows: $\m_1 \leap{\is} \m_2$ if and only if $\m_1 \lea \m_2$ (so in particular $I (\m_1) = I (\m_2)$) and $\NF_{\is} (\m_1 \ast \m_2)$. Let $\K_{\is, I}$ be the sub-abstract class of $\K^I$ (Definition \ref{sys-voc-def}) consisting of all independent $I$-systems and ordered by $\leap{\is}$. \myindex{$\lea$}
\end{defin}

One may want in addition to require that the system $\m_1 \ast \m_2$ is \emph{proper}:

\begin{defin}\LABEL{proper-class}\myindex{$\Kprop_{\is, I}$}
  Let $\is$ be a multidimensional independence relation. Let $I \in \Iis_{\is}$. We let $\Kprop_{\is, I}$ be the class of \emph{proper} independent $I$-systems ordered by $\m_1 \leap{\Kprop_{\is, I}} \m_2$ if and only if either $\m_1 = \m_2$ or $\m_1 \ast \m_2$ is a \emph{proper} independent system.
\end{defin}

Absent from the properties in Definition \ref{multi-def} were any kind of extension or uniqueness. We define these now. Note that existence says there \emph{is} a system, extension that an \emph{existing} system can be extended. Strong uniqueness will hold only in classes of limit models, where also the extensions are limit: it says in particular that there is only one proper system indexed by a given lattice. Uniqueness is the property that will hold in the original class: it says that each independent system has at most one completion, in the sense that any two completions amalgamate (so this can apply only when we fix the cardinality).

\begin{defin}\LABEL{multidim-props}\myindex{existence (for multidimensional systems)}\myindex{extension (for multidimensional systems)}\myindex{strong uniqueness (for multidimensional systems)}\myindex{uniqueness (for multidimensional systems)}
  Let $\is = (\K, \Iis, \NF)$ be a multidimensional independence relation. The following are additional properties that $\is$ may have:
  
\begin{enumerate}
\item Existence: For any $I \in \Iis$ there exists a proper independent $I$-system.
\item Extension: For any $I, J \in \Iis$ with $I$ an initial segment of $J$, if $\m$ is an independent $I$-system, then there exists a $(J \backslash I)$-proper independent $J$-system $\m'$ such that $\m \cong \m' \, {\rest} \, I$.
\item Strong uniqueness: Let $\m_1$, $\m_2$ be independent $J$-systems and let $I$ be an initial segment of $J$ with $I \in \Iis$. Let $f: \m_1 \, {\rest} \, I \cong \m_2 \, {\rest} \, I$. If $\m_1$ and $\m_2$ are $(J \backslash I)$-proper (Definition \ref{proper-def}), then there exists $g: \m_1 \cong \m_2$ extending $f$.
\item Uniqueness: Let $\m_1$, $\m_2$ be independent $J$-systems and let $I$ be an initial segment of $J$ with $J = I \cup \{v\}$, $I < v$. Let $f: \m_1 \, {\rest} \, I \cong \m_2 \, {\rest} \, I$. Then there exists an independent $J$-system $\m$ and systems embeddings (see Definition \ref{sys-ext-def}(\ref{sys-embed}) $g_\ell \colon \m_\ell \rightarrow \m$, $\ell = 1,2$ such that $g_\ell$ extends $f$ and $g_{2}$ extends $f^{-1}.$ 
\end{enumerate}
\end{defin}

Any reasonable multidimensional independence relation induces a two-dimensional independence relation:

\begin{defin}\LABEL{nf-is-def}\myindex{$\nf (\is)$}
  Let $\is$ be a multidimensional independence relation with $\Ps (2) \in \Iis_{\is}$. Define a $4$-ary relation $\nf = \nf (\is)$ by $\nfs{M_0}{M_1}{M_2}{M_3}$ if and only if $(M_0, M_1, M_2, M_3)$ (seen as a system indexed by $\Ps (2)$) is an independent system.
\end{defin}

Note that $\nf(\is)$ is morally the same as $\is$ restricted to $\mathcal{P}(2)$ but they are different for boring reasons (the first is basically a $4$-ary on $\K$, the second is basically a map from $\mathcal{P}(2)$ to $\K$). 

\begin{lem}\LABEL{e19}
  If $\is$ is a multidimensional independence relation with existence, extension, and uniqueness such that $\Ps (2) \in \Iis_{\is}$, then $\nf (\is)$ (from Definition \ref{nf-is-def}) is a two-dimensional independence notion on $\K_{\is}$.
\end{lem}
\begin{proof}
  The least trivial axioms to prove are monotonicity and symmetry. Symmetry can be obtained by using the symmetry property of multidimensional independence relations with the automorphism of $\Ps (2)$ swapping $\{0\}$ and $\{1\}$. Monotonicity actually follows from the other axioms, see \cite[3.23]{indep-categ-v3}.
\end{proof}

\subsection{Building a multidimensional independence relation}

Starting from a two-dimensional independence relation, one can also build a multidimensional independence relation in the natural way:

\begin{defin}\LABEL{is-nf-def}\myindex{$\is (\nf)$}\

  (1) Let $\nf$ be a two-dimensional independence notion on an abstract class $\K$. Let $\is = \is ({\nf})$ be the following multidimensional independence relation:

  \begin{enumerate}
  \item[(A)] $\K_{\is} = \K$.
  
  \item[(B)] $\Iis_{\is}$ is the class of all finite semilattices.
  
  \item[(C)] For $I \in \Iis_{\is}$ and $\m = \seq{M_u : u \in I}$ an $I$-system, $\m \in \NF_{\is}$ if and only if whenever $u_1, u_2, v \in I$ and $u_1 \le v$, $u_2 \le v$, then $\nfs{M_{u_1 \land u_2}}{M_{u_1}}{M_{u_2}}{M_{v}}$.
  \end{enumerate}
 
  (2) We say $\K$ has \emph{uniqueness}\myindex{uniqueness for $\K$} (apply mainly to $\K^{*}$ above) \underline{when} if $M<_{\K^{*}} N_{\ell}$ both have the same cardinality for $\ell = 1, 2$ then $N_{1}, N_{2}$ are isomorphic over $M$ (so $K^{*}$ just\footnote{here $\K^{*} = \K$ is not a reasonable choice as in the usual case, $\K_{\lambda}^{*}$ cannot be an AEC.} an abstract class). The default\footnote{Earlier it was suggested $\K^{*} = \K;$ but for uniqueness to make sense we need this e.g. \ref{constr-thm-0}} case is $K^{*}$ the class of saturated models, and $M <_{\K^{*}} N$ iff $M \leq_{\K} N$ and $N$ limit over $M.$
\end{defin}

\begin{thm}\LABEL{constr-thm-0}
  If $\nf$ is a two-dimensional independence notion on an abstract class $\K$ with no maximal models, then $\is = \is (\nf)$ is a multidimensional independence notion such that $\nf(\is) = \nf$. Moreover, $\is$ restricted to the class of initial segments of $\Ps (2)$ has existence, extension, and uniqueness.
\end{thm}
\begin{proof}
  Most of the axioms follow directly from the definitions. We only prove:

  \begin{itemize}
  \item \underline{Monotonicity 4}: Let $I_1, I_2$, $\m = \seq{M_u : u \in I}$ be as in the statement of monotonicity 4. Let $u_1, u_2, v \in I$ with $u_1 \le v$, $u_2 \le v$. We have to see that $\nfs{M_{u_1 \land u_2}}{M_{u_1}}{M_{u_2}}{M_v}$. If both $u_1$ and $u_2$ are in $I_1$, or both are in $I_2$, then this is immediate from the assumptions that $\m \, {\rest} \, I_1$ and $\m \, {\rest} \, I_2$ are both independent (and preservation of $\nf$ under $\K$-embeddings). Assume now that one is in $I_1 \backslash I_2$ and the other in $I_2 \backslash I_1$. By symmetry, without loss of generality $u_1 \in I_1 \backslash I_2$, $u_2 \in I_2 \backslash I_1$.

    By assumption, $I_1 \cap I_2$ is  maximal in $I$ and $I = I_1 \cup I_2$, so there exists $u' \in I_2$ such that $u_1 \le u'$. Let $u_1' \coloneqq  \land \{u' \in I_2 \cap I_1 \mid u_1 \le u'\}$. In other words, $u_1'$ is the minimal element of $I_1 \cap I_2$ such that $u_1 \le u_1'$. By minimality, we must have that $u_1' \le v$ (otherwise, consider $u_1'' \coloneqq  u_1' \land v$; since $I_1$ is an initial segment, this is in $I_1 \land I_2$ and is still above $u_1$). By the previous paragraph, we have that $\nfs{M_{u_1' \land u_2}}{M_{u_1'}}{M_{u_2}}{M_v}$. Also by the previous paragraph (recalling that $I_1$ is an initial segment, so $u_1' \land u_2 \in I_1 \cap I_2$), $\nfs{M_{u_1 \land u_1' \land u_2}}{M_{u_1}}{M_{u_1' \land u_2}}{M_{u_1'}}$. By transitivity for $\nf$, $\nfs{M_{u_1 \land u_1' \land u_2}}{M_{u_1}}{M_{u_2}}{M_v}$. Now note that $u_1 \le u_1'$ by assumption, hence $u_1 \land u_1' \land u_2 = u_1 \land u_2$, as desired.

  \item \underline{Monotonicity 5}: Fix the data given by the statement of monotonicity 5. Write $\m = \seq{M_u : u \in I}$ and $\m' = \seq{M_u' : u \in I}$.

    \begin{itemize}
    \item We first show that $\m'$ independent implies that $\m$ is independent. Let $u_1, u_2, u_3 \in I$ be such that $u_1, u_2 \le u_3$. We want to see that $\nfs{M_{u_1 \land u_2}}{M_{u_1}}{M_{u_2}}{M_{u_3}}$. There are several cases:
      \begin{itemize}
      \item If $u_1, u_2 \notin I_0$, then $u_1 \land u_2 \notin I_0$ (as $I \backslash I_0 \subseteq I$), hence the result follows directly from the assumption that $\nfs{M_{u_1 \land u_2}'}{M_{u_1}'}{M_{u_2}'}{M_{u_3}'}$ and ambient monotonicity.
      \item If $u_1, u_2 \in I_0$, then apply transitivity.
      \item If $u_1 \in I_0$, $u_2 \notin I_0$, and $u_1 \land u_2 \in I_0$, apply transitivity again.
      \item If $u_1 \in I_0$, $u_2 \notin I_0$, and $u_1 \land u_2 \notin I_0$, use monotonicity.
      \end{itemize}
    \item Assume now that $\m$ is independent. Let $u_1, u_2, u_3 \in I$ be such that $u_1, u_2 \le u_3$. We want to see that $\nfs{M_{u_1 \land u_2}'}{M_{u_1}'}{M_{u_2}'}{M_{u_3}'}$ and we again split into cases. The cases $u_1, u_2 \notin I_0$ and $u_1, u_2 \in I_0$ are similar to before. If $u_1 \in I_0$, $u_2 \notin I_0$, then $u_3 \in I_0$ since $I_0$ is and end segment. Thus we can apply transitivity and we are done.
      \end{itemize}
    \item \underline{Transitivity}: follows from Lemma \ref{trans-lem}.
  \end{itemize}
  
  That $\is$ restricted to the class of initial segments of $\Ps (2)$ has existence, extension, and uniqueness follows from the corresponding properties of $\nf$: we only have to note that a $\Ps (2)$-system $\seq{M_u : u \in \Ps (2)}$ is independent if and only if $\nfs{M_{\emptyset}}{M_{\{0\}}}{M_{\{1\}}}{M_2}$, by the symmetry axiom of $\nf$.
\end{proof}

%%%%%%%%%%%%%%%%%%%%%%%%%%%%%%%%%%%%%%%%%%%%%%
\newpage

\section{Finite combinatorics of multidimensional independence}\LABEL{multi-finite-sec}

In this section, we assume:

\begin{hypothesis}
  $\is = (\K, \Iis, \NF)$ is a multidimensional independence relation.
\end{hypothesis}

We consider properties of $\is$ that are finitary in nature. That is, they do not depend on any kind of closure under unions of chains. This allows us to work at a high level of generality (for example, $\K$ is just assumed to be an abstract class). A crucial question is whether properties (such as extension or uniqueness) that hold for limit systems (that is, for systems consisting of limit models, ordered by being ``limit over'') will also hold for the not necessarily limit other systems. We prove general positive results in this direction here (Theorems \ref{brim-ext-thm} and \ref{brim-uq-thm}). The statements use the notion of a skeleton. This is also a place where we use that $\Iis$ is closed (Definition \ref{closed-def}), as well as the monotonicity axioms in Definition \ref{multi-def}.

We start by proving a few easy consequences of the definition of a multidimensional independence relation. First, $\Ps (1)$-systems are trivial:

\begin{lem}\LABEL{lea-lem}
  $M \lea N$ both be in $\K$. If $\Ps (1) \in \Iis$, then $\seq{M, N}$ is an independent $I$-system.
\end{lem}
\begin{proof}
  By nontriviality, there is an independent $\Ps (1)$-system $\m$ containing $M$. Write $\m = \seq{M, N'}$. By monotonicity 2, $\seq{M, M}$ is independent. By monotonicity 2 again, $\seq{M, N}$ is independent.
\end{proof}

It is also easy to see that existence is weaker than extension:

\begin{lem}\LABEL{ext-implies-exi}
  If $\is$ has extension, then $\is$ has existence.
\end{lem}
\begin{proof}
  Let $I \in \Iis$. Start with the empty system and extend it to a proper independent $I$-system.
\end{proof}

Uniqueness is also weaker than strong uniqueness:

\begin{lem}\LABEL{n3}
  If $\is$ has strong uniqueness and $\K$ has no maximal models, then $\is$ has uniqueness.
\end{lem}
\begin{proof}
  Given two models $M_1, M_2$ to amalgamate over an independent system $\m$, take a strict extension $M_\ell'$ of $M_\ell$ and use strong uniqueness to see that $M_1'$ is isomorphic to $M_2'$ over $\m$, hence $M_1$ and $M_2$ amalgamate over $\m$.
\end{proof}

%% \begin{lem}
%%   If $\is$ has strong uniqueness and $\K$ has no maximal models, then $\is$ has homogeneity.
%% \end{lem}
%% \begin{proof}
%%   Let $\m \coloneqq  \seq{M_u : u \in I}$ be a strict independent system and let $v \in I$ be maximal. Since $\K$ has no maximal models, pick $M \in \K$ such that $M_v \lta M$. Let $J \coloneqq  I \backslash \{v\}$. Note that $J$ is an initial segment of $I$. By monotonicity 2, $\m \, {\rest} \, J \smallfrown \seq{M}$ is an independent $I$-system. By strong uniqueness, there exists an isomorphism $f: M \cong M_v$ fixing $\m \, {\rest} \, J$. Then $N \coloneqq  f[M_v]$ is as desired.
%% \end{proof}

\subsection{Skeletons of independence relations}

It is natural to consider what happens to $\is$ when restricting it to a skeleton of $\K$. We overload terminology and also call such restrictions \emph{skeletons} of $\is$:

\begin{defin}\LABEL{n6}\myindex{skeleton of a multidimensional independence relation}
  Let $\is$ and $\is^\ast$ be independence relations. We say that $\is^\ast$ is a \emph{skeleton} of $\is$ if:

  \begin{enumerate}
  \item[(A)] $\K_{\is^\ast}$ is a skeleton of $\K_{\is}$.
  
  \item[(B)] $\is^\ast = \is \, {\rest} \, \K_{\is^\ast}$.
  \end{enumerate}
\end{defin}

\begin{remark}
  If $\K^\ast$ is a skeleton of $\K$, then it is straightforward to check that $\is \, {\rest} \, \K^\ast$ will be a multidimensional independence relation.
\end{remark}

We will consider the following localized versions of extension and uniqueness:

\begin{defin}\LABEL{n9}
  Let $\m$ be an independent $I$-system and let $\K^\ast$ be a skeleton of $\K$.

  \begin{enumerate}
  \item We call $\m$ a \emph{$\K^\ast$-extension base in $\is$} if $J \in \Iis, J = I \cup \{v\}$ and $I < v$, there exists an independent $J$-system $\m'$ such that $v$ is $\K^\ast$-proper in $\m'$ and $\m \cong \m' \, {\rest} \, I$. When $\K^\ast = \K$ and $\is$ is clear from context, we omit them and call $\m$ an \emph{extension base}.\myindex{extension base}
  
  \item We call $\m$ a \emph{$\K^\ast$-strong uniqueness base in $\is$} if $J \in \Iis, J = I \cup \{v\}$ and $I < v$ and any two independent $J$-systems $\m_1, \m_2$ such that $\m_1 \, {\rest} \, I = \m_2 \, {\rest} \, I = \m$ and $v$ is $\K^\ast$-proper in both $\m_1$ and $\m_2$, we have that $\m_1 \cong_{\m} \m_2$. As before, when $\K^\ast = \K$ and $\is$ is clear from context, we omit may them.\myindex{strong uniqueness base}
  
  \item We call $\m$ a \emph{uniqueness base in $\is$} if for any $J \in \Iis$ such that $J = I \cup \{v\}$ and $I < v$ and any two independent $J$-systems $\m_1, \m_2$ such that $\m_1 \, {\rest} \, I = \m_2 \, {\rest} \, I = \m$, we have that there is an $\is$-independent $J$-system $\m^\ast$ and $f_\ell : \m_\ell \xrightarrow[\m]{} \m^\ast$ systems embeddings for $\ell = 1,2$. When $\is$ is clear from context, we may omit it.\myindex{uniqueness base}
  \end{enumerate}
\end{defin}

It turns out that the parameter $\K^\ast$ in the definition of a $\K^\ast$-extension base can safely be omitted. This will be used without comments:

\begin{lem}\LABEL{n13}
  Let $\m$ be an $\is$-independent system and let $\K^\ast$ be a skeleton of $\K$. Let $\is^\ast \coloneqq  \is \, {\rest} \, \K^\ast$.

  \begin{enumerate}
  \item $\m$ is a $\K$-extension base in $\is$ if and only if $\m$ is a $\K^\ast$-extension base in $\is$.
  
  \item If $\m$ is $\is^\ast$-independent, then $\m$ is a $\K^\ast$-extension base in $\is^\ast$ if and only if $\m$ is a $\K^\ast$-extension base in $\is$.
  \end{enumerate}
\end{lem}

\begin{proof}
  Use monotonicity 2 and play with the definition of a skeleton (as in the proof of Lemma \ref{skel-sys-lem}). 
\end{proof}

Of course, having extension is the same as all independent systems being extension bases, and similarly for the other properties:

\begin{lem}\LABEL{base-equiv}
$\is$ has extension [uniqueness] [strong uniqueness] if and only if every independent system is an extension [uniqueness] [strong uniqueness] base.
\end{lem}
\begin{proof}
  Straightforward.
\end{proof}

Note also that the uniqueness properties transfers down from $\is$ to $\is^\ast$:

\begin{lem}\LABEL{n17}
  Let $\is^\ast$ be a skeleton of $\is$.  If $\is$ has [strong] uniqueness, then $\is^\ast$ has [strong] uniqueness. 
\end{lem}
\begin{proof} 
  Straightforward using the basic properties of a skeleton.
\end{proof}

The extension property also transfers. In fact, more can be said:

\begin{lem}\LABEL{skel-ext}
  Let $\is^\ast$ be a skeleton of $\is$. Write $\K \coloneqq  \K_{\is}$, $\K^\ast \coloneqq  \K_{\is^\ast}$. Assume that $\is$ has extension. Let $I, J \in \Iis$ with $I$ an initial segment of $J$. Let $\m$ be an $\is$-independent $I$-system. There exists an $\is$-independent $J$-system $\m'$ such that:

  \begin{enumerate}
  \item $\m \cong \m' \, {\rest} \, J$.
  \item $\m'$ is $(J \backslash I)$-$\K^\ast$-proper (see Definition \ref{proper-def}).
  \end{enumerate}

  In particular, $\is^\ast$ has extension.
\end{lem}
\begin{proof}
  Play with skeletons as in the proof of Lemma \ref{skel-sys-lem}.
\end{proof}

We now turn to the harder problem of getting the properties in $\is$ if we know them in the skeleton. For extension, this can be done provided that the underlying class is the same (the ordering may be different). The example to keep in mind is $\K$ being the class of saturated models of cardinality $\lambda$ in a given AEC, and $\K^\ast$ consisting of the same class, but ordered by ``being limit over''. Note however that the theorem below also has content when $\is = \is^\ast$: it says that it is enough to prove that \emph{proper} systems are extension bases.

\begin{thm}\LABEL{brim-ext-thm}
  Let $\is^\ast$ be a skeleton of $\is$ such that $|\K_{\is^\ast}| = |\K_{\is}|$. If every proper $\is^\ast$-independent system is an extension base, then $\is$ has extension.
\end{thm}

\begin{proof}
  Write $\K \coloneqq  \K_{\is}$, $\K^\ast \coloneqq  \K_{\is^\ast}$. For $\m$ an $\is$-independent $I$-system, let $k (\m)$ be the number of $u \in I$ which are \emph{not} $\K^\ast$-proper in $\m$. We prove that every $\is$-independent system $\m$ is an extension base by induction on $k (\m)$.

  Let $\m$ be an independent $I$-system and let $J = I \cup \{v\} \in \Iis$, with $I < v$. If $k (\m) = 0$, then $\m$ is an extension base by assumption. Assume now that $k (\m) > 0$. Let $u \in I$ be such that $u$ is \emph{not} $\K^\ast$-proper in $\m$. Note that $u \neq \bot$ (minimal elements are proper). Let $I_0 \coloneqq  [u, \infty)_{I}$ and let $\m_0 \coloneqq  \m \, {\rest} \, I_0$. Note that:

    \begin{enumerate}
    \item[(A)] $I_0 \subseteq I$ and $I_0 \times \{0,1\} \in \Iis$ (by (\ref{multi-1}) in the definition of a multidimensional independence relation), hence also $I_0 \in \Iis$ (see \ref{closed-def}, it is isomorphic to an initial segment of $I_0 \times \{0,1\}$).
    
    \item[(B)] $I \backslash I_0$ is an initial segment of $I$, hence is in $\Iis$.
    
    \item[(C)] $k (\m_0) < k (\m)$, since $u$ is minimal in $I_0$, hence, since $|\K^\ast| = |\K|$, $\K^\ast$-proper in $\m_0$ (we are also using that properness is preserved when passing to subsystems, i.e.\ Lemma \ref{proper-subsys}).
    \end{enumerate}

    Let $\m_0'$ be an $\is^\ast$-independent proper $I_0$-system such that $\m_0 \leap{\is} \m_0'$. 
    
    [Why? Let $\langle v_{\ell}: \ell < \vert I \vert \rangle$ lists $I_{0}$ with no repetition such that $v_{\ell} <_{I} v_{k} \Rightarrow \ell < k$ and for $k \leq \vert I \vert,$ let $I_{k} \coloneqq  I \, {\rest} \, \{ v_{\ell}: \ell < k \}.$ Now by induction on $k \leq \vert I \vert$ we choose $\mathbf{b}_{k}'$ such that: 
    
    \begin{enumerate}
        \item[(a)] $\mathfrak{b}_{k}'$ is a $(J_{k} \times \{0, 1\}, \iota)$-system, 
        
        \item[(b)] $\mathfrak{b}_{k}' \, {\rest} \, J_{k} \times \{ u \}$ is isomorphic to $\mathfrak{b}_{k} \, {\rest} \, J_{k}$ over $\pi: J_{k} \times \{ 0 \} \cong J_{k},$
        
        \item[(c)] $\mathbf{m}_{k}' \, {\rest} \, (J_{k} \times \{ 1 \})$ is $\iota^{\ast}$-independent from $I_{k} \times \{ 1 \}$-system.
    \end{enumerate}
    
    We can carry the induction the induction hypothesis and for $k = \vert I \vert$ we fulfill \ the promise above.] 

    Write $\m_0' = \seq{N_u : u \in I_2}$. For $u \in I$, let $M_u' \coloneqq  N_u$ if $u \notin I_0$ and $M_u' \coloneqq  N_{u}$ otherwise. Let $\m' \coloneqq  \seq{M_u' : u \in I}$. By monotonicity 5, $\m'$ is $\is$-independent. Moreover, $k(\m') < k (\m)$ because $u$ is proper in $\m'$. By the induction hypothesis, $\m'$ must be an extension base. Let $\m''$ be an independent $J$-system such that $v$ is proper in $\m''$ and there is  $f: \m' \cong \m'' \, {\rest} \, I$. Let $\m'''$ be the unique $J$-system so that $\m''' \, {\rest} \, \{v\} = \m'' \, {\rest} \, \{v\}$  and $f \, {\rest} \, \m \cong \m''' \, {\rest} \, I$. By monotonicity 5, $\m'''$ is independent, and hence witnesses that $\m$ is an extension base.
\end{proof}

This implies a way to obtain extension from existence, provided that the skeleton satisfies strong uniqueness (in case $\K_{\is^\ast}$ consists of limit models, ordered by being limit over, it turns out that in reasonable cases strong uniqueness will hold in $\is^\ast$, though there is no hope of it holding in $\is$).

\begin{cor}\LABEL{strong-uq-ext}
  Let $\is^\ast$ be a skeleton of $\is$ with $|\K_{\is^\ast}| = |\K_{\is}|$. If $\is^\ast$ has strong uniqueness and existence, then $\is$ has extension.
\end{cor}
\begin{proof}
  By Theorem \ref{brim-ext-thm}, it suffices to show that every proper $\is^\ast$-independent $I$-system $\m$ is an extension base. So let $J = I \cup \{v\} \in \Iis$ with $I < v$. By existence, there is a proper $\is^\ast$-independent $J$-system $\m^\ast$. By strong uniqueness, $m^\ast \, {\rest} \, I$ is isomorphic to $\m$, and clearly $\m^\ast \, {\rest} \, I$ is an extension base, so $\m$ must be one as well.
\end{proof}

We now turn to studying under what conditions uniqueness in a skeleton implies uniqueness in the original independence relation. We are unable to prove the full analog of Theorem \ref{brim-ext-thm}, so we make the additional assumption that $\Iis$ is closed under certain products. The following is the key lemma:

\begin{lem}\LABEL{brim-uq-lem}
  Let $\is^\ast$ be a skeleton of $\is$. Assume that $\is$ has extension. Let $J \in \Iis_{\is}$ be such that $J = I \cup \{v\}$, for $I < v$. If $J \times \{0,1\} \in \Iis_{\is}$ and any proper $\is^\ast$-independent $I$-system is a uniqueness base, then any $\is$-independent $I$-system is a uniqueness base.
\end{lem}
\begin{proof}
  Let $\m = \seq{M_u : u \in I}$ be an $\is$-independent $I$-system. Let $\m_1, \m_2$ both be $\is$-independent $J$-systems, $\m_\ell = \seq{M_u^\ell : u \in J}$, $\ell = 1,2$ and without loss of generality $\m_1 \, {\rest} \, I = \m_2 \, {\rest} \, I = \m$. We want to find an amalgam of $M_v^1$ and $M_v^2$ fixing $\m$.
  
  Let $J' \coloneqq  J \times \{0,1\}$, $I' \coloneqq  I \times \{0, 1\}$. Identify $J$ with $J \times \{0\}$. By Lemma \ref{skel-ext}, we can find an $I'$-independent system $\m'$ such that $\m' \, {\rest} \, I = \m$ which is $(I' \backslash I)$-$\K^\ast$-proper. For $\ell = 1,2$, let $\m_\ell'$ be the $I' \cup \{(v, 0)\}$-system obtained by adding $M_v^\ell$ to $\m'$. By monotonicity 2 and 3, $\m_\ell'$ is still independent. By extension and some renaming, one can find a $J'$-independent system $\m_\ell''$ such that $\m_\ell'' \, {\rest} \, (I' \cup \{(v, 0)\}) = \m_\ell'$. Moreover, we can arrange that $\m_{\ell}''$ is $(J' \backslash J)$-$\K^\ast$-proper. Now let $J^\ast \coloneqq  J \times \{1\}$, $I^\ast \coloneqq  I \times \{1\}$. We have that $\m^\ast \coloneqq  \m_1'' \, {\rest} \, I^\ast = \m_2'' \, {\rest} \, I^\ast$. Moreover for $\ell = 1,2$, $\m_\ell^\ast \coloneqq  \m_\ell'' \, {\rest} \, J^\ast$ is a proper $\is^\ast$-independent system by construction. Thus it is a uniqueness base, so $\m_1^\ast$ and $\m_2^\ast$ amalgamate over $\m^\ast$. This implies that $M_v^1$ and $M_v^2$ amalgamate over $\m$, as desired.  
\end{proof}

\begin{thm}\LABEL{brim-uq-thm}
  Assume that $\Iis_{\is}$ is closed under products. Let $\is^\ast$ be a skeleton of $\is$ with $|\K_{\is}| = |\K_{\is^\ast}|$. If $\is^\ast$ has strong uniqueness and existence, then $\is$ has uniqueness and extension.
\end{thm}
\begin{proof}
  By Corollary \ref{strong-uq-ext}, $\is$ has extension. We have to see that for every $I \in \Iis$, every $\is$-independent $I$-system is a uniqueness base but this is immediate from Lemma \ref{brim-uq-lem}.
\end{proof}

The following more local result will be used in the proof of Theorem \ref{kprop-ap}. It says roughly that if a uniqueness base is changed by making one of the models on the boundary smaller, then it is still a uniqueness base. A version of it appears in \cite[III.12.26]{shelahaecbook}.

\begin{lem}\LABEL{uq-monot-lem}
  Let $\m = \seq{M_u : u \in I}$ be an independent $I$-system. Let $u^\ast \in I$ be maximal. Let $N_{u^\ast} \in \K$ be such that $M_{u^\ast} \lea N_{u^\ast}$. Let $\m^\ast \coloneqq  \m \, {\rest} \, (I \backslash \{u^\ast\}) \smallfrown N_{u^\ast}$. If $\nf (\is)$ has extension and $\m^\ast$ is a uniqueness base, then $\m$ is a uniqueness base.
\end{lem}
\begin{proof}
  Let $J = I \cup \{v\} \in \Iis$, $I < v$. Let $\m_\ell = \seq{M_u^\ell : u \in J}$ for $\ell = 1,2$ be independent systems with $\m_\ell \, {\rest} \, I = \m$ for $\ell = 1,2$. We want to amalgamate $M_v^1$ and $M_v^2$ over $\m$. Let $\nf \coloneqq  \nf (\is)$.

  Using the extension property for $\nf$, find $f_\ell, N_v^\ell$, $\ell = 1,2$ such that $f_\ell : N_{u^\ast} \xrightarrow[M_{u^\ast}]{} N_v^\ell$ and $\nfs{M_{u^\ast}}{M_v^\ell}{f_\ell[N_{u^\ast}]}{N_v^\ell}$. Let $\m_\ell' \coloneqq  \seq{(M_u^\ell)' : u \in J}$ be defined by $(M_u^\ell)' \coloneqq  M_u$ if $u \in I \backslash \{u^\ast\}$, $(M_{u^\ast}^\ell)' \coloneqq  f_\ell[N_{u^\ast}]$, and $(M_v^\ell)' \coloneqq  N_v^\ell$. By monotonicity 4, $\m_\ell'$ is an independent system.
  
  Let $f \coloneqq  f_2 f_1^{-1} \, {\rest} \, (M_{u^\ast}^1)'$. Then $f: \m_1' \, {\rest} \, I \cong \m_2' I$ and since $\m_1' \cong \m^\ast$ and $\m^\ast$ is a uniqueness base, we know that there exists $g$ extending $f$ amalgamating $N_v^1$ and $N_v^2$. Such an $f$ also amalgamates $M_v^1$ and $M_v^2$ over $\m$, as desired.
\end{proof}

%%%%%%%%%%%%%%%%%%%%%%%%%%%%%%%%%%%%%%%%%%%%%%
\newpage 

\section{The multidimensional amalgamation properties}\LABEL{multi-ap-sec}

Throughout this section, we continue to assume:

\begin{hypothesis}\LABEL{t4}
  $\is = (\K, \Iis, \NF)$ is a multidimensional independence relation.
\end{hypothesis}

We study the existence and uniqueness properties restricted to systems indexed by $\Ps (n)$, for some $n < \omega$. We show (Lemma \ref{ps-ext-lem}, Theorem \ref{ext-base-thm}) that once we have those, we can (under reasonable conditions), obtain the properties for systems indexed by other sets. We also give a relationship between extension and existence using strong uniqueness, akin to Corollary \ref{strong-uq-ext} but stronger: we only use strong uniqueness for smaller systems. This is Theorem \ref{n-dim-ext-uq}. Finally, we give sufficient conditions for the $n$-dimensional properties in terms of amalgamation properties inside classes of $n$-dimensional systems. The most important of these result is Theorem \ref{kprop-ap}.

It will be convenient to restrict oneself to systems indexed by finite initial segments of $\Ps (\omega)$. Independence relations considering only such systems will be called \emph{okay}. We then define the $n$-dimensional properties by saying that they must hold for systems indexed by initial segments of $\Ps (n)$.

\begin{defin}\LABEL{okay-def} \
    \begin{enumerate}
    \item For $n < \omega$, let $\Iis_n$ be the class of all partial orders isomorphic to an initial segment of $\Ps (n)$. Let $\Iis_{<\omega} \coloneqq  \bigcup_{n < \omega} \Iis_n$.\myindex{$\Iis_n$}\myindex{$\Iis_{<\omega}$}
    \item We say that $\is$ is \emph{okay} if $\Iis \subseteq \Iis_{<\omega}$. We say that $\is$ is \emph{$n$-okay} if it is okay and $\Ps (n) \in \Iis_{\is}$.\myindex{okay}\index{$n$-okay|see {okay}}
  \item For $P$ a property from Definition \ref{multidim-props} and $n < \omega$, we say that $\is$ has \emph{$n$-$P$} if $\is$ is $n$-okay and $\is \, {\rest} \, \Iis_n$ has $P$. We say that $\is$ has \emph{$(<n)$-$P$} if it has $m$-$P$ for all $m < n$.\myindex{$n$-$P$} 
  \item For $P$ a property, $\lambda$ an infinite cardinal and $n < \omega$, we say that $\is$ has \emph{$(\lambda, n)$-$P$} if $\is \, {\rest} \, \K_\lambda$ has $n$-$P$. We say that $\is$ has \emph{$(\lambda, <n)$-$P$} if it has $(\lambda, m)$-$P$ for all $m < n$. Similarly define variations like the $(\Theta, n)$-$P$, for $\Theta$ an interval of cardinals.\myindex{$(\lambda,n)$-$P$}
  \end{enumerate}
\end{defin}

\begin{remark}\LABEL{t6}
  Since $\is$ is closed under initial segments, if $\is$ is $n$-okay then $\Iis_n \subseteq \Iis$, so $\is$ is $m$-okay for any $m < n$.
\end{remark}

Note that $\is \, {\rest} \, \Iis_{<\omega}$ is always an okay independence relation, so replacing $\is$ by $\is \, {\rest} \, \Iis_{<\omega}$, we assume for the rest of this section that $\is$ is okay:

\begin{hypothesis}
  $\is$ is \emph{okay}.
\end{hypothesis}

It is easy to characterize the low-dimensional amalgamation properties in terms of familiar properties: 

\begin{lem}\LABEL{basic-small-dim} \
  \begin{enumerate}
  \item $\is$ has $0$-existence if and only if $\is$ has $0$-extension if and only if $\K \neq \emptyset$.
  \item $\is$ has strong $0$-uniqueness if and only if $\K$ has at most one model up to isomorphism.
  \item $\is$ has $0$-uniqueness if and only if $\K$ has joint embedding.
  \item $\is$ has $1$-extension if and only if $\is$ is $1$-okay, $\K$ is not empty, and $\K$ has no maximal models.
  \item $\is$ has $1$-uniqueness if and only if $\is$ is $1$-okay and $\K$ has amalgamation and joint embedding.
  \item If $\is$ has $2$-extension, then $\K$ has disjoint amalgamation.
  \end{enumerate}
\end{lem}
\begin{proof}
Immediate.
\end{proof}

Interestingly, strong $1$-uniqueness together with uniqueness already implies strong uniqueness. In case we are considering classes of limit models ordered by being limit over, strong $1$-uniqueness is exactly the uniqueness of limit models, a key property studied in many papers \cite{shvi635, vandierennomax, nomaxerrata, gvv-mlq, vandieren-symmetry-apal}.

\begin{lem}\LABEL{strong-uq-equiv}
  If $\is$ has strong $1$-uniqueness and uniqueness, then $\is$ has strong uniqueness.
\end{lem}
\begin{proof}
  We use Lemma \ref{base-equiv}. Let $\m_\ell = \seq{M_u^\ell : u \in J}$, $\ell = 1,2$ be independent $J$-systems. Let $v$ be a top element of $J$ and let $I \coloneqq  J \backslash \{v\}$. Assume that $\m \coloneqq  \m_1 \, {\rest} \, I = \m_2 \, {\rest} \, I$ and $v$ is proper in both $\m_1$ and $\m_2$. We show that $M_v^1$ and $M_v^2$ are isomorphic over $\m$. By definition of being proper, we can pick $N_v^\ell \lta M_v^\ell$ so that $M_u^\ell \lea N_v^\ell$, $\ell = 1,2$.
  By uniqueness, there exists $M \in \K$ and $f_\ell : N_v^\ell \rightarrow M$ with $f_\ell$ fixing $\m$, $\ell = 1,2$. If $|J| \le 2$, there is nothing to prove so assume that $|J| > 2$. Then $M_v^1$ must be a proper extension of $M_u^1$ for each $u < v$. By strong $0$-uniqueness, $\K$ must have no maximal models. Thus we can pick $N \in \K$ with $M \lta N$. By strong $1$-uniqueness and some renaming there exists $g_\ell : M_v^\ell \cong N$, with $g_\ell$ extending $f_\ell$, $\ell = 1,2$. In the end, $g \coloneqq  g_2^{-1} g_1$ is the desired isomorphism.
\end{proof}

We now want to show that it is enough to prove that $\Psm (n)$-systems are extension bases to get $n$-extension. To this end, the following invariance of a semilattice will be useful: 

\begin{defin}\LABEL{t15}\myindex{$n (I)$}\myindex{$n (I, u)$}
For $I \in \Iis_{<\omega}$ and $u \in I$, let $n (I, u)$ be the cardinality of $f(u)$, where $f : I \rightarrow \Ps (\omega)$ is any semilattice embedding such that $f[I]$ is an initial segment of $\Ps (\omega)$. Let $n (I)$ be the maximum of $n (I, u)$ for all $u \in I$. When $I$ is empty, we specify $n (\emptyset) = -1$.
\end{defin}

\begin{lem}\LABEL{t18}
  Let $n < \omega$ and $I, J \in \Iis_{<\omega}$ be non-empty.
  
  \begin{enumerate}
  \item If $I \subseteq J$, then $n (I) \le n (J)$.
  \item $n (\Ps (n)) = n$.
  \item $n (\Psm (n)) = n - 1$.
  \item $n (I \times \{0, 1\}) = n (I) + 1$.
  \item If $I \in \Iis_n$, then $n (I) \le n$.
  \end{enumerate}
\end{lem}
\begin{proof}
  Straightforward.
\end{proof}

\begin{lem}\LABEL{ps-ext-lem}
  Let $I, J \in \Iis_{<\omega}$ and let $I$ be an initial segment of $J$. Let $n \coloneqq  n (J)$.

  \begin{enumerate}
  \item If for every $m \le n$, every proper $\Psm (m)$-system is an extension base, then every proper independent $I$-system can be extended to a proper independent $J$-system.
  \item If for every $m \le n$, every proper $\Psm (m)$-system is a strong uniqueness base, then for any proper independent $J$-systems $\m_1, \m_2$ and any isomorphism $f: \m_1 \, {\rest} \, I \cong \m_2 \, {\rest} \, I$, there is an isomorphism $g : \m_1 \cong \m_2$ extending $f$.
  \end{enumerate}
\end{lem}
\begin{proof}
  Identify $J$ with an initial segment of $\Ps (\omega)$.
  
  \begin{enumerate}
  \item Let $\m$ be a proper independent $I$-system. Let $u \in J \backslash I$ be such that $J_0 \coloneqq  I \cup \{u\}$ is an initial segment of $J$. It is enough to show that $\m$ can be extended to a proper independent $J_0$-system. Let $m \coloneqq  |u|$. Note that $m \le n$. Let $\pi : u \rightarrow m$ be a bijection. Since $J_0$ is an initial segment of $\Ps (\omega)$, We have that $\pi$ induces an isomorphism from $(-\infty, u] = \Ps (u)$ onto $\Ps (m)$. In particular, $\m \, {\rest} \, \Psm (u)$ is an extension base. Now extend it to $\Ps (u)$ and use monotonicity to argue that this induces an extension of $\m$ to $J_0$.
\item Similar.
  \end{enumerate}
\end{proof}

\begin{thm}\LABEL{ext-base-thm}
  Let $n < \omega$ and assume that $\is$ is $n$-okay. If for every $m \le n$, every proper $\Psm (m)$-system is an extension base, then $\is$ has $n$-extension.
\end{thm}
\begin{proof}
  Directly from Theorem \ref{brim-ext-thm} and Lemma \ref{ps-ext-lem}.
\end{proof}

The following result will be useful to obtain extension from existence and strong uniqueness for lower-dimensional systems:

\begin{thm}\LABEL{n-dim-ext-uq}
  Let $n < \omega$ and assume that $\is$ is $n$-okay. If $\is$ has strong $(<n)$-uniqueness, then $\is$ has $n$-existence if and only if $\is$ has $n$-extension.
\end{thm}
\begin{proof}
  If $\is$ has $n$-extension, then by Lemma \ref{ext-implies-exi} it always has $n$-existence (this does not use strong uniqueness). Assume now that $\is$ has $n$-existence. By Theorem \ref{ext-base-thm}, it suffices to show that for every $m \le n$, every proper $\Psm (m)$-system is an extension base. Let $m \le n$ and let $\m$ be a proper $\Psm (m)$-system. By $n$-existence, there is a proper $\Ps (m)$-system $\m'$. By strong $(<n)$-uniqueness and Lemma \ref{ps-ext-lem}, $\m' \, {\rest} \, \Psm (m) \cong \m$. Since $\m' \, {\rest} \, \Psm (m)$ is an extension base, so is $\m$.
\end{proof}

We would like to establish a criteria to obtain uniqueness from a kind of amalgamation in the class $\Kprop_{\is, \Ps (n)}$ (Definition \ref{proper-class}). It turns out that a slightly finer version of $\Kprop_{\is, \Ps (n)}$ is useful for this purpose: the idea is that we require the interior of the system to be limit but allow the boundary to not extend the rest in limit way (of course this is formulated abstractly using skeletons).

\begin{defin}\LABEL{kprop2-def}\myindex{$\Kprop_{\is, \is^\ast, I}$}\ 

  (1) Let $n < \omega$ and assume that $\is$ is $n$-okay. Let $\is^\ast$ be a skeleton of $\is$. Let $K^\ast$ be the class of proper $\is$-independent $\Ps (n)$-systems $\m$ such that $\m \, {\rest} \, \Psm (n)$ is $\is^\ast$-independent. Let $\Kprop_{\is, \is^\ast, \Ps (n)}$ be the abstract class whose underlying class is $K^\ast$ and whose ordering is $\m_1 \leap{\Kprop_{\is, \is^\ast, \Ps (n)}} \m_2$ if and only if $\m_1 = \m_2$ or whenever $\m_\ell = \m_\ell^- \smallfrown M_\ell$ with $\m_\ell^- \coloneqq  \m_\ell \, {\rest} \, \Psm (n)$, $\ell = 1,2$, we have that $\m_1 \leap{\is} \m_2$, $M_1$ and $M_2$ are in $\K_{\is^\ast}$, $M_1 \ltap{\K_{\is^\ast}} M_2$, and $\m_1 \, {\rest} \, \Psm (n) \ltap{\Kprop_{\is^\ast, \Psm (n)}} \m_2 \, {\rest} \, \Psm (n)$. 
  
  (2) Assume $I \subsetneq J$ are finite initial segments of $\mathcal{P}(\omega), n \coloneqq  n(J)$ and $I = \{ u \in J: \vert u \vert < n \}.$ We repeat the definition in part (1) with $I, J$ here playing the roles of $\mathcal{P}^{-}(n), \mathcal{P}(n).$
\end{defin}

\begin{remark}\LABEL{t21}
  We have that $\Kprop_{\is^\ast, \Ps (n)} \subseteq \Kprop_{\is, \is^\ast, \Ps (n)} \subseteq \Kprop_{\is, \Ps (n)}$. In particular, $\Kprop_{\is, \is, \Ps (n)} = \Kprop_{\is, \Ps (n)}$.
\end{remark}

\begin{defin}\LABEL{t25}\myindex{$\phi_n$}
  For $n < \omega$, we let $\phi_n (x)$ be the formula in the vocabulary of systems (Definition \ref{sys-voc-def}) which holds inside a $\Ps (n)$-system $\seq{M_u : u \in \Ps (n)}$ if and only if $a \in \bigcup_{u \subsetneq n} M_u$.  That is, $\phi_{n}(x)$ is $\bigvee_{u \subsetneq n}  P_{u}(x),$ where $P_{u}$ is the predicate corresponding to $M_{u}.$ 
\end{defin}

The following easy relationships hold between the $n$-dimensional properties and classes of proper $n$-dimensional systems. See Definition \ref{phi-props-def} for the definitions of $\phi$-categoricity, $\phi$-amalgamation, $\phi$-uniqueness, etc.    

\begin{lem}\LABEL{kprop-skel}
  Let $n < \omega$. Assume that $\is$ is $(n + 1)$-okay and let $\is^\ast$ be a skeleton of $\is$.
  
  \begin{enumerate}
  \item For any $\m_1 \in \Kprop_{\is, \is^\ast, \Ps (n)}$, there exists $\m_2 \in \Kprop_{\is^\ast, \Ps (n)}$ such that $\m_1 \leap{\K_{\is}} \m_2$ and $\m_1 \, {\rest} \, \Psm (n) = \m_2 \, {\rest} \, \Psm (n)$.

  \item\label{kprop-skel-categ} If $\is^\ast$ has strong $(<n)$-uniqueness, then $\Kprop_{\is, \is^\ast, \Ps (n)}$ is $\phi_n$-categorical.
  \item\label{kprop-skel-univ} If $\is^\ast$ has strong $n$-uniqueness, then any $\m \in \Kprop_{\is^\ast, \Ps (n)}$ is $\phi_n$-universal in $(|\Kprop_{\is, \is^\ast, \Ps (n)}|, \leap{\K_{\is}}^d)$.
  \item\label{kprop-skel-uq} If $\is^\ast$ has strong $n$-uniqueness, then $\Kprop_{\is, \is^\ast, \Ps (n)}$ has $\phi_n$-uniqueness.
  
  \item\label{kprop-skel-ap} Let $\bar{\m} = (\m_0, \m_1, \m_2)$ be a $\phi_{n}$-span in $\Kprop_{\is, \is^\ast, \Ps (n)}$, and write $\m_\ell = \seq{M_u^\ell : u \in \Ps (n)}$, $\ell = 0,1,2$. Then $\bar{\m}$ has a $\phi_{n}$-amalgam in $\Kprop_{\is, \is^\ast, \Ps (n)}$ if and only if there exists $N \in \K_{\is}$ and $f_\ell : M_{\Ps (n)}^\ell \xrightarrow[\m_0]{} N$, $\ell = 1,2$, such that $f_1 \, {\rest} \, M_u^1 = f_2 \, {\rest} \, M_u^2$ for all $u \in \Psm (n)$.

  \item\label{kprop-skel-5} If $\m$ is a $\phi_n$-amalgamation base in $\Kprop_{\is, \is^\ast, \Ps (n)}$, then for any $\m' \in \Kprop_{\is, \is^\ast, \Ps (n)}$ with $\m \leap{\Kprop_{\is, \is^\ast, \Ps (n)}} \m'$, we have that $(\m \ast \m') \, {\rest} \, \Psm (n + 1)$ is a uniqueness base in $\is$.
  \end{enumerate}
\end{lem}
\begin{proof} \
  \begin{enumerate}
  \item By the basic properties of a skeleton.
  \item Since $\is^\ast$ has strong $(<n)$-uniqueness, any two $\is^\ast$-independent $\Psm (n)$-systems are isomorphic.
  \item By the first part and strong $n$-uniqueness.
  \item Expanding the definitions, it suffices to show that if $\m_1$, $\m_2$ are two $\Psm (n + 1)$-systems such that $\m_1 \, {\rest} \, \Psm (n) = \m_2 \, {\rest} \, \Psm (n)$, then there is an isomorphism $f: \m_1 \cong \m_2$ fixing $\m_1 \, {\rest} \, \Psm (n)$. This is implied by strong $n$-uniqueness, because $n (\Psm (n + 1)) = n$.
  \item By monotonicity.
  \item Expand the definitions.
  \end{enumerate}
\end{proof}

The following result (used in the proof of Lemma \ref{uq-lambda-step}) will be crucial: under the assumption of density of certain amalgamation bases in the class of $(n + 1)$-dimensional systems, it implies that the \emph{whole} independence relation has $(n + 1)$-uniqueness. More precisely, we also assume that $n \ge 2$, work inside a skeleton, and assume that the skeleton has strong $n$-uniqueness and $(n + 1)$-extension (in the case of interest, these will hold).

\begin{thm}\LABEL{kprop-ap}
  Let $n \in [2, \omega)$. Let $\is$ be a multidimensional independence notion and let $\is^\ast$ be a skeleton of $\is$. If:
  
    \begin{enumerate}
    \item[(A)] $\is^\ast$ has $(n + 1)$-extension and strong $n$-uniqueness.
    
    \item[(B)] For any $\m_0 \in \Kprop_{\is^\ast, \Ps (n)}$, there exists a $\phi_n$-amalgamation base $\m_1 \in \Kprop_{\is, \is^\ast, \Ps (n)}$ such that $\m_0 \leap{\Kprop_{\is, \is^\ast, \Ps (n)}} \m_1$.
  \end{enumerate}

    Then $\is^\ast$ has $(n + 1)$-uniqueness.
\end{thm}
\begin{proof}
  By Theorem \ref{brim-ext-thm}, $\is$ has $(n + 1)$-extension. By Lemma \ref{brim-uq-lem}, the proof of Lemma \ref{strong-uq-equiv}, and Lemma \ref{ps-ext-lem}, it suffices to show that every proper $\is^\ast$-independent $\Psm (n + 1)$-system is a uniqueness base(essentially we are reducing uniqueness to uniqueness for proper systems). Note that $n (\Psm (n + 1)) = n$, so by Lemma \ref{ps-ext-lem} again, any two proper $\is^\ast$-independent $\Psm (n + 1)$-systems are isomorphic. Thus it suffices to show that \emph{some} proper $\is^\ast$-independent $\Psm (n + 1)$-system is a uniqueness base.

  Pick any $\m_0 \in \Kprop_{\is^\ast, \Ps (n)}$ (exists by extension in $\is^\ast$). By assumption, there is a $\phi_n$-amalgamation base $\m_1$ in $\Kprop_{\is, \is^\ast, \Ps (n)}$ which extends $\m_0$. By $(n + 1)$-extension, there exists $\m_2 \in \Kprop_{\is, \is^\ast, \Ps (n)}$ such that $\m_1 \ltap{\Kprop_{\is, \is^\ast, \Ps (n)}} \m_2$. By Lemma \ref{kprop-skel}(\ref{kprop-skel-5}), $\m \coloneqq  (\m_1 \ast \m_2) \, {\rest} \, \Psm (n + 1)$ is a uniqueness base (buy may not be proper, so we have to work more). Write $\m = \seq{M_u : u \in \Psm (n + 1)}$ and let $\m' = \seq{M_u' : u \in \Psm (n + 1)}$ be defined as follows: let $I_0 \coloneqq  \Psm (n + 1) \backslash \{n\}$, $\m' \, {\rest} \, I_0 \coloneqq  (\m_0 \ast \m_2) \, {\rest} \, I_0$, $M_n' = M_n$. It is easy to check that $\m'$ is an independent system and a uniqueness base as well: anytime a model $N$ extends all the elements of $\m'$, it extends all the elements of $\m$. 

  Now let $\m^\ast$ be a $\Psm (n + 1)$-system defined as follows: $\m^\ast \, {\rest} \, I_0 = \m' \, {\rest} \, I_0$ and $\m^\ast \, {\rest} \, \{n\} = \m_0 \, {\rest} \, \{n\}$. Note that then $\m^\ast \, {\rest} \, \Ps (n) = \m_0$, and it is easy to check that $\m^\ast$ is a proper $\is^\ast$-independent system. Moreover, $\m^\ast \, {\rest} \, \{n\} \lea \m' \, {\rest} \, \{n\}$ by definition of $\m'$, so by Lemma \ref{uq-monot-lem}, $\m^\ast$ is a uniqueness base, as desired.
\end{proof}

%%%%%%%%%%%%%%%%%%%%%%%%%%%%%%%%%%%%%%%%%%%%%%
\newpage 

\section{Continuous independence relations}\LABEL{multi-cont-sec}

This section is central. We define here the conditions a multidimensional independence relation should satisfy to be well-behaved with respect to increasing chains. In particular, it should be closed under union of those. We call such multidimensional independence relations \emph{very good}. They are the ones we really want to study. We prove the \emph{construction theorem} (Theorem \ref{constr-thm}), which says that from a very good two-dimensional independence relation, we can build a very good multidimensional independence relation.

We study an important subclass of independent systems: the \emph{limit} ones (Definition \ref{brim-sys-def}). We show that they are quite well-behaved under unions and resolvability (Lemma \ref{isbrim-lem}). There are some technical issues here, since we do not know that the class of $n$-dimensional limit systems is closed under unions until we have proven $(n + 1)$-dimensional uniqueness, but we want to use the former in the proof of the latter. We use the class $\Kprop_{\is, \is^\ast, \Ps (n)}$ (Definition \ref{kprop2-def}) to get around this.

We prove the $n$-dimensional properties by studying the class of limit systems and using the weak diamond (or a strongly compact cardinal). Crucial is the \emph{stepping up lemma} (Lemma \ref{uq-lambda-step}) which shows how to go from the $n$-dimensional properties to the $(n + 1)$-dimensional properties (for limit systems). It leads directly to the \emph{limit excellence theorem} (Theorem \ref{excellence-lem}) giving sufficient conditions for the $n$-dimensional properties. Of course, once we have the properties for limit systems we will be able to get them for regular systems using the results of Section \ref{multi-finite-sec}.

The first definition is the multidimensional analog of Definition \ref{very-good-twodim-def}.

\begin{defin}\LABEL{o3}\myindex{strong continuity (for a multidimensional independence relation)}
  We say that a multidimensional independence relation $\is$ has \emph{strong continuity} if whenever $\seq{\m_i : i \le \delta}$ is a $\lea$-increasing continuous chain of $(I, \K_{\is})$-systems and $\m_i$ is independent for all $i < \delta$, then $\m_\delta$ is independent.
\end{defin}

\begin{defin}\LABEL{very-good-def}\myindex{very good multidimensional independence relation}
  Let $\is$ be a multidimensional independence relation. We say that $\is$ is \emph{very good} if:

        \begin{enumerate}
        \item[(A)] $\is$ is $n$-okay for all $n < \omega$ (so $\Iis_{\is} = \Iis_{<\omega}$), see Definition \ref{okay-def}.
        
        \item[(B)] $\nf (\is)$ is very good.
        
        \item[(C)] $\is$ has strong continuity.
        
        \item[(D)] For any $I \in \Iis$, both $\K_{\is, I}$ and the class $\K^\ast$ of independent $I$-system ordered by $\lea^d$ (see Definition \ref{sys-ext-def}(\ref{sys-ext-disj})) are fragmented AECs. Moreover the AECs they generate have the same models\footnote{This is a multidimensional analog of reflecting down (Definition \ref{reflects-down-def}).}.
        \end{enumerate}
\end{defin}

Very good multidimensional independence relations are obtained by taking a very good two-dimensional independence notion and applying the construction described in Definition \ref{is-nf-def} (and proven to have some good properties in Theorem \ref{constr-thm-0}).

\begin{thm}[The construction theorem]\LABEL{constr-thm}
  If $\nf$ is a very good two-dimensional independence notion on $\K$, then\footnote{Note that $\mathcal{I}_{\is(\nf)}$ is a class of latices, whereas $\mathcal{I}$ is a countable set of latices included in it.} $\is (\nf) \, {\rest} \, \Iis_{<\omega}$ (see Definition \ref{okay-def}) is a very good multidimensional independence notion.
\end{thm}

\begin{proof}
  Let $\is \coloneqq  \is (\nf) \, {\rest} \, \Iis_{<\omega}$. By Theorem \ref{constr-thm-0}, $\is$ is a multidimensional independence notion. By definition, $\is$ is $n$-okay for all $n < \omega$ and $\nf({\is}) = \nf$ (see the statement of Theorem \ref{constr-thm-0}), which is very good by assumption. It also follows from the definition of $\is$ that it has strong continuity, since $\nf$ has it. It remains to see that for any $I \in \Iis$, both $\K_{\is, I}$ and $\K^\ast$ (the class of independent $I$-systems ordered by $\lea^d$) are fragmented AECs that generate the same models. Now $\K$ is a fragmented AEC because $\nf$ is very good. Thus by strong continuity it follows immediately that $\K^\ast$ is a fragmented AEC, and that $\K_{\is, I}$ is a fragmented AEC similarly follows from the definition of $\nf$. It remains to check that those two classes generate the same class of models. Let $\lambda \coloneqq  \min (\Theta)$, let $\K^1$ be the AEC generated by $\K_{\is, I}$, and let $\K^2$ be the AEC generated by $\K^\ast$. Note that $|\K_\lambda^1| = |\K_\lambda^2|$, and $\leap{\K_\lambda^1} = \leap{(\K_{\is, I})_\lambda}$, $\leap{\K^2} = \lea^d$. Since $\leap{\K^\ast}$ extends $\leap{\K_{\is, I}}$, we have that $|\K^1| \subseteq |\K^2|$. To see the other direction, let $\m \in \K^2$. We claim that there exists $\m_0$ in $\K_\lambda^1$ such that for any subset $A$ of $\m$ of size $\lambda$ there is $\m_1 \in \K_\lambda^1$ with $\m_0 \leap{\K^1} \m_1$, $\m_1 \leap{\K^2} \m$, and $\m_1$ containing $A$. This will suffice by the usual directed system argument. Suppose that the claim fails. We build $\seq{\m_i : i < \lambda^+}$ $\leap{\K^2}$-increasing continuous in $\K_\lambda^1$ such that for all $i < \lambda$, $\m_i \not \leap{\K^1} \m_{i + 1}$. This is possible by failure of the claim. This is enough: we know that $\nf$ is very good, hence reflects down. From the definition of $\is$, we get that there must be a club of $i < \lambda^+$ such that $\m_i \leap{\is, I} \m_{i + 1}$, a contradiction.
\end{proof}

Note that a very good multidimensional independence relation has underlying class a fragmented AEC, which may possibly ``shrink'' (e.g.\ it could be a class of saturated models). The next lemma shows that the AEC generated by this fragmented AEC also carries a multidimensional independence relation (which is quite nice, but may fail to be very good because of the requirement that $\nf (\is)$ be very good, hence has extension and uniqueness):

\begin{defin}\LABEL{almost-vg-def}\myindex{almost very good multidimensional independence relation}
  We call a multidimensional independence relation $\is$ \emph{almost very good} if:

  \begin{enumerate}
  \item[(A)] $\is_{\min (\dom (\is))}$ is very good.
  
  \item[(B)] $\is$ satisfies all the properties in the definition of very good, except that $\nf (\is)$ may not be very good.
  \end{enumerate}
\end{defin}

\begin{lem}\LABEL{upward-ext}
  Let $\is$ be a very good multidimensional independence notion. Let $\lambda \coloneqq  \min (\dom (\K_{\is}))$. Then there exists a (unique) multidimensional independence notion $\is'$ such that:

  \begin{enumerate}
  
  \item[(A)] $\is_{\min(\dom(\is))}$ is very good.
  
  \item[(B)] $\lambda = \min (\dom (\K_{\is'}))$ and $\is_\lambda = \is_\lambda'$.
  
  \item[(C)] $\is'$ is almost very good.
  \end{enumerate}
\end{lem}
\begin{proof}
  Let $\K$ be the AEC generated by $\left(\K_{\is}\right)_\lambda$. More generally, for each $I \in \Iis$, let $\K_I$ be the AEC generated by $\left(\K_{\is, I}\right)_\lambda$. We let $\is'$ be defined as follows: $\K_{\is'} = \K$, $\Iis_{\is'} = \Iis_{\is}$, and for any $(\K, I)$-system $\m$, $\m$ is $\is'$-independent if and only if $\m \in \K_I$. It is a straightforward (but long) to check that $\is'$ is indeed the desired multidimensional independence relation.
\end{proof}

Note that if we have all the $n$-dimensional properties at a fixed cardinal $\lambda$, we can transfer them up. More precisely:

\begin{lem}\LABEL{going-up-excellence}
  Let $\is$ be an almost very good multidimensional independence relation with domain $[\lambda_1, \lambda_2]$ (for $\lambda_1 < \lambda_2$) and let $n < \omega$.

  \begin{enumerate}
  \item If $\is$ has $([\lambda_1, \lambda_2), n + 1)$-extension, then $\is$ has $([\lambda_1, \lambda_2], n)$-extension.
  \item If $\is$ has $([\lambda_1, \lambda_2), n + 1)$-uniqueness and $([\lambda_1, \lambda_2), n + 1)$-extension, then $\is$ has $([\lambda_1, \lambda_2], n)$-uniqueness.
  \end{enumerate}

  In particular, if $\is$ has $(\lambda_1, <\omega)$-extension and $(\lambda_1, <\omega)$-uniqueness, then $\is$ has extension and uniqueness.
\end{lem}
\begin{proof}
  This is quite standard, and the definitions are tailored to make this work, so we do not replay the arguments here. See \cite[\S5]{sh87b} for example. As a warmup, the reader may want to prove that no maximal model and disjoint amalgamation in $\lambda$ imply no maximal models in $\lambda^{+}$ (this is a combinatorial version of $(\lambda, 2)$-extension implying $(\lambda^{+}, 1)$-extension - disjoint amalgamation plays the role of tow-dimensional independence). 
\end{proof}

Limit systems are candidates for strong uniqueness, so are an important object of study. Really they are our main example of proper system and we are now finally shifting to specializing our results to them. 

\begin{defin}\LABEL{brim-sys-def}\myindex{limit system}\myindex{$\isbrim$}
  Let $\is$ be a very good multidimensional independence relation. We define a multidimensional independence relation $\is^{\mathrm{lim}}$ as follows:

  \begin{enumerate}
  \item[(A)] $\K_{\is^{\mathrm{lim}}}$ consist of $\K_{\is}^{\text{lim}}$, the class of limit models in $\K_{\is}$, ordered by being limit over (see Definition \ref{sat-class-def}).
  
  \item[(B)] $\Iis_{\is^{\mathrm{lim}}} = \Iis_{\is}$.
  
  \item[(C)] $\NF_{\is^{\mathrm{lim}}} = \NF_{\is} \, {\rest} \, \K_{\is^{\mathrm{lim}}}$.
  \end{enumerate}

  We call a proper $\isbrim$-independent system a \emph{limit} $\is$-independent system.
\end{defin}

\begin{defin}\LABEL{kprop2a-def}\myindex{$\Kpropa_{\is, \is^\ast, I}$}
  Let $\is$ be a very good multidimensional independence relation and let $\is^\ast$ be a skeleton of $\is$. For $I \in \Iis_{\is}$, we let $\Kpropa_{\is, \is^\ast, I}$ be the sub-abstract class of $\Kprop_{\is, \is^\ast, I}$ (as in Definition \ref{kprop2-def}(2)) with the same ordering but consisting only of the systems where all models have the same cardinality. Similarly define $\Kpropa_{\is^\ast, I}$ to be the sub-abstract class of $\Kprop_{\is^\ast, I}$ (Definition \ref{proper-class}) consisting of systems with all models of the same cardinality.
\end{defin}

The class of limit systems satisfies the following properties:

\begin{lem}\LABEL{isbrim-lem}
  Let $\is$ be a very good multidimensional independence relation with domain $\Theta$.
  
  \begin{enumerate}
  \item $\isbrim$ is a skeleton of $\is$.
  \item\label{isbrim-twodim} For any $\lambda \in \Theta$, $\isbrim$ has strong $(\lambda, 2)$-uniqueness and $(\lambda, 2)$-existence.
  
  \item For any subinterval $\Theta_0$ of $\Theta$, $\left(\is_{\Theta_0}\right)^{\mathrm{lim}} = \left(\is^{\mathrm{lim}} \right)_{\Theta_0}$.
  
  \item\label{isbrim-resolv} For any $I \in \Iis_{\is}$, $\Kpropa_{\isbrim, I}$ is a resolvable abstract class with LST number $\min (\Theta)$.
  \item\label{isbrim-cont} Let $\Theta^- \coloneqq  \{\lambda \in \Theta \mid \lambda^+ \in \Theta\}$. 
  
    \begin{enumerate}
    \item Let $I \in \Iis_{\is}$ and let $n \coloneqq  n (I)$. If for any $\lambda \in \Theta^-$, we have that $\isbrim$ has strong $(\lambda, n + 1)$-uniqueness and $(\lambda, n + 1)$-extension, then $\Kpropa_{\is^\ast, I}$ is $\Theta^-$-continuous.
    
    \item Let $n < \omega$. If for any $\lambda \in \Theta^-$, we have that $\isbrim$ has strong $(\lambda, n)$-uniqueness and $(\lambda, n)$-extension, then $\Kpropa_{\is, \is^\ast, \Ps (n)}$ is $\Theta^-$-continuous.
    \end{enumerate}
  \end{enumerate}
\end{lem}

\vspace{-0.6cm}

\begin{proof}

  (1) Directly from the definitions and the (easy to check) fact that $\K_{\is}^{\mathrm{lim}}$ is a skeleton of $\K_{\is}$.
  
  (2) Because $\is$ has it by definition of very good. That is, any relevant property for $\isbrim$ is such property for $\is;$ and applying the property for $\is$ we get there a solution and it is easy to ``convert'' it to a solution in $\isbrim.$ 
  
  (3) Immediate from the definitions.
  
  (4) We leave it to the reader to verify that the LST number is $\min (\Theta)$. We prove that $\Kpropa_{\isbrim, I}$ is resolvable. We check that the conditions of Lemma \ref{resolv-technical} are satisfied. All the conditions are easy, except for resolvability for successors. So let $\m \in \Kpropa_{\isbrim, I}$ have cardinality $\lambda^+$ and let $\m_0 \leap{\Kpropa_{\isbrim, I}} \m_1$ have cardinality $\lambda$, with $\lambda, \lambda^+ \in \Theta$. Since $\K_{\is, I}$ is a fragmented AEC, it is resolvable, hence there exists $\seq{\m_i : i < \lambda^+}$ a resolution of $\m$ over $\m_0$ in $\K_{\is, I}$. Write $\m_i = \seq{M_u^i : u \in I}$. What are we missing to make this a resolution in $\Kpropa_{\isbrim, I}$? We want the $\m_i$'s to be limit, and in fact we also want $\m_i \ast \m_{i + 1}$ to be limit. From the properties of being limit, it is in fact enough to show that there is a club $C$ of $\lambda^+$ such that $\seq{\m_i : i \in C}$ is an increasing chain in $\Kpropa_{\isbrim, I}$. To get this club, first get a club $C_1$ (using Lemma \ref{club-struct}) such that for $i < j$ both in $C_1$, $M_u^i$ is limit and $M_u^j$ is limit over $M_u^i$ (this is possible since $M_u^{\lambda^+}$ is limit, hence saturated). Next, use Fact \ref{brimmed-reflects} on each pair of models in $\m$ and intersect clubs to get that on a club $C$, for all $i \in C$, $M_v^i$ is limit over $M_u^i$ whenever $u < v$. This is not enough to make $\m_i$ limit, as we need for each $u \in I$ an intermediate $M$ such that $M_u^i$ is limit over $M$ and $M$ is limit over each $M_v^i$, $v < u$. Now by assumption $\m_{\lambda^+}$ is limit, so for each $u \in I$ there exists $M_u^\ast$ with $M_u^\ast$ limit over $M_v^{\lambda^+}$ and $M_u^{\lambda^+}$ limit over $M_u^\ast$. Resolve $\m_{\lambda^+} \smallfrown M_u^\ast$ as before to obtain a club $C_u$ and a resolution $\seq{M_u^{\ast, i} : i \in C_u}$ of $M_u^\ast$ such that $M_u^{\ast, i}$ is limit over $M_v^i$ for each $v < u$ and $M_u^i$ is limit over $M_u^{\ast, i}$. Note that we can further require that for $i < j$ both in $C_u$, $M_u^{\ast, j}$ is limit over $M_u^{\ast, i}$, hence the corresponding part of the system $\m_i \ast m_j$ is also limit. Now intersect all the $C_u$'s and intersect further with $C_1$ to get the desired club.

  (5) \underline{Clause (a)}: let $\seq{\m_i : i < \delta}$ be an increasing chain in $\left(\Kpropa_{\is^\ast, I}\right)_{\Theta^-}$. Let $\m_\delta \coloneqq  \bigcup_{i < \delta} \m_i$. We have that $\m_\delta$ is an independent system, the only issue is to prove that it is limit (that is, that the extensions between the models are limit). Once this is done, it will immediately follow that $\m_\delta$ extends $\m_0$. Write $\m_i = \seq{M_u^i : u \in I}$. By Fact \ref{brimmed-union-fact}, we directly have that $M_v^\delta$ is limit over $M_u^\delta$ for each $u < v$. However, we need more: for a fixed $u \in I$, we want $M_u^\delta$ to be limit over $\bigcup_{v < u} M_v^\delta$. Recall (Definition \ref{sat-defs}(\ref{brimmed-set-def})) that this means that there is $M$ such that $M$ is limit over each $M_v^\delta$, $v < u$, and $M_u^\delta$ is limit over $M$. In fact, it suffices to find $M^\ast$ such that $M$ extends each $M_v^\delta$, $v < u$ and $M_u^\delta$ is limit over $M^\ast$ (then pick $M$ limit over $M^\ast$ such that $M_u^\delta$ is limit over $M$).

  We proceed by induction on $\delta$. So we can assume that $\delta = \cf{\delta}$ and the chain $\seq{\m_i : i < \delta}$ is continuous. Fix $u \in I$ and write $A_i \coloneqq  \bigcup_{v < u} M_v^i$. We build $\seq{M_i : i \le \delta}$ an increasing continuous chain of limit models such that for all $i < \delta$:

  \begin{enumerate}
      \item $M_{i + 1}$ is limit over $M_i$.
      
      \item $A_i \subseteq |M_i|$.
      
      \item $M_u^i$ is limit over $M_i$.
      
      \item $\nfs{M_i}{M_u^i}{M_{i + 1}}{M_u^{i + 1}}$ (where $\nf = \nf (\is)$).
  \end{enumerate}

      This is enough: by Fact \ref{brimmed-union-fact}, $M_\delta$ is as desired. This is possible: the base case uses that $\m_0$ is itself a limit system. At limits, take unions (and use Fact \ref{brimmed-union-fact}). For the successor step, assume we are given $i < \delta$. Let $I_0 \coloneqq  \{v \in I \mid v \le u\}$, and let $J \coloneqq  I_0 \cup \{v^\ast\}$, where $v^\ast$ is a new element such that $v^\ast < u$ and $v < v^\ast$ for all $v \in I$ with $v < u$. 

      Let $\m_i^\ast = \seq{M_v^\ast : v \in J}$ be the $J$-system such that $\m_i^\ast \, {\rest} \, I_0 = \m_i \, {\rest} \, I_0$ and $M_{v^\ast}^\ast = M_i$. It is easy to check that this is a limit independent system. We now use $(n + 1)$-extension to build a limit independent $J$-system $\m^{\ast \ast}$ with $\m^\ast \lea{\Kpropa_{\isbrim, J}} \m^\ast$. We do this in two steps: first we build an $I$-system extending $\m^{\ast} \, {\rest} \, I$ (using extension on $I \times \{0, 1\}$, note that $n (I \times \{0,1\}) = n(I) + 1 = n + 1$), then using two-dimensional extension to complete this to an extension $\m^{\ast \ast}$. Now by strong $(n + 1)$-uniqueness (again applied in two steps), we must have that there is $f:  \m^{\ast \ast} \cong_{\m_i} \m_{i + 1}$. Let $M_{i + 1} \coloneqq  f[M_{v^\ast}^{\ast\ast}]$, where $M_{v^\ast}^{\ast \ast}$ is the $v^\ast$-indexed element of $\m^{\ast \ast}$. 
      
  \underline{Clause (b)}: immediate from the previous part applied to $I = \Psm (n)$.
\end{proof}

As hinted at before, limit models are well-behaved with respect to the $n$-dimensional properties: strong uniqueness is equivalent to uniqueness, and existence is equivalent to extension (assuming lower-dimensional uniqueness).

\begin{lem}\LABEL{strong-uq-brimmed-equiv}
  Let $\is$ be a very good independence relation. Let $\lambda \in \dom (\is)$ and let $n < \omega$. Then:

  (1) $\isbrim$ has strong $(\lambda, n)$-uniqueness if and only if $\isbrim$ has $(\lambda, n)$-uniqueness.
  
  (2) If $\isbrim$ has $(\lambda, <n)$-uniqueness, then $\isbrim$ has $(\lambda, n)$-existence if and only if $\isbrim$ has $(\lambda, n)$-extension.

\end{lem}
\begin{proof} \

  \item[(1)] By Lemma \ref{strong-uq-equiv} and Lemma \ref{isbrim-lem}(\ref{isbrim-twodim}).
  
  \item[(2)] By Theorem \ref{n-dim-ext-uq} and the previous part.
\end{proof}

In compact classes, one can take an ultraproduct of limit systems and it will still be limit. In fact, we will see that the existence of such a uniform extension is very powerful, so we call independence relations with this property \emph{nice}:

\begin{defin}\LABEL{o8}\myindex{limit over (for two systems)}
  Let $\is$ be a very good multidimensional independence relation. For $\m_1, \m_2 \in \K_{\is, I}$, say with $\m_\ell = \seq{M_u^\ell : u \in I}$ we say that $\m_2$ is \emph{limit over $\m_1$} if $\m_1 \leap{\is} \m_2$ and $M_u^2$ is limit over $M_u^1$ for all $u \in I$. 
\end{defin}

\begin{defin}\LABEL{nice-def}\myindex{nice very good multidimensional independence relation}
  A very good multidimensional independence relation $\is$ is \emph{nice} if:

  \begin{enumerate}
  \item[(A)] $\K_{\is}$ is categorical in every $\lambda \in \dom (\is)$.
  
  \item[(B)] For any $n < \omega$ and any two $\m_1$, $\m_2 \in \K_{\is, \Ps (n)}$ with $\m_1 \, {\rest} \, \Psm (n) = \m_2 \, {\rest} \, \Psm (n)$, there exists independent $\m_1', \m_2'$ such that:

    \begin{enumerate}
    \item $\m_\ell'$ is limit over $\m_\ell$.
    \item $\m_1' \, {\rest} \, \Psm (n) = \m_2' \, {\rest} \, \Psm (n)$.
    
    \item If $\m_\ell = \seq{M_u^\ell : u \in \Ps (n)}$, $\m_\ell' = \seq{N_u^\ell : u \in \Ps (n)}$, then $\|M_u^\ell\| = \|N_u^\ell\|$ for all $u \in \Ps (n)$ and $\ell = 1,2$.
    \end{enumerate}
  \end{enumerate}
\end{defin}

Compact AECs have a nice independence notion:

\begin{thm}\LABEL{multidim-compact}
  Let $\K$ be a compact AEC. If $\K$ is categorical in some $\mu > \LS (\K)$, then there exists a \emph{nice} very good multidimensional independence notion $\is$ with $\K_{\is} = \Ksat_{\ge \LS (\K)^{+6}}$.
\end{thm}
\begin{proof}
  By Fact \ref{compact-very-good}, there is a categorical very good frame $\s$ on $\Ksat_{\ge \LS (\K)^{+6}}$. Let $\nf$ be the very good two-dimensional independence notion associated with $\s$. Let $\is \coloneqq  \is (\nf)$. By Theorem \ref{constr-thm}, $\is$ is very good, we only have to check that it is nice. From the definition, $\K_{\is}$ is indeed categorical in all cardinals, so we have to check the second clause in the definition of being nice. Let $\m_1, \m_2$ be as there. Write $\m_\ell = \seq{M_u^\ell : u \in \Ps (n)}$. Using \cite[4.11]{tamelc-jsl}, we can build a $\kappa$-complete ultrafilter $U$ such that for each $u \in \Ps (n)$ and $\ell = 1,2$, the $U$-ultrapower $N_u^\ell$ of $M_u^\ell$ (seen as an extension of $M_u^\ell$) is limit over $M_u^\ell$. Let $\m_\ell' \coloneqq  \seq{N_u^\ell : u \in \Ps (n)}$. We show that $\m_\ell \leap{\is} \m_\ell'$. It is then not difficult to show, using a downward L{\"o}wenheim-Skolem-like argument, that one can take appropriate submodels of the $N_u^\ell$'s to be of the same size as the $M_u^\ell$'s and satisfy the requirements.

    Essentially by \cite{bg-apal, bgkv-apal} (and more precisely by the proof of \cite[6.8]{bgkv-apal}), the two-dimensional independence relation on $\is$ is given by $\kappa$-coheir (see \cite[3.2]{bg-apal}). Thus it follows from the claim before the proof of \cite[8.2]{bg-apal} and some forking calculus that $\m_\ell \leap{\is} \m_\ell'$.
\end{proof}

We now start working toward the key \emph{stepping up lemma}, allowing us to move from a $(\lambda, n+1)$-property to a $(\lambda^{+}, n)$-property. Existence is the easiest: to find a $(\lambda, n + 1)$-system, one can simply resolve a $(\lambda^+, n)$-system. This is formalized by the following lemma:

\begin{lem}\LABEL{ext-lambda-step}
  Let $\is$ be a very good multidimensional independence notion with domain $\{\lambda, \lambda^+\}$. If $\isbrim$ has $(\lambda^{+}, n)$-existence, then $\isbrim$ has $(\lambda, n + 1)$-existence.
\end{lem}

\begin{proof}
  By $(\lambda^+, n)$-existence, pick $\m$ a limit independent $(\lambda^+, \Ps (n))$-system. By Lemma \ref{isbrim-lem}(\ref{isbrim-resolv}), $\Kpropa_{\isbrim, \Ps (n)}$ is resolvable, so there exists a strictly increasing continuous chain $\seq{\m_i : i < \lambda^+}$ in $\Kprop_{\isbrim, \Ps (n)}$ where for each $i < \lambda^+$, $\m_i$ is a limit independent $(\lambda, \Ps (n))$-system, and $\m = \bigcup_{i < \lambda^+} \m_i$. Now by definition of $\leap{\Kprop_{\isbrim, \Ps (n)}}$, we must have that $\m_0 \ast \m_1$ is a limit independent $(\lambda, \Ps (n + 1))$-system, as desired.
\end{proof}

Thus we obtain the existence properties for free if the interval of cardinals we are working in is closed under the successor operation.

\begin{lem}\LABEL{ext-lambda-full}
  Let $\is$ be a very good independence notion with domain $\Theta \coloneqq  [\lambda_1, \lambda_2)$. If $\lambda_2$ is a limit cardinal, then for any $\lambda \in \Theta$, $\isbrim$ has $(\lambda, <\omega)$-existence.
\end{lem}

\begin{proof}
  We prove by induction on $n < \omega$ that for any $\lambda \in \Theta$, $\isbrim$ has $(\lambda, n)$-existence. When $n = 0$, this is part of Lemma \ref{isbrim-lem}. Assume now inductively that $\isbrim$ has $(\lambda, n)$-existence for all $\lambda \in \Theta$. Let $\lambda \in \Theta$. Since $\lambda_2$ is limit, $\lambda^+ \in \Theta$. By the induction hypothesis, $\isbrim$ has $(\lambda^+, n)$-existence. By Lemma \ref{ext-lambda-step} (applied to $\is \, {\rest} \, \{\lambda, \lambda^+\}$), $\isbrim$ has $(\lambda, n + 1)$-existence, as desired.
\end{proof}

We now attack the harder question of obtain $(n + 1)$-dimensional uniqueness from the $n$-dimensional properties. This is done using the weak diamond:

\begin{lem}[The stepping up lemma]\LABEL{uq-lambda-step}
  Let $\lambda_1$ and $\lambda_2$ be infinite cardinals such that $2^{\lambda_1} = 2^{<\lambda_2} < 2^{\lambda_2}$. Let $\is$ be a very good independence relation with domain $\Theta \coloneqq  [\lambda_1, \lambda_2]$. Assume that for all $\lambda \in \Theta$, $\isbrim$ has $(\lambda, n)$-existence and $(\lambda, n)$-uniqueness. Then:

  (1) For all $\lambda \in [\lambda_1, \lambda_2)$, $\isbrim$ has $(\lambda, n + 1)$-existence.
  
  (2) There exists $\lambda \in [\lambda_1, \lambda_2)$ such that $\isbrim$ has $(\lambda, n +  1)$-uniqueness.
  
  (3) If $\is$ is nice and $\isbrim$ has $(\lambda_2, n + 1)$-existence, then $\isbrim$ has $(\lambda, n + 1)$-uniqueness for \emph{all} $\lambda \in [\lambda_1, \lambda_2)$. 
\end{lem}

\begin{proof}\

  (1) By Lemma \ref{ext-lambda-step} applied to $\is \, {\rest} \, \{\lambda, \lambda^+\}$.
  
  (2) By Lemma \ref{strong-uq-brimmed-equiv}, for every $\lambda \in [\lambda_1, \lambda_2]$, $\isbrim$ has strong $(\lambda, n)$-uniqueness and $(\lambda, n)$-extension, and for every $\lambda \in [\lambda_1, \lambda_2)$, $\isbrim$ also has $(\lambda, n + 1)$-extension. Write $\K^\ast \coloneqq  \Kpropa_{\is, \isbrim, \Ps (n)}$ (Definition \ref{kprop2a-def}).

    \underline{Claim}: There is $N \in \K_{[\lambda_1, \lambda_2)}^\ast$ which is a $\phi_n$-amalgamation base in $\K_{\|N\|}^\ast$.

    \underline{Proof of Claim}: Let $\K^{\ast \ast}$ be $|\K^{\ast}|$ ordered by $\m \leap{\K^{\ast \ast}} \m'$ if and only if $\m \lea^d \m'$ (Definition \ref{sys-ext-def}(\ref{sys-ext-disj})). We apply Theorem \ref{ap-diamond-thm}, where $\K$, $\K^\ast, \phi$ there stand for $\K^\ast$, $\K^{\ast \ast}, \phi_n$ here (recalling Definition \ref{t25}). Note first that $\K_{[\lambda_1, \lambda_2)}^\ast$ is not empty since $\isbrim$ has $(\lambda, n + 1)$-existence for each $\lambda \in [\lambda_1, \lambda_2)$. Thus if we can check that the hypotheses of Theorem \ref{ap-diamond-thm}, the $\iota = 2$ version, we will obtain the conclusion of the claim. First observe that $\K^\ast$ is $[\lambda_1, \lambda_2)$-continuous by Lemma \ref{isbrim-lem}(\ref{isbrim-cont}). Next, there is a $\phi_n$-universal model in $\K_{\lambda_2}^{\ast \ast}$ by Lemma \ref{kprop-skel}(\ref{kprop-skel-univ}). Further, $\K_{\lambda_2}^{\ast \ast}$ is $\phi_n$-categorical by Lemma \ref{kprop-skel}(\ref{kprop-skel-categ}). Moreover, for any $\lambda \in [\lambda_1, \lambda_2)$, $\K_{\lambda}^\ast$ has $\phi_{n}$-uniqueness (Lemma \ref{kprop-skel}(\ref{kprop-skel-uq})). Finally, for any $\lambda \in [\lambda_1, \lambda_2)$, if we can $\phi_n$-amalgamate in $\K^{\ast \ast}$, then we can $\phi_n$-amalgamate in $\K^\ast$ by Lemma \ref{kprop-skel}(\ref{kprop-skel-ap}). $\dagger_{\text{Claim}}$

   Now let $\lambda \coloneqq  \|N\|$, and apply Theorem \ref{kprop-ap}, where $\is, \is^\ast$ there stand for $\is_\lambda, (\isbrim)_\lambda$ here. We get that $\isbrim$ has $(\lambda, n + 1)$-uniqueness, as desired.

 (3) Assume that $\is$ is nice and $\isbrim$ has $(\lambda_2, n + 1)$-existence. It is enough to show that $\isbrim$ has $(\lambda_1, n + 1)$-uniqueness. Then given an arbitrary $\lambda \in [\lambda_1, \lambda_2)$, we can run the same argument with $[\lambda, \lambda_2]$ playing the role of $\Theta.$ By the previous part, there is $\lambda \in [\lambda_1, \lambda_2)$ such that $\isbrim$ has $(\lambda, n + 1)$-uniqueness. Let $\m_0$ be a limit independent $\Psm (n + 1)$ system in $\lambda_1$ and let $\m_1, \m_2$ be limit independent $\Ps (n + 1)$-systems in $\lambda_1$ such that $\m_1 \, {\rest} \, \Psm (n + 1) = \m_2 \, {\rest} \, \Psm (n + 1) = \m_0$. We build $\seq{\m_\ell^i : i \le \lambda}$ strictly increasing continuous in $\K_{\is, \Ps (n + 1)}$, such that for all $i < \lambda$:
       
       \begin{enumerate}
       \item[(a)] $\m_\ell^0 = \m_\ell$, $\ell = 1,2$.
       
       \item[(b)] All the models in $\m_\ell^i$ have cardinality $\lambda_1 + |i|$.
       
       \item[(c)] $\m_\ell^{i + 1}$ is limit over $\m_\ell^i$, for $\ell = 1,2$.
       
       \item[(d)] $\m_1^i \, {\rest} \, \Psm (n + 1) = \m_2^i \, {\rest} \, \Psm (n + 1)$.
       \end{enumerate}

       This is possible by definition of niceness (using that the extensions are limit, we can take unions at limits). Now let $\m^\ast \coloneqq  \m_1^\lambda \, {\rest} \, \Psm (n + 1) = \m_2^\lambda \, {\rest} \, \Psm (n + 1)$. We claim that $\m^\ast$ is a uniqueness base. From that it will directly follow that $\m_0$ is also a uniqueness base, as desired. We would like to use that $\isbrim$ has $(\lambda, n + 1)$-uniqueness, but we can't do so directly since $\m^\ast$ may not be limit (all the models in the system are limit, but the extensions between the models may not be). We want to apply Lemma \ref{brim-uq-lem}, where $\is, \is^\ast, J$ there stand for $\is_\lambda \, {\rest} \, \Iis_{n + 2}$, $(\isbrim)_\lambda \, {\rest} \, \Iis_{n + 2}, \Ps (n + 1)$ here. To apply it, we need to check that $\is$ has $(\lambda, n + 2)$-extension.

       By Lemma \ref{brim-ext-thm} and categoricity, it is enough to check that $\isbrim$ has $(\lambda, n + 2)$-extension. We know already that $\isbrim$ has $(\lambda', n + 1)$-existence for all $\lambda' \in [\lambda_1, \lambda_2]$ (when $\lambda' = \lambda_2$, this is a hypothesis, and when $\lambda' \in [\lambda_1, \lambda_2)$, this was derived in the first part of this lemma). In particular, $\isbrim$ has $(\lambda^+, n + 1)$-existence. By Lemma \ref{ext-lambda-step}, $\isbrim$ has $(\lambda, n + 2)$-existence. Since $\isbrim$ has $(\lambda, n + 1)$-uniqueness, we get by Lemma \ref{strong-uq-brimmed-equiv} that $\isbrim$ has $(\lambda, n + 2)$-extension, as desired.
\end{proof}

We obtain two sufficient conditions to have all the $n$-dimensional properties: weak GCH or niceness.

\begin{thm}[The brimmed excellence theorem]\LABEL{excellence-lem}
  Let $\seq{\lambda_n : n < \omega}$ be an increasing sequence of infinite cardinals such that $2^{\lambda_n} = 2^{<\lambda_{n + 1}} < 2^{\lambda_{n + 1}}$ for all $n < \omega$. Let $\lambda_\omega \coloneqq  \sup_{n < \omega} \lambda_n$ and let $\Theta \coloneqq  [\lambda_0, \lambda_\omega)$. Let $\is$ be a very good independence notion with domain $\Theta$. Assume that at least one of the following conditions holds:

  \begin{enumerate}
  \item $\is$ is nice.
  \item $\lambda_{n + 1} = \lambda_n^+$ for all $n < \omega$.
  \end{enumerate}

  Then for all $\lambda \in \Theta$, $\isbrim$ has $(\lambda, <\omega)$-existence and $(\lambda, <\omega)$-uniqueness.
\end{thm}
\begin{proof}
  By Lemma \ref{ext-lambda-full} (where $\lambda_1, \lambda_2$ there stand for $\lambda_0, \lambda_\omega$ here), for any $\lambda \in \Theta$, $\isbrim$ has $(\lambda, <\omega)$-existence. Next, we prove by induction on $n < \omega$ that for all $\lambda \in \Theta$, $\isbrim$ has $(\lambda, n)$-uniqueness. When $n = 0$, this is part of Lemma \ref{isbrim-lem}. Assume now that for all $\lambda \in \Theta$, $\is$ has $(\lambda, n)$-uniqueness. Fix $k < \omega$. By Lemma \ref{uq-lambda-step} (where $\lambda_1, \lambda_2$ there stand for $\lambda_k, \lambda_{k +1}$ here), $\isbrim$ has $(\lambda, n + 1)$-uniqueness for all $\lambda \in [\lambda_k, \lambda_{k + 1})$. Since $k$ was arbitrary, this shows that $\isbrim$ has $(\lambda, n + 1)$-uniqueness for all $\lambda \in \Theta$.
\end{proof}

%%%%%%%%%%%%%%%%%%%%%%%%%%%%%%%%%%%%%%%%%%%%%%
\newpage

\section{Building primes}\LABEL{primes-sec}

In this section, we show that prime models over independent systems can be built, provided that we have the $n$-dimensional extension and uniqueness properties. This is the \emph{prime extension theorem}, Theorem \ref{prime-ext-thm}. Recall that a model is prime over a base if it is ``minimal'' over that base in some sense (see Definition \ref{prime-def}). It is well known that prime models are important to categoricity proof. The intuition is that they allow us to describe a model as being built ``point by point'', analogous to how a vector space is spanned by a basis. 

To build primes, we need the following technical concept, which says that a system can in a sense only be extended in an independent way. This is similar to the definition of a \emph{reduced tower} \cite[3.1.11]{shvi635}. The differences are that we work with systems and are slightly more local (we work inside a given system).

\begin{defin}\LABEL{x5}\myindex{reduced inside}
  Let $\is = (\K, \Iis, \NF)$ be a multidimensional independence relation and let $\m$ be an independent $I$-system, with $I = J \cup \{v\}$, $J < v$. Let $\m_\ast$ be an independent $J$-system. We say that $\m$ is \emph{reduced inside $\m_\ast$} if:

  \begin{enumerate}
  \item $\m \, {\rest} \, J \leap{\is} \m_\ast$.
  \item Whenever $\m'$ is an independent $I$-system and $f \colon \m \rightarrow \m'$ is a systems embedding, if:
    \begin{enumerate}
    \item $f$ is the identity on $J$.
    \item $\m \, {\rest} \, J \leap{\is} \m' \, {\rest} \, J$.
    \item $\m' \, {\rest} \, J \leap{\is} \m_\ast$.
    \end{enumerate}

    Then $f (\m) \leap{\is} \m'$. 
  \end{enumerate}
\end{defin}

Reduced systems exist:

\begin{lem}[Existence of reduced systems]\LABEL{reduced-ext-lem}
  Let $\is$ be a very good multidimensional independence notion. Let $\Theta \coloneqq  \dom (\K_{\is})$. Let $I = J \cup \{v\}$ be in $\Iis_{\is}$, $J < v$. Assume that $I$ is an initial segment of $\Ps (n)$. Let $\lambda \in \Theta$ be such that $\lambda^+ \in \Theta$ and $\is$ has $(\lambda, n + 1)$-extension. 

  For any independent $J$-system $\m$ in $\is_{\lambda^+}$, there exists a $\leap{\is}$-increasing chain of $I$-systems $\seq{\m_i : i < \lambda^+}$ in $\is_{\lambda}$ such that each $\m_i$ is reduced in $\m$ and $\bigcup_{i < \lambda^+} \m_i \, {\rest} \, J = \m$.
\end{lem}
\begin{proof}
  Without loss of generality, $\lambda = \min (\Theta)$. Let $\is'$ be as given by Lemma \ref{upward-ext}. Let $\seq{a_i : i < \lambda^+}$ be an enumeration of $\m$. We build $\seq{\m_i^0  \colon i < \lambda^+}$ a $\leap{\is}$-increasing continuous chain of independent $J$-systems in $\is_\lambda$, $\seq{N_i \colon i < \lambda^+}$ in $\K_{\is_\lambda}$, and $\seq{f_{i,j} : i \le j < \lambda^+}$ such that for all $i < \lambda^+$:

  \begin{enumerate}
  \item $f_{i, j}: N_i \rightarrow N_{j}$ form a continuous directed system.
  \item $\m_i^0 \smallfrown N_i$ is an independent $I$-system.
  \item $f_{i , j}$ fixes $\m_i^0$.
  \item $\m_{i + 1}^0$ contains $a_i$.
  \item\label{reduced-constr} If $\m_i^0 \smallfrown N_i$ is \emph{not} reduced in $\m$, then $\m_{i + 1}^0$, $N_{i + 1}$, and $f_{i, i +1}$ witness it. That is, $\m_i^0 \smallfrown f_{i, i + 1}[N_i] \not \leap{\is} \m_{i + 1}^0 \smallfrown N_{i + 1}$.
  \end{enumerate}

  This is possible using the extension property and taking direct limits at limits (we use strong continuity to see that the direct limit is still an independent system). This is enough: let $\m_{\lambda^+}^0$, $f_{i,\lambda^+}$, $N_{\lambda^+}$ be the direct limit of the system. Note that $\m_{\lambda^+}^0 \smallfrown N_{\lambda^+}$ is an independent $I$-system in $\is'$, and since $\K_{\is', I}$ is an AEC, this system can be $\leap{\is}$-resolved into independent systems of cardinality $\lambda$. By Lemma \ref{club-struct}, this implies that there is a club $C \subseteq \lambda^+$ such that for any $i \in C$, $\m_i^0 \smallfrown f_{i, \lambda^+}[N_i] \leap{\is} \m_{\lambda^+}^0 \smallfrown N_{\lambda^+}$. Now if $i \in C$ and $\m_i^0 \smallfrown N_i$ is \emph{not} reduced in $\m$, then by  property (\ref{reduced-constr}), $\m_i^0 \smallfrown f_{i, i+1}[N_i] \not \leap{\is} \m_{i + 1}^0 \smallfrown N_{i + 1}$. Taking the image of this statement under $f_{i + 1, \lambda^+}$, we get that $\m_i^0 \smallfrown f_{i, \lambda^+} [N_i] \not \leap{\is} \m_{i + 1}^0 \smallfrown f_{i + 1, \lambda^+}[N_{i + 1}]$. By monotonicity 2, this means that $\m_i^0 \smallfrown f_{i, \lambda^+} [N_i] \not \leap{\is} \m_{i + 1}^0 \smallfrown N_{\lambda^+}$, a contradiction. Thus for any $i \in C$, $\m_i^0 \smallfrown N_i$ \emph{is} reduced in $\m$. Let $\seq{\alpha_i : i < \lambda^+}$ be a normal enumeration of $C$ and let $\m_i \coloneqq  \m_{\alpha_i}^0 \smallfrown N_{\alpha_i}$. 
\end{proof}

We use reduced systems to build prime ones:

\begin{defin}\LABEL{x8}
  Let $\is$ be a multidimensional independence relation.

  \begin{enumerate}
  \item\myindex{prime system} Let $I = J \cup \{v\}$ be in $\Iis_{\is}$ with $J < v$, and let $\m$ be an independent $I$-system. We say that $\m$ is \emph{prime} if whenever $\m'$ is an independent $J$-system with $\m \, {\rest} \, J = \m' \, {\rest} \, J$, there exists $f: \m \rightarrow \m'$ a systems embedding fixing $\m \, {\rest} \, J$.
  \item\myindex{prime extension} We say that $\is$ has \emph{prime extension} if for any $I = J \cup \{v\} \in \Iis_{\is}$ with $J < v$ and any independent $J$-system $\m$, there is a \emph{prime} $I$-system $\m'$ such that $\m \cong \m' \, {\rest} \, J$. We define variations such as $(\lambda, n)$-prime extension as in Definition \ref{okay-def}.
  \end{enumerate}
\end{defin}

\begin{lem}\LABEL{prime-ext-lem}
  Let $\is$ be a very good multidimensional independence notion, let $\Theta \coloneqq  \dom (\K_{\is})$. Let $I = J \cup \{v\}$ be in $\Iis_{\is}$, $J < v$. Assume that $I$ is an initial segment of $\Ps (n)$. Let $\lambda \in \Theta$ be such that $\lambda^+ \in \Theta$, $\is$ has $(\lambda, n + 1)$-extension and $(\lambda, n + 1)$-uniqueness. If $\K_{\is}$ is categorical in $\lambda^+$, then any independent $J$-system $\m$ in $\is_{\lambda^+}$ can be extended to a \emph{prime} independent $I$-system in $\is_{\lambda^+}$.
\end{lem}
\begin{proof}
  Similar to \cite[III.4.9]{shelahaecbook} (or \cite[3.6]{prime-categ-mlq}), using Lemma \ref{reduced-ext-lem} to get the resolution into reduced triples.
\end{proof}

\begin{thm}[The prime extension theorem]\LABEL{prime-ext-thm}
  Let $\is$ be a very good multidimensional independence notion, let $\Theta \coloneqq  \dom (\K_{\is})$. Let $\lambda \in \Theta$ be such that $\lambda^+ \in \Theta$ and $\K_{\is}$ is categorical in $\lambda^+$. If $\is$ has $(\lambda, n)$-extension and $(\lambda, n)$-uniqueness, then $\is$ has $(\lambda^+, <n)$-prime extension.
\end{thm}
\begin{proof}
  Immediate from Lemma \ref{prime-ext-lem}.
\end{proof}

%%%%%%%%%%%%%%%%%%%%%%%%%%%%%%%%%%%%%%%%%%%%%%
\newpage

\section{Excellent classes}\LABEL{excellent-sec}

In this section, we introduce the definition of an \emph{excellent} AEC (this generalizes the definition in \cite{sh87a, sh87b}). We show how to obtain excellent AECs from the setups considered earlier (Lemma \ref{excellence-up}) and explain why excellent classes are tame and have primes (Theorem \ref{excellent-struct}), hence (as we will see in the next section) admit categoricity transfers. We also show how to get excellence in compact AECs (Theorem \ref{compact-excellent-thm}) and in $(<\omega)$-extendible good frames, using the weak diamond (Theorem \ref{excellence-wgch}).

\begin{defin}\LABEL{excellent-def}\myindex{excellent}
  Let $\is$ be a multidimensional independence relation. We call $\is$ \emph{excellent} if:

  \begin{enumerate}
  \item $\K_{\is}$ is an AEC.
  \item $\is$ is very good.
  \item $\is$ has extension and uniqueness.
  \end{enumerate}

  We say that an AEC $\K$ is \emph{excellent} if there exists an excellent multidimensional independence relation $\is$ such that $\K = \K_{\is}$.
\end{defin}

Reasonable multidimensional independence relations will be tame:

\begin{lem}\LABEL{excellence-local-lem}
  Let $\is$ be a multidimensional independence relation. If $\K_{\is}$ is an AEC, $\is$ has extension and uniqueness, and $\is$ satisfies the definition of being very good, except possibly for $\s (\nf (\is))$ being a good frame, then $\K_{\is}$ is $\LS (\K_{\is})$-tame.
\end{lem}
\begin{proof}
  As in \cite[III.1.10]{shelahaecbook} (see also \cite{superior-aec}).
\end{proof}

The next two lemmas tell us how to derive excellence from the conclusion of the limit excellence theorem (Theorem \ref{excellence-lem}).

\begin{lem}\LABEL{excellence-up}
  Let $\is$ be a very good multidimensional independence relation. Let $\lambda \coloneqq  \min (\dom (\is))$. If $\is$ has $(\lambda, <\omega)$-extension and $(\lambda, <\omega)$-uniqueness, then there exists a unique excellent multidimensional independence relation $\is'$ such that $(\is')_\lambda = \is_\lambda$. In particular, the AEC generated by $(\K_{\is})_\lambda$ is excellent.
\end{lem}
\begin{proof}
  Let $\K$ be the AEC generated by $(\K_{\is})_\lambda$. Let $\is'$ be as described by Lemma \ref{upward-ext}. It satisfies all the conditions in the definition of being very good, except perhaps for existence, uniqueness, and $\s (\nf (\is'))$ being a good frame. By Lemma \ref{going-up-excellence}, $\is'$ has extension and uniqueness. By Lemma \ref{excellence-local-lem}, $\K$ is $\LS (\K)$-tame. Since $\is'$ has extension, $\K$ has in particular amalgamation. By Fact \ref{frame-ext}, there is a good frame $\s$ with underlying class $\K$. Let $\ts \coloneqq  \s (\nf (\is'))$. It is clear from the definition that $\s_\lambda = \ts_\lambda$. Moreover, $\ts$ satisfies most of the axioms of a good frame (Fact \ref{nfm-props}). The only axiom that can fail is local character. However since $\nf (\is')$ is very good, it has local character, so every type does not $\ts$-fork over a model of size $\lambda$. By Lemma \ref{canon-lem} and Remark \ref{canon-lem-rmk}, this implies that $\s = \ts$, so $\ts$ is a good frame. Thus $\is'$ is very good, and hence excellent.
\end{proof}

\begin{lem}\LABEL{excellence-general}
  Let $\is$ be a very good multidimensional independence notion and let $\Theta \coloneqq  \dom (\is)$. Let $\lambda \coloneqq  \min (\Theta)$ and let $\K$ be the AEC generated by $\left(\K_{\is}\right)_\lambda$. If:

  \begin{enumerate}
  \item[(A)] $\K$ is categorical in $\lambda$.
  
  \item[(B)] $\isbrim$ has $(\lambda, <\omega)$-existence and $(\lambda, <\omega)$-uniqueness.
  \end{enumerate}

  Then $\K$ is excellent.
\end{lem}
\begin{proof}
  By Lemma \ref{strong-uq-brimmed-equiv}, $\is^\ast$ also has strong $(\lambda, <\omega)$-uniqueness. By Theorems \ref{brim-ext-thm} and \ref{brim-uq-thm}, we obtain that $\is$ also has $(\lambda, <\omega)$-extension and $(\lambda, <\omega)$-uniqueness. Now apply Lemma \ref{excellence-up}.
\end{proof}

Assuming a weak version of the generalized continuum hypothesis, we obtain excellence in $(<\omega)$-extendible very good frames. We will use the following notation:

\begin{defin}\LABEL{wgch-def}\myindex{$\WGCH (S)$}\index{$\WGCH$|see {$\WGCH (S)$}}
  For $S$ a class of cardinals, we write $\WGCH (S)$ for the statement ``$2^\lambda < 2^{\lambda^+}$ for all $\lambda \in S$''. We write $\WGCH$ instead of $\WGCH (\operatorname{Card})$, where $\text{Card}$ is the class of all cardinals.
\end{defin}

\begin{thm}\LABEL{excellence-wgch}
  Let $\s$ be a $(<\omega)$-extendible categorical very good $\lambda$-frame. Let $\K$ be the AEC generated by $\K_{\s}$. If $\WGCH ([\lambda, \lambda^{+\omega}))$ holds, then $\K$ is excellent.
\end{thm}

\begin{proof}
  Let $\Theta \coloneqq  [\lambda, \lambda^{+\omega})$. By Fact \ref{very-good-facts}(\ref{very-good-4}), there is a very good categorical frame $\ts$ such that $\dom (\ts) = \Theta$ and $\ts_{\lambda^{+n}} = \s^{+n}$ for all $n < \omega$. By definition of being very good, there is a very good two-dimensional independence notion $\nf$ on $\K_{\ts}$.
    
    Fix $\is$ and $\is^\ast$ satisfying the conclusion of Theorem \ref{constr-thm}. By Lemma \ref{excellence-lem}, we have that $\is^\ast$ has $(\lambda, <\omega)$-existence and $(\lambda, <\omega)$-uniqueness. Now apply Lemma \ref{excellence-general}.
\end{proof}

In a compact AEC, excellence follows from categoricity (without any cardinal arithmetic hypothesis)

\begin{thm}\LABEL{compact-excellent-thm}
  Let $\K$ be a compact AEC. If $\K$ is categorical in some $\mu > \LS (\K)$, then $\Ksatp{\LS (\K)^{+6}}$ is an excellent AEC.
\end{thm}
\begin{proof}
  By Corollary \ref{compact-unbounded}, $\K$ is categorical in a proper class of cardinals, so we might as well assume that $\mu \ge \beth_{\omega} (\LS (\K))$. Fix $\is$ satisfying the conclusion of Theorem \ref{multidim-compact}. Let $\Theta \coloneqq  [\LS (\K)^{+6}, \beth_\omega (\LS (\K)))$. By Lemma \ref{excellence-lem}, we have that $\isbrim$ has $(\LS (\K)^{+6}, <\omega)$-existence and $(\LS (\K)^{+6}, <\omega)$-uniqueness. The result now follows from Lemma \ref{excellence-general}.
\end{proof}

An appropriate subclass of saturated models of an excellent class will always have primes in the following sense:

\begin{defin}[{\cite[III.3.2]{shelahaecbook}}]\LABEL{prime-def}\myindex{prime triple}\myindex{has primes}
  Let $\K$ be an abstract class.

  \begin{enumerate}
  \item A \emph{prime triple} is a triple $(a, M, N)$ such that $M \lea N$, $a \in |N|$, and whenever $\gtp (b / M; N') = \gtp (a / M; N)$, there exists a $\K$-embedding $f: N \xrightarrow[M]{} N'$ such that $f (b) = a$.
  \item $\K$ \emph{has primes} if for any $p \in \gS (M)$ there exists a prime triple $(a, M, N)$ such that $p = \gtp (a / M; N)$.
  \end{enumerate}
\end{defin}

\begin{thm}[Structure of excellent AECs]\LABEL{excellent-struct}
  If $\K$ is an excellent AEC, then $\K$ is not empty, has amalgamation, joint embedding, no maximal models and is $\LS(\K)$-tame. Moreover, $\Ksatp{\LS (\K)^+}$ has primes.
\end{thm}

\begin{proof}
  Let $\lambda \coloneqq  \LS (\K)$. Let $\is$ be an excellent multidimensional independence relation with $\K_{\is} = \K$. $\K$ is not empty and has no maximal models because $\is$ has $1$-extension, and $\K$ has amalgamation and joint embedding because $\is$ has $1$-uniqueness (see Lemma \ref{basic-small-dim}). By Lemma \ref{excellence-local-lem}, $\K$ is $\LS (\K)$-tame. Let $\is^\ast$ be the restriction of $\is$ to the fragmented AEC $\K^\ast$ such that $\K^{\ast}_{\le \lambda} = \K_\lambda$ and $\K^{\ast}_{>\lambda} = \Ksatp{\lambda^+}$. It is easy to check that $\is^\ast$ is still a very good multidimensional independence relation with extension and uniqueness. By Theorem \ref{prime-ext-thm}, $\is^\ast$ has $(\lambda^+, <\omega)$-prime extension. By standard arguments similar to those in \cite[\S5]{sh87b}, one gets that $\is^\ast_{\ge \lambda^+}$ has prime extension.

  Now by \cite[III.4.9]{shelahaecbook} (see also \cite[3.6]{prime-categ-mlq}), $\K^\ast_{\lambda^+}$ has primes. It remains to check that $\K^\ast_{>\lambda^+}$ has primes. Let $M \in \K^\ast_{>\lambda^+}$ and let $p \in \gS (M)$. Let $N \in \K^\ast$ be such that $M \lea N$ and $N$ realizes $p$, say with $a$. By local character for $\nf \coloneqq  \nf (\is^\ast)$, there exists $M_0, N_0 \in \K_\lambda^\ast$ such that $a \in |N_0|$ and $\nfs{M_0}{M}{N_0}{N}$. Since $\K_{\lambda^+}^\ast$ has primes, we may pick $N_0^\ast \lea N_0$ such that $(a, M_0, N_0)$ is a prime triple in $\K^\ast$. By monotonicity, we still have $\nfs{M_0}{M}{N_0^\ast}{N}$, so by prime extension (applied to the independent system $\m$ consisting of $M_0$, $M$, and $N_0^\ast$), there exists a prime model $N^\ast$ over $\m$, which can be taken to be contained inside $N$. Now check that $(a, M, N^\ast)$ is a prime triple.
\end{proof}

%%%%%%%%%%%%%%%%%%%%%%%%%%%%%%%%%%%%%%%%%%%%%%
\newpage

\section{Excellence and categoricity: the main theorems}\LABEL{main-sec}

In this section, we combine the results derived so far about excellence with known categoricity transfers to obtain the main theorems of this paper. In addition to Facts \ref{morley-omitting}, \ref{categ-unbounded-fact}, the facts we will use about categoricity are:

\begin{fact}\LABEL{categ-old-facts}
  Let $\K$ be an $\LS (\K)$-tame AEC with amalgamation, and arbitrarily large models. Assume that $\K$ is categorical in \emph{some} $\mu > \LS (\K)$. Then:

  \begin{enumerate}
  \item \cite[10.9]{Vas17}\footnote{The main ideas of the proof appear in \cite{ap-universal-apal, categ-primes-mlq}.} If $\K$ has primes then $\K$ is categorical in all $\mu' \ge \mu$.
  \item \cite[10.3,10.4]{Vas17}\footnote{The upward part of the transfer (i.e.\ getting categoricity in $\mu' \ge \mu$) is due to Grossberg and VanDieren \cite{tamenesstwo, tamenessthree}.} If $\K$ is categorical in $\LS (\K)$ and $\mu$ is a successor, then $\K$ is categorical in \emph{all} $\mu' \ge \LS (\K)$.
  \end{enumerate}
\end{fact}

We obtain the following categoricity transfer for excellent classes (recall Definition \ref{hanf-def}):

\begin{thm}\LABEL{categ-excellence-thm}
  Let $\K$ be an excellent AEC. If $\K$ is categorical in \emph{some} $\mu > \LS (\K)$, then there exists $\chi < \hanf{\LS (\K)}$ such that $\K$ is categorical in \emph{all} $\mu' \ge \min (\mu, \chi)$. If in addition $\K$ is categorical in $\LS (\K)$, then $\K$ is categorical in all $\mu' > \LS (\K)$.
\end{thm}
\begin{proof}
  By Fact \ref{categ-unbounded-fact}, we may assume that $\mu > \LS (\K)^+$. By Fact \ref{categ-old-facts} and Theorem \ref{excellent-struct}, $\Ksatp{\LS (\K)^+}$ is categorical on a tail of cardinals. By Fact \ref{morley-omitting}, $\K$ is categorical in all $\mu' \ge \min (\mu, \chi)$. If $\K$ is categorical in $\LS (\K)$, then it will be categorical in a successor, hence by Fact \ref{categ-old-facts} again, $\K$ will be categorical in all cardinals above $\LS (\K)$.
\end{proof}

When working inside a $(<\omega)$-extendible categorical good $\lambda$-frame, it is known how to transfer categoricity across finite successors of $\lambda$, so we can be more precise:

\begin{lem}\LABEL{categ-excellence-lem}
  Let $\s$ be a $(<\omega)$-extendible categorical good $\lambda$-frame. Let $\K$ be the AEC generated by $\K_{\s}$. If there exists $n < \omega$ such that the AEC generated by $\K_{\s^{+n}}$ is excellent, then:
  
    \begin{enumerate}
        \item $\K$ has arbitrarily large models.
    
        \item If in addition $\K$ is categorical in \emph{some} $\mu > \lambda$, then $\K$ is categorical in \emph{all} $\mu' > \lambda$, $\K_{\ge \mu}$ has amalgamation, no maximal models, and is $\mu$-tame.
  \end{enumerate}
\end{lem}

\begin{proof}
  Fix $n < \omega$ such that the AEC $\K^{*}$ generated by $\K_{\s^{+n}}$ is excellent. 

  Let 

  \begin{enumerate}
      \item[$(\ast)_{1}$] $\K_{\ell}$ be the AEC generated by $\s^{+\ell}$, so $\K_{\ell}$ is categorical in $\lambda^{+ \ell}.$ 
  \end{enumerate}

  Now, 

  \begin{enumerate}
      \item[$(\ast)_{2}$] 

      \begin{enumerate}
          \item[$\boxplus$]

          \begin{enumerate}
              \item[(a)] if $\ell < \omega,$ then $\K_{\ell}$ is categorical in $\lambda^{\ell +1}$ iff $\K_{\ell}$ is categorical in $\lambda^{\ell +  2},$ 

              \item[(b)] if $0 < \ell < n$ then $\K_{\ell -1}$ is categorical in $\lambda^{+ \ell +1}$ iff $\K_{\ell}$ is categorical in $\lambda^{+ \ell + 1}$ iff $M \in \K_{\lambda^{+ \ell + 1}} \Rightarrow M \in \K_{\ell - 1} \cong M \in \K_{\ell}.$
          \end{enumerate}
      \end{enumerate}
  \end{enumerate}

  [Why? By \cite[Theorem A.9]{Vas17}, here the doubts about its use in \ref{frame-categ-succ} do not arise.]

  Hence, 

  \begin{enumerate}
      \item[$(\ast)$]  either for all $\ell < k < \omega,$ $\K_{\ell}$ is categorical in $\lambda^{+1}$ \underline{or} for all $\ell < k < \omega,$ $K_{\ell}$ is not categorical in $\lambda^{+ k}.$
  \end{enumerate}

  So if the first possibility in $(\ast)$ holds, by\ref{categ-excellence-thm} we are done. So we can assume the second one in $(\ast)$ holds, hence necessarily $\mu \geq \lambda^{ + \omega}.$ Still $\K_{\ast} = \K_{n}$ is categorical in $\lambda^{+ n},$ hence by \ref{categ-excellence-thm} it is categorical in $\lambda^{n+ n +1}$ but this contradicts ``the second case in $(\ast)$ holds'', so we are done. 
\end{proof}

Thus assuming some instances of the weak diamond, we obtain the following eventual categoricity transfer for $(<\omega)$-extendible frames:

\begin{cor}\LABEL{categ-extendible-cor}
  Let $\s$ be a $(<\omega)$-extendible categorical good $\lambda$-frame. Let $\K$ be the AEC generated by $\K_{\s}$. Assume $\WGCH ([\lambda^{+n}, \lambda^{+\omega}))$ holds for some $n < \omega$. Then $\K$ has arbitrarily large models and if $\K$ is categorical in \emph{some} $\mu > \lambda$, then $\K$ is categorical in \emph{all} $\mu' > \lambda$ and moreover $\K_{\ge \mu}$ has amalgamation, no maximal models, and is $\mu$-tame.
\end{cor}
 
\begin{proof}
  Let $\ts \coloneqq  \s^{+(n + 3)}$. By Fact \ref{very-good-facts}(\ref{very-good-2}), $\ts$ is very good. By Theorem \ref{excellence-wgch}, the AEC $\K^\ast$ generated by $\K_{\ts}$ is excellent. Now apply Lemma \ref{categ-excellence-lem}.
\end{proof}

Specializing to compact AECs, we get:

\begin{thm}\LABEL{compact-main-thm}
  Let $\K$ be a compact AEC. Let $\mu > \LS (\K)$. If $\K$ is categorical in $\mu$, then there exists $\chi < \hanf{\LS (\K)}$ such that $\K$ is categorical in all $\mu' \ge \min (\mu, \chi)$.
\end{thm}

\begin{proof}
  By Corollary \ref{compact-unbounded}, $\K$ has amalgamation, no maximal models, and is categorical in a proper class of cardinals. We will also use without comments Fact \ref{tameness-ap}, which says in particular that $\Ksatp{\lambda}$ is an AEC for any $\lambda > \LS (\K)$. By Theorem \ref{compact-excellent-thm}, $\Ksatp{\LS (\K)^{+6}}$ is excellent. By Theorem \ref{categ-excellence-thm}, $\Ksatp{\LS (\K)^{+6}}$ is categorical in all high-enough cardinals. Since $\K$ itself is categorical in a proper class of cardinals, this implies that $\Ksatp{\LS (\K)^+}$ is also categorical in all high-enough cardinals, and so in particular in a high-enough successor cardinal. Now $\Ksatp{\LS (\K)^+}$ is categorical in $\LS (\K)^+ = \LS (\Ksatp{\LS (\K)^+})$, and hence by Fact \ref{categ-old-facts} it is categorical in \emph{all} $\mu' \ge \LS (\K)^+$. Now apply Morley's omitting type theorem for AECs (Fact \ref{morley-omitting}) to get the desired conclusion about categoricity in $\K$.
\end{proof}
\begin{cor}\LABEL{compact-main-cor}
  Let $\K$ be a compact AEC. If $\K$ is categorical in \emph{some} $\mu \ge \hanf{\LS (\K)}$, then $\K$ is categorical in \emph{all} $\mu' \ge \hanf{\LS (\K)}$.
\end{cor}
\begin{proof}
  This is a special case of Theorem \ref{compact-main-thm}.
\end{proof}
\begin{cor}\LABEL{compact-main-cor-2} \
  \begin{enumerate}
  \item Let $\K$ be an AEC and let $\kappa > \LS (\K)$ be a strongly compact cardinal. If $\K$ is categorical in \emph{some} $\mu \ge \hanf{\kappa}$, then $\K$ is categorical in \emph{all} $\mu' \ge \hanf{\kappa}$.
  
  \item Let $T$ be a theory in $\Ll_{\kappa, \omega}$, $\kappa$ a strongly compact cardinal. If $T$ is categorical in \emph{some} $\mu \ge \hanf{|T| + |\tau (T)| + \kappa}$, then $T$ is categorical in \emph{all} $\mu' \ge \hanf{|T| + |\tau (T)| + \kappa}$.
  
  \item Let $\psi$ be an $\Ll_{\kappa, \omega}$ sentence $\kappa$-strongly compact. If $\psi$ is categorical in \emph{some} $\mu \ge \hanf{\kappa}$, then $\psi$ is categorical in \emph{all} $\lambda' \ge \hanf{\kappa}$.
  \end{enumerate}
\end{cor}

\begin{proof}
  The first two parts are by Fact \ref{compact-fact} and Corollary \ref{compact-main-cor}. The third part follows from the second because an $\Ll_{\kappa, \omega}$ sentence has a vocabulary of size strictly less than $\kappa$.
\end{proof}

Regarding AECs with amalgamation, the following lemma says that (assuming WGCH) with a little bit of tameness, we can transfer categoricity. Note that the proof only uses amalgamation below the categoricity cardinal, but transfers it above.

\begin{lem}\LABEL{categ-tame-lem}
  Let $\K$ be an AEC with arbitrarily large models. Assume $\WGCH ([\LS (\K), \LS (\K)^{+\omega}))$. Let $\mu > \LS (\K)^+$ and assume that $\K_{<\max (\mu, \LS (\K)^{+\omega})}$ has amalgamation and no maximal models. If $\Ksat_{(\LS (\K), \LS (\K)^{+\omega})}$ is $\LS (\K)$-tame and $\K$ is categorical in $\mu$, then there is $\chi < \hanf{\LS (\K)}$ such that $\K$ is categorical in all $\mu' \ge \min (\mu, \chi)$.
\end{lem}

\begin{proof}
  We use without further comments (Fact \ref{tameness-ap}) that for any $\lambda \in (\LS (\K), \mu]$, $\Ksatp{\lambda}$ is an AEC with L{\"o}wenheim-Skolem-Tarski number $\lambda$. In particular, the model of cardinality $\mu$ is saturated. By Fact \ref{tame-frame-constr}, there is a (categorical) good frame $\s$ on $\Ksat_{[\LS (\K)^+, \LS (\K)^{+\omega})}$. By Fact \ref{extendible-wd}, $\s$ is $(<\omega)$-extendible. By Corollary \ref{categ-extendible-cor}, $\Ksatp{\LS (\K)^{+}}$ is categorical in all $\mu' \ge \LS (\K)^+$ and has amalgamation above $\mu$. In particular, $\K$ is categorical in all $\mu' \ge \mu$ and $\K_{\ge \mu}$ has amalgamation. Since we knew that $\K_{<\mu}$ had amalgamation, we have that $\K$ has amalgamation. By Morley's omitting type theorem for AECs (Fact \ref{morley-omitting}), there is $\chi < \hanf{\LS (\K)}$ such that $\K$ is categorical in all $\mu' \ge \min (\mu, \chi)$.
\end{proof}

Since tameness can be derived from high-enough categoricity, we obtain:

\begin{thm}\LABEL{ap-main}
  Let $\K$ be an AEC with arbitrarily large models. Let $\mu \ge \aleph_{\LS (\K)^+}$. Assume there exists unboundedly-many $\lambda < \aleph_{\LS (\K)^+}$ such that $\WGCH ([\lambda, \lambda^{+\omega}))$ holds. If $\K$ is categorical in $\mu$ and $\K_{<\mu}$ has amalgamation and no maximal models, then there exists $\chi < \aleph_{\LS (\K)^+}$ such that $\K$ is categorical in all $\mu' \ge \min (\hanf{\chi}, \mu)$.
\end{thm}
\begin{proof}
  By Fact \ref{tameness-ap}(\ref{tameness-ap-2}) (where $\lambda$ there stands for $\aleph_{\LS (\K)^+}$ here), there is $\chi_0 < \aleph_{\LS (\K)^+}$ such that $\Ksat_{(\chi_0, \aleph_{\LS (\K)^+})}$ is $\chi_0$-tame. By Lemma \ref{categ-tame-lem} applied to $\K_{\ge \chi_0}$, we get the result.
\end{proof}

\begin{cor}\LABEL{ap-main-cor}
  Assume $\WGCH$ (see Definition \ref{wgch-def}). Let $\K$ be an AEC with amalgamation. If $\K$ is categorical in \emph{some} $\mu \ge \hanf{\aleph_{\LS (\K)^+}}$, then $\K$ is categorical in \emph{all} $\mu' \ge \hanf{\aleph_{\LS (\K)^+}}$.
\end{cor}
\begin{proof}
  This is a special case of Theorem \ref{ap-main}. It is well-known that AECs with a model in $\hanf{\LS (\K)}$ have arbitrarily large models. Moreover, by considering the equivalence relation ``$M$ and $N$ embed into a common model'', we can partition the AEC into disjoint classes, each of which have joint embedding. One can then work inside the unique class with arbitrarily large models, which will have joint embedding and no maximal models. It is then easy to see that the original class must also be categorical in all high-enough cardinals. See for example \cite[10.13]{indep-aec-apal} for the details of this argument.
\end{proof}

\begin{cor}\LABEL{a78}
    Assume $2^{\theta} < 2^{\theta^{+}}$ for all cardinals $\theta.$ Let $\K$ be an AEC and let $\kappa > \LS(\K)$ be a measureable cardinal. If $\K$ is categorical in some $\mu \geq \beth_{(2^{\aleph_{\kappa^{+}}})^{}+},$ then $\K$ is categorical in \it{all} $\mu' \geq \beth_{(2^{\aleph_{\kappa^{+}}})^{}+}.$
\end{cor}

\begin{proof}
  By the main theorem of \cite{Sh:E102}, $\K_{[\kappa, \mu)}$ has amalgamation and $\K_{\geq \kappa}$ has no maximal models. Apply Theorem \ref{ap-main}.
\end{proof}

We can also obtain results from completely local hypotheses about the number of models in the $\lambda^{+n}$'s. This was already stated in \cite[III.12.43]{shelahaecbook}. Below, $\Ii (\K, \lambda)$ denotes the number of models (up to isomorphism) in $\K_\lambda$. On $\muunif$, see \cite[VII.0.4]{shelahaecbook2} for the definition and \cite[VII.9.4]{shelahaecbook2} on what is known (morally, $\muunif (\lambda^+, 2^\lambda) = 2^{\lambda^+}$ when $2^\lambda < 2^{\lambda^+}$).\myindex{$\Ii (\K, \lambda)$}\myindex{$\muunif$}

\begin{remark}\LABEL{14.38}
    Recall for comparison \cite[7.12]{Sh:734} 
    
    \begin{enumerate}
        \item[$(\boxplus)$] For some cardinal $\lambda_{\ast} < \chi$ and a cardinal $\lambda_{\ast \ast} < \beth_{1, 1}(\lambda_{\ast}^{+\omega})$ above $\lambda_{\ast}, \K$ is categorical in every cardinal $\lambda \geq \lambda_{\ast \ast}$ but in no $\lambda \in (\lambda_{\ast}, \lambda_{\ast \ast})$ provided that: 
        
        \begin{enumerate}
            \item[$\circledast_{\K}^{\mu, \chi}$] 
            
            \begin{enumerate}
                \item[(a)] $K$ is an a.e.c. categorical in $\mu,$
                
                \item[(b)] $\K$ has amalgamation and $\rm{JEP}$ in every $\lambda < \aleph_{\chi}, \lambda \geq \rm{LS}(\K),$
                
                \item[(c)] $\chi$ is a limit cardinal, $\cf \chi > \rm{LS}(\K),$ and for arbitrarily large $\lambda < \chi$ the sequence $\langle 2^{\lambda^{+n}}: n < \omega \rangle$ is increasing, 
                
                \item[(d)] $\mu > \beth_{1, 1}(\lambda)$ for every $\lambda < \chi$ hence $\mu \geq \aleph_{\chi},$
                
                \item[(e)] every $M \in K_{< \aleph_{\chi}}$ has a $\leq_{\K}$-extension in $K_{\mu}.$
            \end{enumerate}
        \end{enumerate}
    \end{enumerate}
\end{remark}

\begin{thm}\LABEL{few-model-thm-1}
  Let $\K$ be an AEC and let $\lambda \ge \LS (\K)$ be such that $\WGCH ([\lambda, \lambda^{+\omega)})$ holds. Assume that $\K$ is categorical in $\lambda$, $\lambda^+$, $\K_{\lambda^{++}} \neq \emptyset$, and $\Ii (\K, \lambda^{+n}) < \muunif (\lambda^{+n}, 2^{\lambda^{+(n - 1)}})$ for all $n \in [2, \omega)$. Then $\K$ is categorical in all $\mu > \lambda$.
\end{thm}
\begin{proof}
  As in the proof of \cite[II.9.2]{shelahaecbook}, we get that there is an $\omega$-successful $\goodp$ $\lambda$-frame $\s$ on $\K_\lambda$. This implies (Fact \ref{very-good-facts}(\ref{very-good-1})) that $\s$ is $(<\omega)$-extendible. Now apply Lemma \ref{categ-excellence-lem}.
\end{proof}

When $\lambda = \aleph_0$, the hypotheses can be weakened and we obtain the following generalization of the main result of \cite{sh87a, sh87b}:

\begin{thm}\LABEL{few-model-thm-2}
  Let $\K$ be a $\PC_{\aleph_0}$ AEC (so $\LS (\K) = \aleph_0$). Assume $\WGCH ([\aleph_0, \aleph_\omega))$. Assume that $\K$ is categorical in $\aleph_0$, $1 \le \Ii (\K, \aleph_1) < 2^{\aleph_1}$, and $\Ii (\K, \aleph_n) < \muunif (\aleph_n, 2^{\aleph_{n - 1}})$ for all $n \in [2, \omega)$. Then $\K$ has arbitrarily large models. Moreover, if $\K$ is categorical in \emph{some} uncountable cardinal, then $\K$ is categorical in \emph{all} uncountable cardinals.
\end{thm}

\begin{proof}
  As in the proof of Theorem \ref{few-model-thm-1}, using \cite[II.9.3]{shelahaecbook} instead of \cite[II.9.2]{shelahaecbook}.
\end{proof}

\printindex

\bibliographystyle{amsalpha}
\bibliography{multidim}

\end{document}